\documentclass[reqno,11pt]{amsart}
\usepackage{amsfonts,amsmath,amssymb}
\usepackage{graphics}
\usepackage{graphicx}
\usepackage{xcolor}

\usepackage[margin=1.3in]{geometry}
\numberwithin{equation}{section}

\def\og{{}^{0}\mathbf{g}}

\def\Div{\mbox{{\bf Div}}}
\def\a{{\alpha}}

\def\b{{\beta}}
\def\be{{\beta}}

\def\Ga{\Gamma}

\def\De{\Delta}

\def\eps{\epsilon}

\def\Si{\Sigma}

\def\varep{\varepsilon}
\def\vphi{\varphi}
\def\pr{{\partial}}
\def\al{\alpha}

\def\c{\cdot}

\def\AA{{\mathcal A}}

\def\BB{{\mathcal B}}

\def\NN{{\mathcal N}}
\def\II{{\mathcal I}}

\def\FF{{\mathcal F}}

\def\HH{{\mathcal H}}
\def\LL{{\mathcal L}}

\def\TT{{\mathcal T}}

\def\SS{{\mathcal S}}
\def\NN{{\mathcal N}}

\def\JJ{{\mathcal J}}
\def\KK{{\mathcal K}}

\def\RR{{\mathcal R}}
\def\QQ{{\mathcal Q}}

\def\Db{\bar{\D}}

\def\D{{\bf D}}

\def\M{{\bf M}}

\def\Q{{\bf Q}}

\def\P{{\bf P}}

\def\Z{{\bf Z}}

\def\T{{\bf T}}

\def\g{{\bf g}}

\def\RRR{{\mathbb R}}

\def\HHH{{\mathbb H}}
\def\f12{{\frac 1 2}}

\def\tr{\mbox{tr}}

\def\f{\widetilde{f}}

\def\rH{\,r_{\HH}}

\newcommand{\bea}{\begin{eqnarray}}
\newcommand{\eea}{\end{eqnarray}}
\def\beaa{\begin{eqnarray*}}
\def\eeaa{\end{eqnarray*}}

\newcommand{\nn}{\nonumber}

\def\ba{\begin{array}}
\def\ea{\end{array}}

\def\Phiring{\mathring{\Phi}}

\def\Db{\mathring{\D}}

\def\Llin{\,  ^{\mbox{(lin)}} \mkern-4mu \LL}

\def\NI{\noindent}

\def\Ricc{\mbox{Ric}}

\def\r_h{{r_{\mathcal{H}}}}
\newtheorem{theorem}{Theorem}[section]
\newtheorem{lemma}[theorem]{Lemma}
\newtheorem{proposition}[theorem]{Proposition}
\newtheorem{corollary}[theorem]{Corollary}
\newtheorem{definition}[theorem]{Definition}
\newtheorem{remark}[theorem]{Remark}

\begin{document}

\title[Global stability of the wave-map equation in Kerr spaces]{On the global stability of the wave-map equation in Kerr spaces with small angular momentum}

\author{Alexandru D. Ionescu}
\address{Princeton University}
\email{aionescu@math.princeton.edu}

\author{Sergiu Klainerman}
\address{Princeton University}
\email{aionescu@math.princeton.edu}

\thanks{The first author is supported in part by a Packard
Fellowship and NSF grant DMS-1065710.}
\thanks{The second author is supported by  the NSF grant  DMS-1065710}

\begin{abstract}
This paper is motivated by the problem of the nonlinear stability of the Kerr solution  for axially symmetric perturbations.   We consider   a   model problem concerning the  axially symmetric perturbations  of  a  wave map   $\Phi$  defined from  a fixed Kerr solution $\KK(M,a)$, $0\le a \le M $,
with values  in the two dimensional hyperbolic space $\HHH^2$.   A particular  such wave map is given
 by  the    complex     Ernst potential     associated  to     the  axial Killing vectorfield    $\Z$  of $\KK(M,a)$.
  We conjecture  that this     stationary solution      is stable, under small axially symmetric perturbations,  in the  domain of outer communication (DOC) of  $\KK(M,a)$, for all $0\le a<M$ and we provide preliminary support for  its validity,  by  deriving convincing stability estimates  
  for the   linearized  system.

\end{abstract}
\maketitle

\setcounter{tocdepth}{1}
\tableofcontents

\section{Introduction}\label{intro}

According to the    general expectations   the  Kerr  family    $\KK(a,M)$, in the sub-extremal    regime  $|a|<M$, is stable under general perturbations.
More precisely, it is expected  that:
\bigskip 

 {\bf  Kerr Stability      Conjecture.}     {   \it  An initial data set $(\Si_0,  g_0, k_0) $,  sufficiently close  to the initial data set  of a   fixed  sub-extremal   Kerr   spacetime  $\KK(M_i, a_i)$, admits a maximal, vacuum, future,  Cauchy development  $(\M, \g)$,  with a complete   future  null infinity  $\II^{+}$   and   whose    causal  past      $J^{-}(\II^+)$ is bounded in the future by a  smooth, complete,   event horizon $\HH^+$.      Moreover     $(\M, \g)$    remains close 
     to    $\KK(M_i, a_i)$  and      approaches  asymptotically   another       sub-extremal Kerr   spacetime  $\KK(M_f, a_f)$. 
}
\bigskip

Despite  its   extraordinary importance, in both mathematical and astrophysical \footnote{
     If  the Kerr family would  turn out  to be   unstable  under  perturbations,
      black holes   would be nothing more than mathematical artifacts.   See \cite{Chand}   for a comprehensive  account of  efforts made by physicists   
       to establish  the  linear stability of the Kerr  family.   }    terms,     and despite          half a century 
       of   sustained   efforts   to settle it, the conjecture remains wide open.      The main known  mathematical   
        arguments  in  favor  of the conjecture     are in fact      few   and, so far,  not at all decisive.
        \begin{enumerate}
        \item     We know  that the  Minkowski space, corresponding to $a=0$, $M=0$     is stable, see \cite{Ch-Kl}. 
        \item      We  know that,   perturbativly,   the Kerr family     exhausts  all   stationary, smooth, 
               solutions  of the Einstein vacuum   equations,    see    \cite{I-Kl0} and  \cite{A-I-Kl}.   In other words, any 
          stationary solution sufficiently  close to a        sub-extremal  Kerr   
            must    belong  to the  Kerr  family.   A full   review of    rigidity 
                  results  in the smooth setting   is discussed in   \cite{I-Kl}.
            
            \item We  possess  a   significantly large class of   examples  of dynamical  black holes, 
            settling down to   a sub-extremal Kerr, constructed  from infinity, see \cite{D-H-R}.

          \item  Most importantly, we have now a   satisfactory    understanding of  the so called\textit{ poor man linearization}.   More
           precisely,  we have a   general method   for  establishing    boundedness and  quantitative
               decay  of solutions            to  the   scalar wave equation $\square_{\g_{M, a}} \phi=0$,  
                   for    all   sub-extremal    Kerr  metrics 
             $\g_{M, a}$.     Such results  were  first established  in  Schwarzschild, see  
                 \cite{B-S1}, \cite{B-S2}, \cite{B-S3}, \cite{B-St}, \cite{DaRo1}, \cite{MaMeTaTo}  and later  extended 
                 for  $|a| \ll  M$ in   \cite{DaRo2},  \cite{TaTo}, \cite{A-Blue1}. The  full sub-extremal regime   was recently
                 settled   in \cite{mDiRySR2014}.

           \item    We have   results establishing     the non-existence of  exponentially growing modes 
            for the more  realistic    linearized  Teukolsky equations, see \cite{P-T}, \cite{Wh}\footnote{ Results
            on boundedness and decay  for   these equations   near Schwarzschild were 
            recently announced by Dafermos, Holzegel and Rodnianski, see  \cite{Daf}.}.      
                       \end{enumerate}

            The    goal of this paper is to provide additional evidence for the conjecture              
            in the special  case          of axi-symetric perturbations.

\subsection{A non-linear model problem}

 As well known (see  \cite{gW1990}    )  the Ernst potential  $\Phi=(\Phi^1, \Phi^2)$  of a Killing vectorfield  $\Z$  on a  $3+1$ dimensional   Einstein-vacuum manifold $(\M, \g)$
  can be interpreted as a wave map   $\Phi:\M\longrightarrow \HHH^2$   where $\HHH^2$  denotes   the  upper-half   Poincare space   with constant negative curvature  $K=-1$.   More precisely,
   \bea
     \square_\g\Phi ^a+ \g^{\mu\nu}\Ga^a_{bc}(\Phi )\pr_\mu\Phi ^b \pr_\nu \Phi ^c=0,
 \label{eq:Wave-Map}
     \eea
where $\Ga$ denotes the Christoffel symbols of the metric $h$ of $\HHH$.  The  full, axially symmetric,  space-time  metric  $\g$   decomposes   into  its   dynamic   component   $\Phi$ and         a \textit{reduced} $1+2$ metric $\hat{\g} $  defined on the orbit space $\hat{\M}=\M/\Z$ verifying,
\bea
\Ricc(\hat \g)_{\alpha\beta}=<\pr_\alpha\Phi, \pr_\beta \Phi>_h\label{reduced-Einst}
\eea
Thus, in  axial  symmetry,   the   Einstein vacuum equations   are equivalent\footnote{See \cite{gW1990}  for a very clear exposition of the reduction. Note that  \eqref{eq:Wave-Map}   can also be interpreted as a  wave map from  $ \widehat{\M} $   to  $ \HHH$.    } to the  coupled system
\eqref{eq:Wave-Map}--\eqref{reduced-Einst},  on  the  reduced  space-time  $\widehat{\M}$. A particular, stationary, solution  of the system is  provided  by  the  pair $(  \widehat{\g}_{M,a} ,\Phi_{M,a})=(A, B)$, denoting the   decomposition of  
  the  Kerr metric $\g_{M,a}$ of a fixed  Kerr spacetime $\M=\KK(M,a)$.
The full  problem of the  stability  of the Kerr solution, for axially symmetric perturbations,  can be  reformulated  as   a  problem  of  stability   of this special   solution          for  the  system  \eqref{eq:Wave-Map}--\eqref{reduced-Einst}.
 As this is still  an extremely difficult problem we  make one    further important  simplification
   by  partially linearizing  the system, that is we  fix the  reduced  metric  $  \widehat{g}=\widehat{\g}_{M,a} $ 
    but allow   fully nonlinear perturbations of  $\Phi_{M,a}$. It is easy to see that this amounts  
     to    the  problem of  stability     of axially symmetric  perturbations of    the  stationary solution  $\Phi_{M,a}$  of the   wave map system \eqref{eq:Wave-Map},    where   $\g$ is  fixed to  be   the Kerr metric $\g_{M,a}$.  
     \bigskip
     
  \noindent   {\bf Partial Stability   Conjecture.} \textit{The stationary solution   $\Phi_{M,a}:\KK(M,a)\longrightarrow \HHH $ of the   wave  map system \eqref{eq:Wave-Map} with $\g=\g_{M,a}$    the   metric   of $\KK(M,a)$, $|a|<M$,  is     future      assymptotically stable  in   the  domain of outer communication   of $\KK(M,a)$,
       for all   smooth,    axially symmetric,  admissible, perturbations\footnote{  That is:  for all   axially symmetric   initial data,  defined  on a spacelike hypersurface $\Si_0$, which are sufficiently close   to the  corresponding data of   $\Phi_{M,a}$ and vanishing in a suitable way on the axis of symmetry.}. }

\bigskip

\begin{remark}
We note that the conjecture is   consistent  with the full nonlinear stability conjecture,  for  axially symmetric  perturbations.
 More precisely  the validity of the     Kerr stability conjecture,     for axially symmetric perturbations,   implies (in principle)   the validity           of    our  partial stability  conjecture,  at least     for initial data   in the orthogonal  complement of   a finite dimensional space (corresponding to possible modulation).     In this paper we produce   convincing evidence  that the  conjecture is    in fact   true    for all  initial  data.  
\end{remark}

We take the first  step  in proving the conjecture   by  deriving   stability estimates  
  for the   linearized  system.       More precisely  we   introduce    the linearized  variables           
    \beaa
        \Phi=\Phi_{M,a}+  A\Psi,            \qquad \Psi=      (\phi,\psi).
        \eeaa
 and show that the linearized   equations  in  $\Psi$  possess  a  a coercive,     conserved,  energy 
   quantity  (for all $|a|\le M$)      and verify,  at least  for $a/M$ small,    a  Morawetz   type   estimate    comparable to those 
   derived    in recent years, see    \cite{B-S1}, \cite{B-S2}, \cite{B-S3}, \cite{B-St}, \cite{DaRo1}, \cite{MaMeTaTo},       for     the     scalar wave equation $\square \phi=0$.

 \begin{remark}
  In the simplest case $a=0$ the system for $\Psi=(\phi,\psi)$ is the decoupled system
  \begin{equation}
  \label{eq:modelS}
  \square\phi=0,\qquad \square\psi-\Big(\frac{4}{r^2(\sin\theta)^2}-\frac{8M}{r^3}\Big)\psi=0.
  \end{equation}
Note the non-trivial nature of the potential      for the $\psi$ equation,  singular  on the axis.
The precise form  of the potential is important in order     to derive the needed stability estimates.

 \end{remark}
  \subsection{Kerr metric} 
 The domain of outer communications of the Kerr spacetime $\mathcal{K}(M,a)$, in standard Boyer--Lindquist coordinates, is given by
\begin{equation}\label{zq1}
\g_{a,M}=-\frac{q^2\Delta}{\Sigma^2}(dt)^2+\frac{\Sigma^2(\sin\theta)^2}{q^2}\Big(d\phi-\frac{2aMr}{\Sigma^2}dt\Big)^2 +\frac{q^2}{\Delta}(dr)^2+q^2(d\theta)^2,
\end{equation}
where
\begin{equation}\label{zq2}
\begin{cases}
&\Delta=r^2+a^2-2Mr;\\
&q^2=r^2+a^2(\cos\theta)^2;\\
&\Sigma^2=(r^2+a^2)q^2+2Mra^2(\sin\theta)^2=(r^2+a^2)^2-a^2(\sin\theta)^2\Delta.
\end{cases}
\end{equation}
Observe that
\begin{equation}\label{zq2.1}
(2mr-q^2)\Sigma^2=-q^4\Delta+4a^2m^2r^2(\sin\theta)^2.
\end{equation}
Note also   the useful  identities,
\begin{equation}
\label{eq:useful-id}
\frac{\Sigma^2}{q^2}=q^2+(p+1)a^2(\sin\theta)^2,\qquad \Delta=q^2(1-p)+a^2(\sin\theta)^2,\qquad  p:=\frac{2Mr}{q^2}.
\end{equation}
Thus   the metric can also be written in the form,
 \bea
 \label{zq1'}
\g_{ a, M}\, &=& \,    -\frac{\left(\Delta-a^2\sin^2\theta\right)}{q^2}dt^2-
\frac{ 4 aMr }{q^2}   \sin^2\theta     dt  d\phi+\frac{q^2}{\Delta}dr^2+ q^2 d\theta^2+
\frac{ \Si^2}{q^2}\sin^2\theta
d\phi^2
\eea
and,
\beaa
\g_{tt} \g_{\phi \phi}- \g_{t\phi}^2 =-\De (\sin\theta)^2.
\eeaa
The volume element $d\mu$  of $\g$  is given by  
\beaa
d\mu&=&q^2 |\sin\theta| dt dr d\theta d\phi.
\eeaa
We also   note   that  $\T=\pr_t $, $\Z=\pr_\phi$ are both Killing and $\T$  is only time-like in the     complement of the ergoregion,
 i.e.  $q^2> 2 Mr $.

 The    domain of outer communication    of $\KK(M,a)$ is given by, 
\begin{equation*}
\RR=\{(\theta,r,t,\vphi)\in(-\pi,\pi)\times(\rH,\infty)\times\mathbb{R}\times\mathbb{S}^1\},
\end{equation*}
where     $\rH:=M+\sqrt{M^2-a^2}$,  the larger root of $\De$, corresponds to the      event horizon. 
The  metric  posesses the Killing v-fields $\T=\pr_t$ and $\Z=\pr_\phi$.

 The  Ernst potential $\Phiring=(A, B)$  associated to the Killing vector-field $\Z=\pr_\vphi$, is given explicitly  by the formula,
\begin{equation}\label{zq4}
A+iB:=\frac{\Sigma^2(\sin\theta)^2}{q^2}-i\Big[2aM(3\cos\theta-(\cos\theta)^3)+\frac{2a^3M(\sin\theta)^4\cos\theta}{q^2}\Big],\qquad  A=\g(\Z,\Z).
\end{equation}

  One can easily check\footnote{Or derive from first principles, see \cite{gW1990}.}  that $(A, B)$  verify the system,
\begin{equation}\label{zq5}
\begin{split}
&A\square A=\D^\mu A\D_\mu A-\D^\mu B\D_\mu B,\\
&A\square B=2\D^\mu A\D_\mu B.
\end{split}
\end{equation}
where   $\square=\square_{g_{M,a}}$  denotes the 
 usual  wave operator    with respect   to  the metric.
 We can interpret    $\Phiring:=(A, B)$  as   a stationary,   axisymmetric,   wave map   from    $\KK(M,a)$
   to the hyperbolic space  $\HHH^2 =(\RRR_+^2, h)$   with the metric $h$  given by,
\beaa
ds^2=\frac{1}{A^2} \big(dA^2+ dB^2\big)
\eeaa

  \subsection{ Reinterpreting    the   conjecture}
     As mentioned above  the goal of this paper is to investigate  
       the future   global asymptotic   stability,   in the exterior   region of $\KK(M,a)$,       of  the  special   stationary  map  $ \mathring{\Phi}=(A,B)$,  under general axially symmetric  perturbations.  In other words  we  consider   solutions
     $\Phi=(X,Y)$ of the wave map system,
   \begin{equation}
\begin{split}
\label{Wave-Map1}
&X\square X=\D^\mu X\D_\mu X-\D^\mu Y\D_\mu Y,\\
&X\square Y=2\D^\mu X\D_\mu Y.
\end{split}
\end{equation}  
            which are $\Z$-invariant, i.e. $\Z(\Phi^1)=\Z(\Phi^2)=0$,     and  whose initial conditions  on a  given space-like hypersurface in 
     $\RR$ are a small perturbation of  the initial data of $\mathring{\Phi}$.      We have to be careful however   that  the perturbed  map $\Phi=(X, Y)$    has the same axis of rotation as   $\Phiring=(A,B)$,
     i.e.   $\Phi=\Phiring$   on the axis of  symmetry of $\KK(M, a )$,  i.e.   $\sin^2\theta=0$. 
To make sure that this latter   condition is satisfied  we  search for 
  solutions  $\Phi=(X, Y) $     of the form,
\bea
\label{eq:nonl-perturb}
\Phi=\Phiring+  A\Psi,            \qquad \Psi=      (\phi,\psi).
\eea
 with $\psi $ vanishing on the axis of symmetry  $\AA$. 
With these notation  we can interpret   the system \eqref{Wave-Map1}  as a nonlinear system of 
 of equations   for $\Psi$, depending also on  the fixed $\Phiring$,            of the form,
 \bea
 \FF(\Phiring; \Psi)=0.\label{system-FF}
 \eea
  Our   Conjecture   can thus be interpreted as   a statement on the stability of    the trivial solution 
  $\Psi\equiv 0$    fot the  nonlinear system      \eqref{system-FF} .
 \bigskip
 
\NI    {\bf Conjecture.} \textit{ The  trivial solution $\Psi=0$ of the           nonlinear system   \eqref{system-FF}  is     future asymptotically stable    in the exterior region $r\ge r_\HH$ for arbitrary, smooth, axially symmetric,  admissible  (i.e.  such that  $\psi=0$  on  the axis $\AA$)  initial conditions
     on a $\Z$-invariant     spacelike hypersurface.}     
     \bigskip
     \subsection{Main Difficulties}     A simple comparison with the far  simpler case 
     of nonlinear  systems of wave equations    in Minkowski space    shows that we cannot expect
          the conjecture  to  be valid  without  addressing  the following  obstacles.
       \begin{enumerate}
       \item \textit{Strong linear stability.}  
       To start with,      one needs   to show that        the solutions  to the     wave map system   system     cannot  grow out of control.
         It does not suffice   to     show  that    the   solutions   to   the linearized equations  are simply bounded;         one  needs to prove    quantitative decay estimates 
        comparable to the  known decay estimates for   the  standard wave equation in the  Minkowski space $\RRR^{1+3}$. Moreover these  estimates have to be robust, i.e. the methods used in their derivation     can be extended, in principle,   to the nonlinear  equations.

       \item \textit{Nonlinear stability.}   Though  strong linear stability  is  an essential ingredient   in the proof  of nonlinear stability, it is by no means enough. The nonlinear terms of the equation   also have to satisfy   special structural conditions, such as  the null condition.

   \item \textit{Degeneracy on the axis.}  An additional difficulty  is the degeneracy of  our system
    on the axis of symmetry, i.e where $A$  vanishes, see \eqref{eq:modelS}.     Our functional analysis framework, see Definition \ref{MainDef}, is adapted to handle such a situation.
      \end{enumerate}  
  The first  difficulty is  the most serious     one.    The case  when the linearized equation  is simply   $\square_\g \Psi=0$     has now been  well understood in full  generality,  for all $|a|<M$   and   under no  symmetry assumptions,     see \cite{mDiRySR2014}   and the references therein. 
 Our   linearized   equations  differ significantly, however,  from this    case.     
       Indeed    taking the Fr\'echet derivative of $\FF$ with respect to $\Psi$
   we obtaine a linear operator  with coefficients which depend on $\Phiring=(A,B)$ in a non-trivial fashion.
     The linearized  equations  are in fact  of  the form:
 \begin{equation}
 \label{eq:Linear}
 \begin{split}
0&= \square\phi+2\frac{\D^\mu B}{A}\D_\mu\psi-2\frac{\D^\mu B\D_\mu B}{A^2}\phi+2\frac{\D^\mu B\D_\mu A}{A^2}\psi
\\
0&=\square\psi-2\frac{\D^\mu B}{A}\D_\mu\phi-\frac{\D^\mu A\D_\mu A+\D^\mu B\D_\mu B}{A^2}\psi.
\end{split}
\end{equation}
  and cannot be  decoupled.
  It is not apriori clear   that such an equation  possesses a   well defined notion of energy, i.e. a      conserved and   coercive(   integral     quantity     similar to the standard  energy quantity for   $\square\Psi=0$.   Though   the existence of such a quantity is by no means enough   to prove strong linear stability it is an absolutely necessary first step.
  Our first result  is   the following:
  
\begin{theorem}
\label{thm:Energy}
The  linearized  equations  \eqref{eq:Linear}  (for axi-symmmetric solutions  $\Psi$)  admit  an  
  energy-momentum tensor type quantity  $\QQ_{\mu\nu}=\QQ[\Psi]_{\mu\nu}$    and a source   $\JJ_\nu$,  both quadratic in
  $(\Psi, \pr \Psi)$,  depending also on $(\Phiring, \pr\Phiring)$,  verifying the following: 
  \begin{enumerate}
  \item[(a)]   $\QQ(X, Y)>0$, \quad  for any future-oriented,  timelike,  vector-fields $X,Y$; 
  \item[(b)]  $\D^\nu \QQ_{\mu\nu} =\JJ_{\nu}$.
  \item[(c)] $\g(\T, \JJ)=0$.
  \item[(d)]  $\QQ(Z,  X) =0$, \quad  for any vector-field $X$   orthogonal to $\Z$.

  \end{enumerate}

 \end{theorem}
  The   underlying reason for   the existence of  a    quantity     verifying   (b) and (c)    is  a somewhat less familiar 
    manifestation  of   Noether's principle, which we     discuss  below. The positivity (a), on the other  hand,    is  a   consequence of the     negative  curvature  properties   of $\HHH$.     The property (d)
    can be easily derived  from the   form of $\QQ$, displayed  below, and the  $\Z$-invariance of $\Psi$.

    As a consequence of     the Theorem we deduce  that   the current $\P_\mu:= \Q_{\mu\nu} \T^\nu$ is conserved, i.e.
    \beaa
    \D^\mu  \P_\mu=0.
    \eeaa 
     which leads, by integration  on    causal domains,    to  conserved   energy type quantities and fluxes.  In view of 
     (d)   the     energy is   a  coercive quantity in $\RR$  with a    mild  degeneracy  on the horizon $r=\rH$.
           Theorem \ref{thm:Energy}  is   thus     a strong  first    indication of the  validity 
       of our conjecture  for all  values of  the Kerr parameters, $|a|<M$.    
         Yet, as alluded above,     the bounds provided  by the energy are not by themselves  enough 
         to   even prove  the boundedness  of solutions  to  the system \eqref{eq:Linear},
         subject to   nice initial conditions.
         
             To actually  go beyond  the bounds provided by the energy  and  
       prove strong linear stability     we encounter        the same    difficulties   as   for  the simpler case 
     of  axially symmetric\footnote{In the case of  general solutions  there is another major obstacle,
      namely the  lack of coerciveness of the  energy  in    the  ergoregion.     The  strong  linear stability of  
       $\square\phi=0$  in Kerr  has  recently  been fully resolved     for all values 
       $|a|<m$ in \cite{mDiRySR2014}.       } solutions of the standard wave 
      equation $\square\phi=0$ in  the DOC of  $\KK(a,m)$,
          i.e.  
       degeneracy   of the   energy at the horizon,   presence of trapped null geodesics
        and    slow  decay at null infinity.     As it is now well understood,   the major  ingredient 
          for proving      strong linear stability     for  linear systems  on black holes is   the derivation
          of an integrated  decay estimate of Morawetz type.    Such estimates,   which    degenerate 
          in the trapping region,    i.e.   region of    $\KK(M,a)$ which contain 
          trapped null geodesics,   are   quite   subtle, and difficult   to derive.

          Fortunately,  in the case  of axial symmetry,
            all trapped null geodesics are   restricted to the   hypersurface 
              at $r=r_*$, the largest  root of  the polynomial    equation in $r$,   $r^3-3 r^2 M+a^2(r+M)=0$.
              This   allows one, in principle, to use  a vector-field  method approach   similar to     that  
               used in the derivation of                the  Morawetz type  integrated  decay estimate 
                 for solutions of  the scalar  wave  equation  in  Schwarzschild.     The main new difficulties are the  presence  of  the 
                  source term  $\JJ$  in   the divergence  equation $\Div \QQ=\JJ$,  and the degeneracy on the axis.
                    We overcome these difficulties in this paper,   for small     values of $a/M$.     Inspired  by  the  $r$-weighted  estimates of      Dafermos--Rodnianski\footnote{Their estimates     provide  similar   decay information
                     for  the \textit{outgoing energy}    associated to null hypersurfaces. },   see  \cite{Da-Ro3},   we also  prove  a stronger  version
                    of  the Morawetz  estimate  which provides  decay information
                     for  an appropriate  notion of  {\textit{outgoing energy}}    associated to space-like hypersurfaces.

                           A precise   version of our second theorem 
                          requires  a  space-like  $\Z$-invariant    foliation    $\Si_t$   of the entire domain of outer communication,     transversal to  the horizon and     whose leaves are  transported  by   $\T$.      In what follows we give a   first, informal,     version of the theorem, for the linearized  
                           equations  \eqref{eq:Linear}   in which we do not specify the foliation.  A more precise version will be given later     in this  section.

                          To state the theorem we  choose a  smooth, increasing    function  $\chi_{\geq 4M}$  supported  for $r\ge 4M$,    equal to $1$  for 
                          $r\ge 6M$,  and   define the {\it{outgoing energy density}} $(e(\phi),e(\psi))$,

\begin{equation*}
\begin{split}
&e(\phi)^2:=\frac{(\partial_1\phi)^2}{r^2}+(L\phi)^2+\frac{M^2\big[(\partial_2\phi)^2+(\partial_3\phi)^2\big]}{r^2}+\frac{\phi^2}{r^2},\\
&e(\psi)^2:=\frac{(\partial_1\psi)^2+\psi^2(\sin\theta)^{-2}}{r^2}+(L\psi)^2+\frac{M^2\big[(\partial_2\psi)^2+(\partial_3\psi)^2\big]}{r^2}+\frac{\psi^2}{r^2}.
\end{split}
\end{equation*}
 where $L$  is the  \textit{future outgoing}     vectorfield,
\begin{equation*}
L:=\chi_{\geq 4M}(r) \Big(\partial_r+\frac{r}{r-2M}\partial_t\Big).
\end{equation*}

            \begin{theorem}\label{MainTheorem0}              Assume that $ (\phi, \psi)$ is an admissible $\Z$-invariant  solution of  the linear system
             \eqref{eq:Linear}. 
                          Then, for any $\alpha\in(0,2)$ and any $t_1\leq t_2  $,
                          \beaa
                          \BB_\a(t_1, t_2)+\int_{\Sigma_{t_2}        }         \frac{r^\alpha}{M^{\alpha}}\big[e(\phi)^2+e(\psi)^2\big]\,d\mu_t\le   C_\alpha \int_{\Sigma_{t_2}        }         \frac{r^\alpha}{M^{\alpha}}&\big[e(\phi)^2+e(\psi)^2\big]\,d\mu_t
                          \eeaa
                          with $d\mu_t$ the induced measure on $\Si_t$ and $\BB_\a$ the  bulk  integral,
                         \begin{equation*}
                         \begin{split}
\mathcal{B}_\alpha(t_1,t_2):=\int_{\mathcal{D}_{[t_1,t_2]}}\frac{r^\alpha}{M^{\alpha}}\Big\{&\frac{(r-r^\ast)^2}{r^3}\frac{|\partial_\theta\phi|^2+|\partial_\theta \psi|^2+\psi^2(\sin\theta)^{-2}}{r^2}+\frac{1}{r}\big[(L\phi)^2+(L\psi)^2\big]\\
&+\frac{1}{r^3}\big(\phi^2+\psi^2\big)+\frac{M^{2}}{r^{3}}\big[(\partial_r\phi)^2+(\partial_r\psi)^2\big]\\
&+\frac{M^2(r-r^\ast)^2}{r^5}\big[(\partial_t \phi)^2+(\partial_t\psi)^2\big]\Big\}\,d\mu.
\end{split}
\end{equation*}
                          
                  \end{theorem}         
                                                    Note that, as expected the  integrand    of the bulk integral   $\BB_\a$  degenerates at  $r=r_*$.  
                                                   Though  the  presence of the  $r^\a$-weights   in our Morawetz type   estimate appear  to be new  even in the particular case 
                                                   of the  standard     scalar wave equation,  they were  clearly 
                                                     inspired by   the work of Dafermos-Rodnianski     \cite{Da-Ro3}. The  main new idea  in  
                                                       \cite{Da-Ro3} was   to observe that  one can replace the   $(t,r)$  weights
                                                       of the classical conformal multiplier method, along outgoing  null hypersurfaces,   by weights   which depend only 
                                                       on $r$, provided  that one has already derived  a local decay estimate.
                                                       The new twist  in our work is  to  show that  similar  estimates   can be  derived  
                                                       on  spacelike hypersurfaces.       Unlike   in the case of \cite{Da-Ro3}, where the proof  of  $r$-weighted
                                                        estimates  are   can be neatly   separated    from  the  main local decay estimate,     we  are obliged  in our work   to prove 
                                                   them  simultaneously. Proving a simultaneous estimate, on both the space-time integral, requires much more careful choices of the multipliers at infinity.

                                 \subsection{ Proof of Theorem \ref{thm:Energy}}   In this section we give  a first, informal,   derivation of Theorem \ref{thm:Energy}, 
                        based  on    first principles, which can be easily   generalized to other situations.
                        In the next section we   shall     re-derive   the  result by a straightforward verification. 
                        
                      Observe    first  that   the linear  system  \eqref{eq:Linear}
                       is  derivable from    a          Lagrangian\footnote{ One can  identify  $\LL$  as  the   quadratic form in $\Psi$   generated  by   the Taylor expansion at  $\Phiring$ of the  Lagrangian   of the  original, nonlinear, system
                        \eqref{Wave-Map1}.}
                  $\LL[\Phiring, \Psi] $,  $\Phiring=(\Phiring^1, \Phiring^2)=(A,B)$, $\Psi=(\Psi^1, \Psi^2)=(\phi,\psi)$, defined as follows:                  
 \bea
 \LL[\Phiring, \Psi ] &=&\g^{\mu\nu}\big[ \Db_\mu \phi \Db_\nu \phi+ \Db_\mu \psi \Db_\nu \psi+A^{-2}
(\phi \pr_\mu  B -\psi\pr_\mu A)(\phi \pr_\nu  B -\psi\pr_\nu A)\big]\nn\\
 \eea
 with,
 \beaa
 \Db_\mu\phi&=&\pr_\mu  \phi+ A^{-1} \pr_\mu B \, \psi\qquad 
\Db_\mu \psi=\pr_\mu  \psi-  A^{-1} \pr_\mu B \, \phi
 \eeaa
 We then  define, as ususal,   the energy momentum tensor of the linearized    field equation
  to be the quantity,
  \bea
  \QQ[\Phiring, \Psi]_{\mu\nu}&:=&\frac{\pr\LL}{\pr g^{\mu\nu}}-\frac 1 2 \g_{\mu\nu}\LL
     \eea                         
        We also define  the  source:
  \bea
  \JJ[\Phiring, \Psi]_\mu&:=&2\frac{\pr\LL[\psi]}{\pr  \Phiring^c}\pr_\mu \Phiring^c,\qquad c=1,2
  \eea 
  Note that,  in view of   the stationarity  of   $\Phiring$, 
  \beaa
  \T^\mu \JJ_\mu=2\frac{\pr\LL[\psi]}{\pr  \Phiring^c} \T^\mu \pr_\mu \Phiring^c-=0.
  \eeaa
  \begin{lemma} We have  the   local  conservation law:
  \beaa
  \D^\nu \QQ_{\mu\nu}= \JJ_\mu
  \eeaa
  \end{lemma}
  \begin{proof}
  Let $\chi_s$ be the 
one-parameter group of local diffeomorphisms generated by a 
 given vectorfield  $X$.  We shall use  the
 flow $\chi$ to vary the fields $\Psi$ according to  
\beaa 
\g_s& =& (\chi_s)_{*} \g,\quad  
\psi_s= (\chi_s)_{*}\Psi, \quad \phi_s=(\chi_s)_{*}\Phiring
\eeaa 
>From the invariance of the action integral under diffeomorphisms,\,\,
 $\SS(s) =
\SS[\Psi_s,\g_s,\Phiring_s] = \SS[\Psi,\g;\Phiring].$
Therefore,
\beaa
0 & = & \frac{d}{ds}\SS(s)\Bigm|_{s=0} 
=\int \frac{\pr\LL}{\pr\Psi^a}  X(\Psi^a)dv_{\g} +\int \big(\frac{\pr\LL}{\g^{\mu\nu}}-\frac 1 2 \g_{\mu\nu}\Llin \big)\dot{\g}_{\mu \nu} dv_\g+\int\frac{\pr\LL}{\pr\Phiring^a} \dot{\Phiring}^a    dv_{\g} \\
&=& \int \QQ^{\mu\nu} (\D_\mu X_\nu+\D_\nu X_\mu)     dv_{\g}  +2\int \JJ ^\mu  X_\mu dv_\g= -2\int   \D_\nu   \QQ^{\mu\nu} X_\mu  dv_\g+2\int \JJ ^\mu  X_\mu dv_\g
\eeaa
Since the vectorfield  $X^\mu $ is arbitrary we deduce,
\beaa
 -\D_\nu   \QQ^{\mu\nu}+ \JJ^\mu=0
\eeaa
 as desired.
\end{proof}
   In view of the definitions of $\QQ$ and $\LL$  we       can write 
   \bea
   \label{eq:QQexpression}
   \QQ_{\mu\nu}&=&\TT_{\mu\nu}-\frac 1 2 \g_{\mu\nu} ( \tr_\g \TT),      \qquad      \TT_{\mu\nu}:=E_\mu E_\nu+F_\mu F_\nu+M_\mu M_\nu
   \eea
   where,
     \begin{equation*}
  E_\mu:= \Db_\mu \phi=    \pr_\mu\phi+\psi A^{-1}\pr_\mu B,\,\,\,
  F_\mu:=\Db_\mu \psi=\pr_\mu\psi-\phi A^{-1}\pr_\mu B,\,\,\,
  M_\mu:= A^{-1}(\phi\pr_\mu B-\psi\pr_\mu A).
  \end{equation*}
The positivity property (a) is now an immediate consequence of  the  structure    \eqref{eq:QQexpression} of the energy momentum tensor.
Property (a)  is clearly verified    in the region where  $\T$ is time-like. It is well known that  at  every point of  the  ergoregion  where $r>r_H$,
    there exists a linear   combination of $\T$ and $\Z$,    $\T+c \Z$,     which is   timelike.     Therefore, since $\T\c E=\T\c F=\T\c M=0$,
\beaa
0<\TT(\T+c\Z, X)=\TT(\T, X). 
\eeaa
On the other hand, since $X$ is orthogonal to $\Z$, 
\beaa
\g(\T+cZ, X)=\g(\T, X).
\eeaa
Hence,
\beaa
0<\QQ(\T+c\Z, X)=\QQ(\T, X),
\eeaa
as  desired.

     \subsection{ New  coordinates}\label{sect-coord}
              As  well known the    Boyer-Lindquist  coordinates, are   singular  near  the horizon and as such are   not appropriate 
                for our task.   To avoid this difficulty it has become standard   to   define a new set of variables    which      are  well behaved       across the horizon and     coincide  with the Boyer-Lindquist coordinates  away from   it.

           We fix first a smooth function $\chi:\mathbb{R}\to[0,1]$ supported in the interval $(-\infty,5M/2]$ and equal to $1$ in the interval $(-\infty,9M/4]$, and define $g_1,g_2:(\r_h,\infty)\to\mathbb{R}$ such that
\begin{equation}\label{cha1}
g_1'(r)=\chi(r)\frac{2Mr}{\Delta},\qquad g_2'(r)=\chi(r)\frac{a}{\Delta}.
\end{equation}
We define the functions
\begin{equation}\label{cha2}
t_+:=t+g_1(r),\qquad \phi_+:=\phi+g_2(r).
\end{equation}
Therefore
\begin{equation*}
dt_+=dt+\chi(r)\frac{2Mr}{\Delta}dr,\qquad d\phi_+=d\phi+\chi(r)\frac{a}{\Delta}dr.
\end{equation*}
In $(\theta,r,t_+,\phi_+)$ coordinates, the metric $\g$ becomes\footnote{See the appendix  for more calculations in these coordinates.}    
\begin{equation}\label{cha3}
\begin{split}
\g&=q^2(d\theta)^2+\Big[\frac{q^2}{\Delta}(1-\chi^2(r))+\frac{2Mr+q^2}{q^2}\chi^2(r)\Big](dr)^2\\
&+2\chi(r)\frac{2Mr}{q^2}drdt_+-2\chi(r)\frac{a(\sin\theta)^2(q^2+2Mr)}{q^2}drd\phi_+\\
&+\frac{2Mr-q^2}{q^2}(dt_+)^2-\frac{4aMr(\sin\theta)^2}{q^2}dt_+d\phi_++\frac{\Sigma^2(\sin\theta)^2}{q^2}(d\phi_+)^2.
\end{split}
\end{equation}
     
Let
\begin{equation}\label{zq3}
\partial_1=\partial_\theta=\frac{d}{d\theta},\qquad\partial_2=\partial_r=\frac{d}{dr},\qquad\partial_3=\partial_t=\frac{d}{dt_+}=\T,\qquad\partial_4=\partial_\phi=\frac{d}{d\phi_+}=\Z.
\end{equation}
The nontrivial components of the metric $\g$ are
\begin{equation}\label{exp1}
\begin{split}
&\g_{11}=q^2,\qquad \g_{33}=\frac{2Mr-q^2}{q^2},\qquad \g_{34}=-\frac{2aMr(\sin\theta)^2}{q^2},\qquad \g_{44}=\frac{\Sigma^2(\sin\theta)^2}{q^2},\\
&\g_{22}=\frac{q^2}{\Delta}(1-\chi^2(r))+\frac{2Mr+q^2}{q^2}\chi^2(r),\\
&\g_{23}=\chi(r)\frac{2Mr}{q^2},\qquad \g_{24}=-\chi(r)\frac{a(\sin\theta)^2(q^2+2Mr)}{q^2}.
\end{split}
\end{equation}
The metric $\g$ extends smoothly to the larger open set
\begin{equation*}
\widetilde{\RR}=\{(\theta,r,t_+,\phi_+)\in(-\pi,\pi)\times(0,\infty)\times\mathbb{R}\times\mathbb{S}^1\}.
\end{equation*}
For $t\in\mathbb{R}$ and $c\in(0,\infty)$ let
\begin{equation}\label{surf1}
\Sigma^c_t:=\{(\theta,r,t_+,\phi_+)\in\widetilde{R}:t_+=t\text{ and }r>c\}.
\end{equation}
The surfaces $\Sigma_t^{\r_h}$, $t\in\mathbb{R}$, form a $\Z$-invariant foliation of spacelike surfaces of the domain of outer communications of the Kerr spacetime $\mathcal{K}(M,a)$. Moreover, the foliation is compatible with the Killing vector-field $\T$, i. e. $\Phi_{t_1}(\Sigma^c_{t_2})=\Sigma^c_{t_1+t_2}$ for any $t_1,t_2\in\mathbb{R}$,  where $\Phi_t$ denotes  the flow  associated 
 to $\T$. 

 As mentioned  earlier  we are interested in solutions  of  the form   \eqref{eq:nonl-perturb}, i.e., 
 $\Phi=(A',B')=(A,B)+\varepsilon (A\phi,A\psi)$ of the wave-map equation \eqref{zq5}, in  causal   domains  of the form 
\begin{equation}\label{surf2}
\mathcal{D}^{c}_I:=\cup_{t\in I}\Sigma^{c}_t=\{(\theta,r,t_+,\phi_+)\in\widetilde{\RR}:t_+\in I\text{ and }r>c\},
\end{equation}
where $I\subseteq\mathbb{R}$ is an interval and $c<\r_h$. Notice that if $c<\r_h$ then $\mathcal{D}^{c}_\mathbb{R}$ contains a 
small neighborhood of the future event horizon $\mathcal{H}^+$ as well as  the entire domain of outer communication. 
For any $c\in(0,\infty)$ and any interval $I\subseteq\mathbb{R}$ let 
\begin{equation}\label{surf3}
\mathcal{N}^{c}_I:=\{(\theta,r,t_+,\phi_+)\in\widetilde{R}:t_+\in I\text{ and }r=c\}.
\end{equation}
Notice that the hypersurfaces $\mathcal{N}^{c}_I$ are spacelike if $c<\r_h$, null (and contained in the future event 
horizon $\mathcal{H}^+$) if $c=\r_h$, and timelike if $c>\r_h$.

\subsection{Precise version of our second theorem} We define now our main function spaces:

\begin{definition}\label{MainDef}
For any $m\in\mathbb{Z}^+$, $c\in(0,\infty)$, and $t\in\mathbb{R}$ let $H^m(\Sigma_t^c)$ denote the usual $L^2$-based Sobolev space of functions on the hypersurface $\Sigma_t^c$, with respect to the induced Kerr metric (see \eqref{exp1}). Let
\begin{equation}\label{maindef1}
\widetilde{H}^m(\Sigma_t^c):=\Big\{f:\Sigma_t^c\to\mathbb{R}:\,\|f\|_{\widetilde{H}^m(\Sigma_t^c)}:=\|f\|_{H^m(\Sigma_t^c)}+\sum_{m'+m''=1}^m\|(\widetilde{\partial}_1/r)^{m'}\widetilde{\partial}_2^{m''}f\|_{L^2(\Sigma_t^c)}<\infty\Big\},
\end{equation}
where, by definition,
\begin{equation}\label{maindef2}
\widetilde{\partial}_1g:=\Big(\partial_1-\frac{2\cos\theta}{\sin\theta}\Big)g,\qquad \widetilde{\partial}_2g:=\partial_2g.
\end{equation}
For any $g\in C^1(\Sigma_t^c)$ satisfying $\Z(g)=0$ let
\begin{equation}\label{grad}
\nabla g:=(\partial_1g/r,\partial_2g),\qquad \widetilde{\nabla} g:=(\widetilde{\partial}_1g/r,\widetilde{\partial}_2g).
\end{equation}
Finally, let
\begin{equation}\label{maindef3}
\mathbf{H}^m(\Sigma_t^c):=\{(\phi,\psi):\Sigma_t^c\to\mathbb{R}\times\mathbb{R}:\,\|(\phi,\psi)\|_{\mathbf{H}^m(\Sigma_t^c)}:=\|\phi\|_{H^m(\Sigma_t^c)}+\|\psi\|_{\widetilde{H}^m(\Sigma_t^c)}<\infty\}.
\end{equation}
\end{definition}

For any $R\geq 33M/16$ let $\chi_{\geq R}:[0,\infty)\to[0,1]$ denote a smooth increasing function supported in $[R,\infty)$, equal to $1$ in $[2R-2M,\infty)$, and satisfying the natural differential inequalities. Let
\begin{equation}\label{outs4.9}
L:=\chi_{\geq 4M}(r)\Big(\partial_2+\frac{r}{r-2M}\partial_3\Big),
\end{equation}
For any $t\in\mathbb{R}$ and $(\phi,\psi)\in\mathbf{H}^1(\Sigma_t^c)$ we define the {\it{outgoing energy density}} $(e(\phi),e(\psi))$,
\begin{equation}\label{outs101.9}
\begin{split}
&e(\phi)^2:=\frac{(\partial_1\phi)^2}{r^2}+(L\phi)^2+\frac{M^2\big[(\partial_2\phi)^2+(\partial_3\phi)^2\big]}{r^2}+\frac{\phi^2}{r^2},\\
&e(\psi)^2:=\frac{(\partial_1\psi)^2+\psi^2(\sin\theta)^{-2}}{r^2}+(L\psi)^2+\frac{M^2\big[(\partial_2\psi)^2+(\partial_3\psi)^2\big]}{r^2}+\frac{\psi^2}{r^2}.
\end{split}
\end{equation}

We work in the axially symmetric case, therefore the relevant trapped null geodesics are still confined to a codimension $1$ set. Assuming that $a\ll M$, it is easy to see that the equation 
\begin{equation*}
r^3-3Mr^2+a^2r+Ma^2=0
\end{equation*} 
has a unique solution $r^\ast\in(M,\infty)$. Moreover, $r^\ast\in[3M-a^2/M,3M]$.

\begin{theorem}\label{MainTheorem}
Assume that $M\in (0,\infty)$, $N_0:=4$, $a\in[0,\overline{\varep}M]$ and $c_0\in[\r_h-\overline{\varep}M,\r_h]$, where $\overline{\varep}\in (0,1]$ is a sufficiently small constant. Assume that $T\geq 0$, and $(\phi,\psi)\in C^k([0,T]:\mathbf{H}^{N_0-k}(\Sigma_t^{c_0}))$, $k\in[0,N_0]$, is a solution 
of the system
\begin{equation}\label{asump0}
\begin{split}
\square\phi+2\frac{\D^\mu B}{A}\D_\mu\psi-2\frac{\D^\mu B\D_\mu B}{A^2}\phi+2\frac{\D^\mu B\D_\mu A}{A^2}\psi&=\mathcal{N}_\phi,\\
\square\psi-2\frac{\D^\mu B}{A}\D_\mu\phi-\frac{\D^\mu A\D_\mu A+\D^\mu B\D_\mu B}{A^2}\psi&=\mathcal{N}_\psi,
\end{split}
\end{equation}
satisfying
\begin{equation}\label{asump1}
\Z(\phi,\psi)=0\qquad\text{ in }\mathcal{D}^{c_0}_{[0,T]}.
\end{equation}
Then, for any $\alpha\in(0,2)$ and any $t_1\leq t_2\in[0,T]$,
\begin{equation}\label{outs100.9}
\begin{split}
\mathcal{B}_\alpha^{c_0}(t_1,t_2)+\int_{\Sigma_{t_2}^{c_0}}\frac{r^\alpha}{M^{\alpha}}\big[e(\phi)^2+e(\psi)^2\big]\,d\mu_t
&\leq \overline{C}_\alpha\int_{\Sigma_{t_1}^{c_0}}\frac{r^\alpha}{M^{\alpha}}\big[e(\phi)^2+e(\psi)^2\big]\,d\mu_t\\
&+\overline{C}_\alpha\int_{\mathcal{D}^{c_0}_{[t_1,t_2]}}\frac{r^\alpha}{M^{\alpha}}
 \big[e(\phi,\mathcal{N}_\phi)+e(\psi,\mathcal{N}_\psi)\big]\,d\mu,
\end{split}
\end{equation}
where $\overline{C}_\alpha$ is a large constant that may depend on $\alpha$,
\begin{equation}\label{outs102.9}
\begin{split}
\mathcal{B}_\alpha^{c_0}(t_1,t_2):=\int_{\mathcal{D}^{c_0}_{[t_1,t_2]}}\frac{r^\alpha}{M^{\alpha}}\Big\{&\frac{(r-r^\ast)^2}{r^3}\frac{|\partial_1\phi|^2+|\partial_1\psi|^2+\psi^2(\sin\theta)^{-2}}{r^2}+\frac{1}{r}\big[(L\phi)^2+(L\psi)^2\big]\\
&+\frac{1}{r^3}\big(\phi^2+\psi^2\big)+\frac{M^{2}}{r^{3}}\big[(\partial_2\phi)^2+(\partial_2\psi)^2\big]\\
&+\frac{M^2(r-r^\ast)^2}{r^5}\big[(\partial_3\phi)^2+(\partial_3\psi)^2\big]\Big\}\,d\mu,
\end{split}
\end{equation}
and, for $f\in\{\phi,\psi\}$,
\begin{equation}\label{outs100N}
e(f,\mathcal{N}_f):=|\mathcal{N}_f|\Big[(Lf)^2+\frac{M^2\big[(\partial_2f)^2+(\partial_3f)^2\big]+f^2}{r^2}\Big]^{1/2}.
\end{equation}
\end{theorem}

The point of proving an energy estimate such as \eqref{outs100.9} involving outgoing energies is that it leads directly to decay estimates. For example, we have the following corrolary:

\begin{corollary}\label{MainCoro}
Assume that $N_1=8$ and $(\phi,\psi)\in C^k([0,T]:\mathbf{H}^{N_1-k}(\Sigma_t^{c_0}))$, $k\in[0,N_1]$, is a solution of the system \eqref{asump0} with $\mathcal{N}_\phi=\mathcal{N}_\psi=0$. Then, for any $t\in[0,T]$ and $\beta<2$,
\begin{equation}\label{Alx30}
\int_{\Sigma_{t}^{c_0}}\big[e(\phi)^2+e(\psi)^2\big]\,d\mu_t\lesssim_{\beta} (1+t/M)^{-\beta}\sum_{k=0}^4M^{2k}\int_{\Sigma_{0}^{c_0}}\frac{r^2}{M^2}\big[e(\T^k\phi)^2+e(\T^k\psi)^2\big]\,d\mu_t.
\end{equation}
\end{corollary}

The point   of the corollary is the almost $(1+t/M)^{-2}$ decay of the outgoing energy on the hypersurface $\Sigma_{t}^{c_0}$, in terms of initial data; such a decay is not possible, of course,  for   the   standard energy. One can further commute the equation with the vector-field $\partial_r$ and use elliptic estimates to prove control decay of higher order outgoing energies as well. Such  estimates, with improved decay,  can then be combined, in principle, with a bootstrap argument to analyze globally the full semilinear system and prove the Partial Stability Conjecture in the case $a\ll M$. Note that the precise form of the system is given in Proposition \ref{equivalence}; the nonlinearities $\mathcal{N}_\phi^\varep$ and $\mathcal{N}_\psi^\varep$ are quadratic and appear to satisfy suitable null conditions which are needed to prove global existence.

The  explicit      loss of derivatives   of the estimate   \eqref{Alx30}  can be improved; however    some loss  is  unavoidable    due  to the degeneracy  of the  bulk integral   at $r=r^*$ in \eqref{outs100.9}.
 We note    that the analogous   decay estimate    for the standard wave equation    in Minkowski space  follows,   with $\b=2$ and  without the  loss of derivatives,   from  the    conservation  of  the conformal energy                (see,   for example,   section 3  in \cite{Kl}).  
 \subsection{Conclusions}   
   The estimates presented in this paper  offer   convincing evidence 
     for  the  validity of our conjecture. Further work is needed  to remove  the smallness  condition for $a/M$, provide  sufficiently  strong pointwise  decay  estimate  in the  wave zone region and implement the standard  approach for   proving  global existence   results  for  nonlinear 
       wave equations which  satisfy the null condition\footnote{ Such a   program   was  carried  out by J.    Luk     (in the simpler case  of the nonlinear stability of the trivial solution),     for  semi-linear  wave    equations verifying the null condition,  see   \cite{Luk}.  }.

\subsection{Organization} The rest of the paper is organized as follows. In section \ref{keyiden} we derive the main identities in the paper, including the precise form of the  system and the divergence identities; this provides an alternative explicit proof of Theorem \ref{thm:Energy}. In section \ref{Outli} we give an outline of the proof of the main theorem in the simplified case \eqref{eq:modelS}. In sections \ref{MoraEst} and \ref{lok} we give a complete proof of the main Theorem \ref{MainTheorem}, first in the case of the pure wave equation on the Schwarzschild space, and then for the full system on the Kerr spaces. In section \ref{CoRoProof} we provide a proof of Corollary \ref{MainCoro}, using Theorem \ref{MainTheorem} and an elliptic estimate. Finally, the appendix contains several explicit calculations in Kerr spaces, some Hardy inequalities, and some properties of the modified Sobolev spaces $\widetilde{H}^m$.

\section{Derivation of the main algebraic identities. Theorem \ref{thm:Energy} revisited}\label{keyiden}

Assume that $(A',B')=(A,B)+(\varepsilon A\phi,\varepsilon A\psi)$ is a solution of the wave-map equation \eqref{zq5} on some interval $I$, where $(\phi,\psi)\in C^k(I:\mathbf{H}^{N_1-k}(\Sigma_t^{c_0}))$, $k=0,\ldots,N_1$. The functions $(\phi,\psi)$ satisfy the system
\begin{equation*}
\begin{split}
&A^2\square\phi+2A\D^\mu B\D_\mu\psi-2\D^\mu B\D_\mu B\phi+2\D^\mu B\D_\mu A\psi\\
&+\varepsilon\big[A\phi\square(A\phi)-\D^\mu(A\phi)\D_\mu(A\phi)+\D^\mu(A\psi)\D_\mu(A\psi)\big]=0,
\end{split}
\end{equation*}
and
\begin{equation*}
\begin{split}
&A^2\square\psi-2A\D^\mu B\D_\mu\phi-(\D^\mu A\D_\mu A+\D^\mu B\D_\mu B)\psi\\
&+\varepsilon\big[A\phi\square(A\psi)-2\D^\mu(A\phi)\D_\mu(A\psi)]=0.
\end{split}
\end{equation*}
Using the formulas \eqref{zq5} these equations become
\begin{equation*}
\begin{split}
&A^2(1+\varepsilon\phi)\square\phi+2A\D^\mu B\D_\mu\psi-2\D^\mu B\D_\mu B\phi+2\D^\mu B\D_\mu A\psi\\
&+\varepsilon\big[A^2\D^\mu\psi\D_\mu\psi+2A\psi\D^\mu A\D_\mu\psi+\D^\mu A\D_\mu A\psi^2-A^2\D^\mu\phi\D_\mu\phi-\D^\mu B\D_\mu B\phi^2\big]=0,
\end{split}
\end{equation*}
and
\begin{equation*}
\begin{split}
&A^2(1+\varepsilon\phi)\square\psi-2A\D^\mu B\D_\mu\phi-(\D^\mu A\D_\mu A+\D^\mu B\D_\mu B)\psi\\
&+\varepsilon\big[-2A^2\D^\mu\phi\D_\mu\psi-\D^\mu A\D_\mu A\phi\psi-\D^\mu B\D_\mu B\phi\psi-2A\psi\D^\mu A\D_\mu\phi]=0.
\end{split}
\end{equation*}
We divide the equations by $A^2(1+\varepsilon\phi)$ to conclude that
\begin{equation}\label{no1}
\begin{split}
&\square\phi+2\frac{\D^\mu B}{A}\D_\mu\psi-2\frac{\D^\mu B\D_\mu B}{A^2}\phi+2\frac{\D^\mu B\D_\mu A}{A^2}\psi=\varepsilon\mathcal{N}_\phi^\varepsilon,\\
&\square\psi-2\frac{\D^\mu B}{A}\D_\mu\phi-\frac{\D^\mu A\D_\mu A+\D^\mu B\D_\mu B}{A^2}\psi=\varepsilon\mathcal{N}_\psi^\varepsilon,
\end{split}
\end{equation}
where
\begin{equation*}
\begin{split}
\mathcal{N}_\phi^\varepsilon&=\frac{A^2\D^\mu\phi\D_\mu\phi-A^2\D^\mu\psi\D_\mu\psi-2A\psi\D^\mu A\D_\mu\psi+\D^\mu B\D_\mu B\phi^2-\D^\mu A\D_\mu A\psi^2}{A^2(1+\varepsilon\phi)}\\
&+\frac{\phi}{A^2(1+\varepsilon\phi)}[2A\D^\mu B\D_\mu\psi-2\D^\mu B\D_\mu B\phi+2\D^\mu B\D_\mu A\psi],
\end{split}
\end{equation*}
and
\begin{equation*}
\begin{split}
\mathcal{N}_\psi^\varepsilon&=\frac{2A^2\D^\mu\phi\D_\mu\psi+(\D^\mu A\D_\mu A+\D^\mu B\D_\mu B)\phi\psi+2A\psi\D^\mu A\D_\mu\phi}{A^2(1+\varepsilon\phi)}\\
&-\frac{\phi}{A^2(1+\varepsilon\phi)}[2A\D^\mu B\D_\mu\phi+(\D^\mu A\D_\mu A+\D^\mu B\D_\mu B)\psi].
\end{split}
\end{equation*}
The formulas for the nonlinear terms $\mathcal{N}_\phi^\varepsilon$ and $\mathcal{N}_\psi^\varepsilon$ can be simplified, and the calculations can be reversed. To summarize, we have proved the following:

\begin{proposition}\label{equivalence}
Assume $I\subseteq\mathbb{R}$ is an interval, $\varepsilon>0$, and $(\phi,\psi)\in C^k(I:\mathbf{H}^{N_1-k}(\Sigma_t^{c_0}))$, $k=0,\ldots,N_1$. Then $(A',B')=(A,B)+(\varepsilon A\phi,\varepsilon A\psi)$ is a solution of the wave-map equation \eqref{zq5} on the interval $I$ if and only if $(\phi,\psi)$ satisfy the nonlinear system 
\begin{equation}\label{cx1.5}
\begin{split}
&\square\phi+2\frac{\D^\mu B}{A}\D_\mu\psi-2\frac{\D^\mu B\D_\mu B}{A^2}\phi+2\frac{\D^\mu B\D_\mu A}{A^2}\psi=\varepsilon\mathcal{N}_\phi^\varepsilon,\\
&\square\psi-2\frac{\D^\mu B}{A}\D_\mu\phi-\frac{\D^\mu A\D_\mu A+\D^\mu B\D_\mu B}{A^2}\psi=\varepsilon\mathcal{N}_\psi^\varepsilon,
\end{split}
\end{equation}
where
\begin{equation}\label{cx2}
\begin{split}
&\mathcal{N}_\phi^\varepsilon=\frac{A^2(\D^\mu\phi\D_\mu\phi-\D^\mu\psi\D_\mu\psi)+(\phi\D^\mu B-\psi\D^\mu A)(2A\D_\mu\psi-\phi\D_\mu B+\psi\D_\mu A)}{A^2(1+\varepsilon\phi)},\\
&\mathcal{N}_\psi^\varepsilon=\frac{2A^2\D^\mu\phi\D_\mu\psi+2A(\psi\D^\mu A-\phi\D^\mu B)\D_\mu\phi}{A^2(1+\varepsilon\phi)}.
\end{split}
\end{equation}
\end{proposition}

\subsection{The energy-momentum tensor}\label{EMTensor} We study now solutions of the system 
\begin{equation}\label{cx1}
\begin{split}
\square\phi+2\frac{\D^\mu B}{A}\D_\mu\psi-2\frac{\D^\mu B\D_\mu B}{A^2}\phi+2\frac{\D^\mu B\D_\mu A}{A^2}\psi&=\mathcal{N}_\phi,\\
\square\psi-2\frac{\D^\mu B}{A}\D_\mu\phi-\frac{\D^\mu A\D_\mu A+\D^\mu B\D_\mu B}{A^2}\psi&=\mathcal{N}_\psi.
\end{split}
\end{equation}
Our main goal is to construct a suitable energy-momentum tensor that verifies a good divergence equation. More precisely, let
\begin{equation}\label{zq30}
E_\mu:=\D_\mu\phi+\psi A^{-1}\D_\mu B,\quad F_\mu:=\D_\mu\psi-\phi A^{-1}\D_\mu B,\quad M_\mu:=\frac{\phi\D_\mu B-\psi\D_\mu A}{A}.
\end{equation}
Using the formulas
\begin{equation}\label{zq31}
A\D_\mu\phi=AE_\mu-\psi \D_\mu B,\quad A\D_\mu\psi=AF_\mu+\phi \D_\mu B,
\end{equation}
the identities \eqref{cx1} and \eqref{zq5} show that
\begin{equation}\label{zq32}
\begin{split}
&\D^\mu E_\mu+\frac{\D^\mu BF_\mu}{A}-\frac{\D^\mu BM_\mu}{A}=\mathcal{N}_\phi,\\
&\D^\mu F_\mu-\frac{\D^\mu BE_\mu}{A}+\frac{\D^\mu AM_\mu}{A}=\mathcal{N}_\psi,\\
&\D^\mu M_\mu-\frac{\D^\mu BE_\mu}{A}+\frac{\D^\mu AF_\mu}{A}=0.
\end{split}
\end{equation}
We also calculate
\begin{equation}\label{zq33}
\begin{split}
\D_\mu E_\nu-\D_\nu E_\mu&=\frac{F_\mu\D_\nu B-F_\nu\D_\mu B}{A}+\frac{M_\mu\D_\nu B-M_\nu\D_\mu B}{A},\\
\D_\mu F_\nu-\D_\nu F_\mu&=-\frac{E_\mu\D_\nu B-E_\nu\D_\mu B}{A}-\frac{M_\mu\D_\nu A-M_\nu\D_\mu A}{A},\\
\D_\mu M_\nu-\D_\nu M_\mu&=\frac{E_\mu\D_\nu B-E_\nu\D_\mu B}{A}-\frac{F_\mu\D_\nu A-F_\nu\D_\mu A}{A}.\\
\end{split}
\end{equation}

Let
\begin{equation}\label{zq35}
\begin{split}
&T_{\mu\nu}:=E_\mu E_\nu+F_\mu F_\nu+M_\mu M_\nu,\\
&Q_{\mu\nu}:=T_{\mu\nu}+\g_{\mu\nu}\mathcal{L},\\
&\mathcal{L}:=-(1/2)\g^{\al\be}T_{\al\be}=-(1/2)(E_\al E^\al+F_\al F^\al+M_\al M^\al).
\end{split}
\end{equation}
We calculate the divergence
\begin{equation*}
\begin{split}
\D^\mu Q_{\mu\nu}&=E_\nu\D^\mu E_\mu+E^\mu(\D_\mu E_\nu-\D_\nu E_\mu)\\
&+F_\nu\D^\mu F_\mu+F^\mu(\D_\mu F_\nu-\D_\nu F_\mu)\\
&+M_\nu\D^\mu M_\mu+M^\mu(\D_\mu M_\nu-\D_\nu M_\mu),
\end{split}
\end{equation*}
Using \eqref{zq32} and \eqref{zq33} we calculate
\begin{equation*}
\begin{split}
E_\nu\D^\mu E_\mu+E^\mu(\D_\mu E_\nu-\D_\nu E_\mu)&=\frac{E_\nu(\D^\mu BM_\mu-\D^\mu BF_\mu)-F_\nu E^\mu\D_\mu B-M_\nu E^\mu\D_\mu B}{A}\\
&+\frac{\D_\nu B(E^\mu F_\mu+E^\mu M_\mu)}{A}+\mathcal{N}_\phi E_\nu,
\end{split}
\end{equation*}
\begin{equation*}
\begin{split}
F_\nu\D^\mu F_\mu+F^\mu(\D_\mu F_\nu-\D_\nu F_\mu)&=\frac{E_\nu F^\mu\D_\mu B+F_\nu(\D^\mu BE_\mu-\D^\mu AM_\mu)+M_\nu F^\mu\D_\mu A }{A}\\
&+\frac{-\D_\nu B E^\mu F_\mu-\D_\nu AF^\mu M_\mu}{A}+\mathcal{N}_\psi F_\nu,
\end{split}
\end{equation*}
and
\begin{equation*}
\begin{split}
M_\nu\D^\mu M_\mu+M^\mu(\D_\mu M_\nu-\D_\nu M_\mu)&=\frac{-E_\nu M^\mu\D_\mu B+F_\nu M^\mu\D_\mu A +M_\nu(\D^\mu BE_\mu-\D^\mu AF_\mu) }{A}\\
&+\frac{\D_\nu BM^\mu E_\mu-\D_\nu AM^\mu F_\mu}{A}.
\end{split}
\end{equation*}
Therefore
\begin{equation}\label{zq36}
\D^\mu Q_{\mu\nu}=\frac{2\D_\nu BM^\mu E_\mu-2\D_\nu AM^\mu F_\mu}{A}+\mathcal{N}_\phi E_\nu+\mathcal{N}_\psi F_\nu.
\end{equation}

\subsection{     Divergence  Identities }
Given a vector-field $X$, a function $w$, and $1$-forms $m,m'$ we define the form
\begin{equation}\label{zq40}
P_\mu=P_\mu[X,w,m,m']:=Q_{\mu\nu}X^\nu+\frac{1}{2}w(\phi E_\mu+\psi F_\mu)-\frac{1}{4}\D_\mu w(\phi^2+\psi^2)+\frac{1}{4}(m_\mu\phi^2+m'_\mu\psi^2). 
\end{equation}
Then, using \eqref{zq30}--\eqref{zq32} we calculate the divergence
\begin{equation*}
\begin{split}
\D^\mu P_\mu&=X^\nu \JJ_\nu+\frac{1}{2}Q_{\mu\nu}{}^{(X)}\pi^{\mu\nu}+\frac{1}{2}\D^\mu w(\phi E_\mu+\psi F_\mu)+\frac{1}{2}w(\D^\mu\phi E_\mu+\D^\mu\psi F_\mu)\\
&+\frac{1}{2}w(\phi \D^\mu E_\mu+\psi \D^\mu F_\mu)-\frac{1}{4}\square w(\phi^2+\psi^2)-\frac{1}{2}\D^\mu w(\phi\D_\mu\phi+\psi\D_\mu\psi)\\
&+\frac{1}{4}(\phi^2\D^\mu m_\mu+\psi^2\D^\mu m'_\mu)+\frac{1}{2}(\phi m^\mu\D_\mu\phi+\psi{m'}^\mu\D_\mu\psi)\\
&=X^\nu \JJ_\nu+\frac{1}{2}Q_{\mu\nu}{}^{(X)}\pi^{\mu\nu}-\frac{1}{4}\square w(\phi^2+\psi^2)\\
&+\frac{1}{4}(\phi^2\D^\mu m_\mu+\psi^2\D^\mu m'_\mu)+\frac{1}{2}(\phi m^\mu\D_\mu\phi+\psi{m'}^\mu\D_\mu\psi)+E',
\end{split}
\end{equation*}
where
\begin{equation*}
\begin{split}
E'&=\frac{1}{2}\D^\mu w(\phi E_\mu+\psi F_\mu-\phi\D_\mu\phi-\psi\D_\mu\psi)+\frac{1}{2}w(\D^\mu\phi E_\mu+\D^\mu\psi F_\mu+\phi \D^\mu E_\mu+\psi \D^\mu F_\mu)\\
&=0+\frac{1}{2}w(E^\mu E_\mu +F^\mu F_\mu +M^\mu M_\mu+\phi\mathcal{N}_\phi+\psi\mathcal{N}_\psi).
\end{split}
\end{equation*}
Therefore
\begin{equation*}\label{zq41}
\begin{split}
\D^\mu P_\mu&=X^\nu \JJ_\nu+\frac{1}{2}Q_{\mu\nu}{}^{(X)}\pi^{\mu\nu}-\frac{1}{4}\square w(\phi^2+\psi^2)-w\mathcal{L}\\
&+\frac{1}{4}(\phi^2\D^\mu m_\mu+\psi^2\D^\mu m'_\mu)+\frac{1}{2}(\phi m^\mu\D_\mu\phi+\psi{m'}^\mu\D_\mu\psi)+\frac{1}{2}w(\phi\mathcal{N}_\phi+\psi\mathcal{N}_\psi).
\end{split}
\end{equation*}
\subsection{Summary. }
We summarize  the results  of the section  in  the following:

\begin{proposition}\label{prop001}
(i) Assume that $(\phi,\psi)\in C^k(I:\mathbf{H}^{N_0-k}(\Sigma_t^{c_0}))$, $k=0,\ldots,N_0$ satisfy the system \eqref{cx1}. Let
\begin{equation}\label{zq51}
\begin{split}
&E_\mu:=\D_\mu\phi+\psi A^{-1}\D_\mu B,\qquad F_\mu:=\D_\mu\psi-\phi A^{-1}\D_\mu B,\qquad M_\mu:=\frac{\phi\D_\mu B-\psi\D_\mu A}{A},\\
&Q_{\mu\nu}:=E_\mu E_\nu+F_\mu F_\nu+M_\mu M_\nu+\g_{\mu\nu}\mathcal{L},\\
&\mathcal{L}:=-\frac{1}{2}(E_\al E^\al+F_\al F^\al+M_\al M^\al).
\end{split}
\end{equation}
Then
\begin{equation}\label{zq52}
\D^\mu Q_{\mu\nu}=:J_\nu=\frac{2\D_\nu BM^\mu E_\mu-2\D_\nu AM^\mu F_\mu}{A}+\mathcal{N}_\phi E_\nu+\mathcal{N}_\psi F_\nu.
\end{equation}

(ii) Let 
\begin{equation}\label{zq51.5}
P_\mu=P_\mu[X,w,m,m']:=Q_{\mu\nu}X^\nu+\frac{1}{2}w(\phi E_\mu+\psi F_\mu)-\frac{1}{4}\D_\mu w(\phi^2+\psi^2)+\frac{1}{4}(m_\mu\phi^2+m'_\mu\psi^2),
\end{equation}
where $X$ is a smooth vector-field, $w$ is a smooth function, and $m,m'$ are smooth $1$-forms. Then 
\begin{equation}\label{zq52.5}
\begin{split}
2\D^\mu P_\mu&=2X^\nu J_\nu+Q_{\mu\nu}{}^{(X)}\pi^{\mu\nu}-2w\mathcal{L}+(\phi m^\mu\D_\mu\phi+\psi{m'}^\mu\D_\mu\psi)\\
&+\frac{1}{2}\phi^2(\D^\mu m_\mu-\square w)+\frac{1}{2}\psi^2(\D^\mu m'_\mu-\square w)+w(\phi\mathcal{N}_\phi+\psi\mathcal{N}_\psi).
\end{split}
\end{equation}
\end{proposition}
Note that  theorem  \ref{thm:Energy}  is an immediate consequence of the  first part of the proposition. Indeed, assuming
 that $(\NN_\phi, \NN_\psi)=0$   it   is immediate that   $J$ is orthogonal to $\T$.  The positivity
  of the energy momentum tensor $\QQ$  is an immediate consequence of its form  \eqref{zq51}.

\section{Main ideas in the proof of Theorem  \ref{MainTheorem}  }\label{Outli} In this section we      provide  main ideas  and   motivation  for    the various choices  we need to make    in the proof of theorem \ref{MainTheorem}.     Our proof follows   the  well established   pattern 
of proving         integrated   local energy decay estimates on black holes,   such as Schwarzschild,    for which the ergoregion  is trivial  and   the   trapped  region       is contained  to a level surface     $r=r^*>\rH$.  It is quite fortunate that      our   axially   symmetric  linearized system can be treated in the same manner.     Though  our treatment    follows  the    clear  and    efficient   approach   of    \cite{MaMeTaTo},      we    should  point out that many of the   ideas go back    to    other  authors   such as  \cite{B-S3}, \cite{B-St},\cite{DaRo1}.   An essential ingredient in the proof  is   to take into account the red shift      effects of the   horizon,  idea which goes back to  \cite{DaRo1}.
 
             In our problem we need to make two important modifications.  Most importantly, to get any estimate at all,  we  need to account for the     source term $\JJ$. This requires, in particular, a   serious  modification of the     current  $P_\mu$ in \eqref{zq52.5},   modification  which adds considerably to the complexity of the proof. 
             
  The second important modification  has to do with  the   presence of weights in   our main estimate.  Typically,   integrated    decay estimates    are  designed to deal with  the  region close to the  black hole,    most importantly the trapping region. They are then complemented  by   weighted estimates   in the asymptotic  region.  Thus, for example,    J. Luk (see  \cite{Luk}),    relies  on an    integrated local   decay estimate (proved    earlier  by Dafermos-Rodnianski (see \cite{DaRo2})     for  small  $a/M$), which he combines with   weighted     estimates   in the asymptotic   region  based on a straightforward adaptation of the  classical   conformal method.   The use of  conformal  method, however,
 is  quite  awkward    in the   black hole  region,    because  the  weights  involved  in  the conformal method  lead   to errors  which  grow   linearly    in  $t$.    This    problem  was  later        fixed by a different method of Dafermos-Rodnianski in  \cite{Da-Ro3},  who   replace  the conformal method by  $r$-weighted estimates.  The new method   allows one to prove decay estimate for the     energy associated   to       hypersurfaces  which are spacelike   near the black hole region but become     null   in the asymptotic region.  This depends, however, on having first derived an integrated  local decay estimate\footnote{The $r$-weighted estimates    produce  boundary terms    which    are estimated   with the help of  the integrated decay estimate. Because  of the degenerate  nature of this latter,    the method leads   to   an overall a loss of derivatives.}.  
     In our work here   we refine the analysis significantly by deriving   $r$-weighted  estimates for  the 
 \textit{outgoing     energy}  across spacelike hypersurfaces,     \textit{simultaneously}  with the integrated     local  decay  estimates.
 
 \subsection{Outline of the proof}\label{outline} We discuss now the main ideas in the proof. For simplicity, we consider only the equation for $\psi$ in the Schwarzschild case $a=0$, which carries most of the  conceptual difficulties of the problem. In this case $B=0$, $A=r^2(\sin\theta)^2$, and the equation is
\begin{equation}\label{OU1}
\square\psi-\frac{4-8(M/r)(\sin\theta)^2}{r^2(\sin\theta)^2}\psi=0.
\end{equation}
As in \eqref{prop001} we define
\begin{equation}\label{OU2}
\begin{split}
F_\mu:=\D_\mu\psi,\quad M_\mu:=\frac{-\psi\D_\mu A}{A},\quad Q_{\mu\nu}:=F_\mu F_\nu+M_\mu M_\nu-\frac{1}{2}\g_{\mu\nu}(F_\al F^\al+M_\al M^\al).
\end{split}
\end{equation}
For suitable triplets $(X,w,m')$ we define 
\begin{equation}\label{OU3}
\widetilde{P}_\mu=\widetilde{P}_\mu[X,w,m']:=Q_{\mu\nu}X^\nu+\frac{w}{2}\psi F_\mu-\frac{\psi^2}{4}\D_\mu w+\frac{\psi^2}{4}m'_\mu-\frac{X^\nu\D_\nu A}{A}\frac{\D_\mu A}{A}\psi^2.
\end{equation}
Notice the correction $-\frac{X^\nu\D_\nu A}{A}\frac{\D_\mu A}{A}\psi^2$, compared to the definition of $P$ in \eqref{zq51.5}, which is needed to partially compensate for the source term $J$. Then we have the divergence identity
\begin{equation}\label{OU4}
2\D^\mu \widetilde{P}_\mu=\sum_{j=1}^5 L^j,
\end{equation}
where
\begin{equation}\label{OU5}
\begin{split}
&L^1=L^1[X,w,m']:=Q_{\mu\nu}{}^{(X)}\pi^{\mu\nu}+w(F_\alpha F^\alpha+M_\alpha M^\alpha),\\
&L^2=L^2[X,w,m']:=\psi{m'}^\mu\D_\mu\psi,\\
&L^3=L^3[X,w,m']:=\frac{1}{2}\psi^2(\D^\mu m'_\mu-\square w),\\
&L^4=L^4[X,w,m']:=-2\D^\mu\Big[\frac{X^\nu\D_\nu A}{A}\frac{\D_\mu A}{A}\Big]\psi^2.
\end{split}
\end{equation}

The divergence identity gives
\begin{equation}\label{OU7}
\int_{\Sigma_{t_1}^c} \widetilde{P}_\mu n_0^\mu\,d\mu_{t_1}=\int_{\Sigma_{t_2}^c}\widetilde{P}_\mu n_0^\mu\,d\mu_{t_2}+\int_{\mathcal{N}^c_{[t_1,t_2]}}\widetilde{P}_\mu k_0^\mu\,d\mu_c+\int_{\mathcal{D}^c_{[t_1,t_2]}}\D^\mu \widetilde{P}_\mu\,d\mu,
\end{equation}
where $t_1,t_2\in[0,T]$, $c\in(c_0,2M]$, $n_0:=n/|\g^{33}|^{1/2}$, $k_0:=k/|\g^{22}|^{1/2}$, and the integration is with respect to the natural measures induced by the metric $\g$. To prove the main theorem we  need  to choose a  suitable multiplier triplet  $(X,w,m')$ in a such a way that all the terms in the identity above are nonnegative. This is the {\it{method of simultaneous inequalities}} of Marzuola--Metcalfe--Tataru--Tohaneanu \cite{MaMeTaTo}.

To accomplish our task we need    to superimpose four different choices of multiplier triplets $(X, w, m')$,   denoted $(X_{(k)},w_{(k)},m'_{(k)} )$, $k\in\{1,2,3,4\}$.  The first multiplier ($ k=1$) is important in a neighborhood of the trapped set $\{r=3M\}$; the second multiplier ($k=2$)  is important in a neighborhood of the horizon $\{r=\rH\}$; the third multiplier ($k=3$) is important at infinity, in the construction of outgoing energies at infinity; the fourth multiplier is important to control the term $L^4$, which is connected to the presence of the nontrivial potential in \eqref{OU1}. 
  
\subsubsection{The multipliers $(X_{(1)},w_{(1)},m'_{(1)} )$ and $(X_{(2)},w_{(2)},m'_{(2)} )$.} The first two multipliers are similar to the multipliers used in \cite{MaMeTaTo} in the case of the homogeneous wave equation. Set
\begin{equation*}
\begin{split}
&X_{(1)}:=f_1(r)\partial_2+g_1(r)\partial_3,\qquad f_1(r):=\frac{a_1(r)\Delta}{r^2},\qquad g_1(r):=\frac{a_1(r)\chi(r)2M}{r}+1,\\
&w_{(1)}(r,\theta):=f'_1(r)+f_1(r)\partial_r\log\big(r^4/\Delta)-\eps_1\widetilde{w}(r),\\
&\widetilde{w}(r):=M^2(r-33M/16)^3(r-r^\ast)^2r^{-8}\mathbf{1}_{[33M/16,\infty)}(r),\\
&m'_{(1)}:=0,\\
\end{split}
\end{equation*}
where $r^\ast=3M$, $\eps_1\in(0,1]$ is a small constant, and $a_1:(0,\infty)\to\mathbb{R}$ is a smooth function. The important function $a_1$, which vanishes on the trapped region $\{r=r^\ast\}$, is defined by
\begin{equation*}
\begin{split}
&R(r):=(r-r^\ast)(r+2M)+6M^2\log\Big(\frac{r-2M}{r^\ast-2M}\Big),\\
&a_1(r):=r^{-2}\delta^{-1}\kappa(\delta R(r))+\Big[\frac{r^\ast-2M}{r}-\frac{6M^2}{r^2}\log\Big(\frac{r-\r_h}{r^\ast-\r_h}\Big)\Big]\chi_{\geq DM}(r),
\end{split}
\end{equation*}
where $D$ is a sufficiently large constant, $\delta=\eps_2^2M^{-2}$ for a small positive constant $\eps_2$, and $\kappa:\mathbb{R}\to\mathbb{R}$ is an increasing smooth function satisfying $\kappa(y)=y$ on $[-1,\infty)$ and $\kappa(y)=-2$ on $(-\infty,-3]$. This is essentially the choice of \cite{MaMeTaTo}, except for the correction at infinity, containing the cutoff function $\chi_{\geq DM}$; this correction is needed in order to match properly with the third multiplier at infinity to produce outgoing energies.

In a small neighborhood of the horizon we need to use the redshift effect. We define the second multiplier
\begin{equation*}
\begin{split}
&X_{(2)}:=f_2(r)\partial_2+g_2(r)\partial_3,\qquad f_2(r):=-\eps_2a_2(r),\qquad g_2(r):=\eps_2 a_2(r)(1-\eps_2),\\
&w_{(2)}(r):=-2\eps_2 a_2(r)/r,\qquad m'_{(2)2}=m'_{(2)3}:=\eps_2M^{-2}\gamma(r),\qquad m'_{(2)1}=m'_{(2)4}:=0,
\end{split}
\end{equation*}
where
\begin{equation*}
a_2(r):=
\begin{cases}
M^{-3}(9M/4-r)^3\qquad&\text{ if }r\leq 9M/4,\\
0\qquad&\text{ if }r\geq 9M/4,
\end{cases}
\end{equation*}
and $\gamma:[c_0,\infty)\to[0,1]$ is a function supported in $[c_0,17M/8]$, satisfying $\gamma(2M)=1/2$ and the more technical property \eqref{other2}. As in \cite{MaMeTaTo}, the multipliers $(X_{(1)},w_{(1)},m'_{(1)} )$ and $(X_{(2)},w_{(2)},m'_{(2)} )$ cooperate well  to generate  mostly positive bulk contributions. More precisely, the constants $\eps_1,  \eps_2$ can be chosen such that, for some absolute constant $\eps_3>0$,
\begin{equation}\label{OU10}
\begin{split}
\sum_{j=1}^4\big(L^j_{(1)}+L^j_{(2)}\big)&\geq \eps_3\sum_{Y\in\{F,M\}}\Big[\frac{(r-r^\ast)^2}{r^3}(Y_1/r)^2+\frac{M^2}{r^3}(Y_2)^2+\frac{M^2(r-r^\ast)^2}{r^5}(Y_3)^2\Big]\\
&+\eps_3\frac{M}{r^4}\psi^2-\eps_3^{-1}\frac{M}{r^4}\mathbf{1}_{[DM,\infty)}(r)\psi^2+\widetilde{L},
\end{split}
\end{equation}
where 
\begin{equation}\label{OU11}
\begin{split}
\widetilde{L}:=&\frac{8\Delta(r^2-4Mr)}{r^7}a_1(r)\psi^2+(1-2C_1\eps_1)\mathbf{1}_{[r^\ast,\infty)}(r)\Big\{\frac{M}{r^4}\Big(7-\frac{44M}{r}+\frac{72M^2}{r^2}\Big)\psi^2\\
&+\frac{8a_1(r)(r-r^\ast)}{r^4}\frac{(\cos\theta)^2}{(\sin\theta)^2}\psi^2+\frac{2a_1(r)(r-r^\ast)}{r^4}(F_1)^2+2a'_1(r)\frac{\Delta^2}{r^4}(F_2)^2\Big\}.
\end{split}
\end{equation}
Moreover, letting $\widetilde{P}_{(j)}:=\widetilde{P}_\mu[X_{(j)},w_{(j)},m'_{(j)}]$, $j=1,2$, we have
\begin{equation}\label{OU12}
2(\widetilde{P}_{(1)\mu}+\widetilde{P}_{(2)\mu}) k^\mu\geq\eps_3\sum_{Y\in\{F,M\}}\big[(Y_1/r)^2+(Y_2)^2(2-c/M)\big]+\eps_3M^{-2}\psi^2-\eps_3^{-1}(F_3)^2,
\end{equation}
along $\mathcal{N}^c_{[t_1,t_2]}$. Also, with $p=2M/r$,
\begin{equation}\label{OU13}
\begin{split}
2(\widetilde{P}_{(1)\mu}+\widetilde{P}_{(2)\mu}) n^\mu&\geq -\eps_3^{-1}\big\{\widetilde{e}_0+\mathbf{1}_{[8M,2DM]}(r)(F_3)^2\big\}\\
&-\frac{\chi_{\geq 8M}(r)(1-p)}{r^2}\partial_2(r\psi^2)+\eps_3(F_2)^2\mathbf{1}_{(c_0,17M/8]}(r),
\end{split}
\end{equation}
and
\begin{equation}\label{OU14}
\begin{split}
2(\widetilde{P}_{(1)\mu}+\widetilde{P}_{(2)\mu}) n^\mu&\leq \eps_3^{-1}\big\{\widetilde{e}_0+\mathbf{1}_{[8M,2DM]}(r)(F_3)^2\big\}\\
&-\frac{\chi_{\geq 8M}(r)(1-p)}{r^2}\partial_2(r\psi^2)+\eps_3^{-1}(F_2)^2\mathbf{1}_{(c_0,17M/8]}(r),\\
\end{split}
\end{equation}
where
\begin{equation*}
\begin{split}
\widetilde{e}_0&=\frac{(F_1)^2+(M_1)^2}{r^2}+(L\psi)^2+\frac{M^2|r-2M|}{r^3}(F_2)^2+\frac{M^2}{r^2}(F_3)^2+\frac{1}{r^2}\psi^2.
\end{split}
\end{equation*}
Notice that the bulk terms in \eqref{OU10} are mostly positive, with the exception of the term $\widetilde{L}$. The terms along $\mathcal{N}^c_{[t_1,t_2]}$ are also mostly positive. On the other hand, the bounds \eqref{OU13} and \eqref{OU14} we have so far on the integrals along the hypersurfaces $\Sigma_t^{c}$ are very weak; these bounds will be improved by choosing a suitable multiplier $(X_{(3)},w_{(3)},m'_{(3)})$ at infinity.

\subsubsection{The multiplier $(X_{(4)},w_{(4)},m'_{(4)} )$.} Our next goal is to control the term $\widetilde{L}$ in \eqref{OU11}. This is a new term, when compared to solutions of the homogeneous wave equation, connected to the nontrivial potential in \eqref{OU1} and the bulk term $L^4$ in \eqref{OU5}. Since $a'_1(r)\geq 0$ and $a_1(r)(r-r^\ast)\geq 0$, this term can only be problematic in the region $\{r\in[r^\ast,4M]\}$. We define
\begin{equation*}
\begin{split}
&X_{(4)}:=0,\qquad w_{(4)}:=0,\\
&\widetilde{m}'_{(4)1}(r,\theta):=-(1-2C_1\eps_1)
\frac{8(r-r^\ast)a_1(r)\chi_{\leq 6R}(r)}{r^2}\frac{\cos\theta}{\sin\theta}\mathbf{1}_{[r^\ast,\infty)}(r),\\
&\widetilde{m}'_{(4)2}(r):=(1-2C_1\eps_1)\frac{2b(r)}{\Delta},\qquad \widetilde{m}'_{(4)3}:=0,\qquad\widetilde{m}'_{(4)4}:=0,
\end{split}
\end{equation*}
for a suitable function $b$ supported in $[r^\ast,4M]$. Careful estimates, as in Lemma \ref{clan2}, and completion of squares show that one can choose the function $b$ in such a way that
\begin{equation}\label{OU20}
L^1_{(4)}=L^4_{(4)}=0,\qquad \widetilde{L}+L^2_{(4)}+L^3_{(4)}\geq -C_2|2M-c_0|r^{-4}\psi^2
\end{equation}
for some constant $C_2$ sufficiently large, and 
\begin{equation}\label{OU21}
\big|2\widetilde{P}_{(4)\mu}n^\mu\big|\lesssim \eps_3^{-1}\psi^2/r^2\qquad \text{ and }\qquad 2\widetilde{P}_{(4)\mu}k^\mu=0\,\,\text{ along }\,\,\mathcal{N}^c_{[t_1,t_2]}.
\end{equation}
These two bounds can be combined with \eqref{OU10}--\eqref{OU14} to effectively remove the contribution of the term $\widetilde{L}$.

 \subsubsection{The multiplier $(X_{(3)},w_{(3)},m'_{(3)} )$.} Finally, we are ready to define the multiplier at infinity and close the estimate. First of all, to obtain any simultaneous estimate at all, we need to make sure that the contributions of the integrals of $2\widetilde{P}_\mu n^\mu$ on the hypersurfaces $\Sigma_t^c$ are positive. So far, these integrals are far from positive, in view of the estimates \eqref{OU13}, \eqref{OU14}, and \eqref{OU21}. 
 
 The formula \eqref{Qy1} shows that
 \begin{equation*}
 2n^\mu\widetilde{P}_\mu[K\partial_3,0,0]=2n^\mu Q_{\mu\nu}(K\partial_3)^\nu=K\sum_{Y\in\{F,M\}}\big[\g^{11}(Y_1)^2+\g^{22}(Y_2)^2+(-\g^{33})(Y_3)^2\big].
 \end{equation*}
Therefore, one could make the integrals of $2\widetilde{P}_\mu n^\mu$ along the hypersurfaces $\Sigma_t^c$ positive by adding a multiplier of the form $(K\partial_3,0,0)$, for some positive constant $K$ sufficiently large, and using a Hardy estimate to control the integral of the $0$'s order term in terms of the first order terms. Notice that such a multiplier does not affect the bulk integrals. This is precisely the argument used in \cite{MaMeTaTo} to close the simultaneous estimate for the standard energy for the wave equation.

In our case, however, we are looking to prove stronger estimates involving outgoing energies. A multiplier of the form $(K\partial_3,0,0)$ is not allowed, since this would create contributions at infinity of the form  $(F_2)^2+(F_3)^2$, which are unacceptable in view of the definition \eqref{outs101.9}. Instead, we choose the last multiplier of the form
\begin{equation}\label{OU26}
\begin{split}
&X_{(3)}:=f_3\partial_2+\Big(\frac{f_3}{1-p}+g_3\Big)\partial_3,\qquad w_{(3)}:=\frac{2f_3}{r},\\
&m'_{(3)1}:=m'_{(3)4}:=0,\qquad m'_{(3)2}:=\frac{2h_3}{r(1-p)},\qquad m'_{(3)3}:=-\frac{2h_3}{r},
\end{split}
\end{equation} 
for some suitable functions $f_3,g_3,h_3$. The function $f_3$ should behave like  $(r/M)^\a$ for large $r$, in order to produce the desired power in the outgoing energy. To make sure that it does not interfere with the crucial  trapping region we  have to  choose it to vanish for  $r\leq 8M$. 
 The role of the function $g_3$ is to match, to some extent, the role played by the multiplier $K\T$ in the  boundary estimate discussed earlier. Thus we choose $g_3$  to be a very large constant when $r\leq C_4M$, for some large constant $C_4$,   but we choose it to decay as $r\to \infty$, at the rate $r^{\alpha-2}$,
 such that it does not interfere with the outgoing energy. Precise choices are provided in \eqref{outs40.7}--\eqref{outs51.7},
\begin{equation*}
f_3(r):=\eps_4\chi_{\geq 8M}(r)e^{\beta(r)},\qquad g_3(r):=\int_{r}^\infty \Big[\rho(s)+\frac{\eps_4M^2}{s^3}f_3(s)\Big]\,ds,
\end{equation*}
where 
\begin{equation*}
\beta(8M):=0,\qquad \beta'(r):=\Big(\frac{4M}{r^2}+\frac{1}{r}\Big)\big(1-\chi_{\geq C_4^4M}(r)\big)+\frac{\alpha}{r}\chi_{\geq C_4^4M}(r),
\end{equation*}
and
\begin{equation*}
\rho(r):=\delta M^{-1}\Big[\chi_{\geq C_4M}(r)+\chi_{\geq 4C_4^4M}(r)\Big(C_4^7e^{\beta(r)}\frac{M^3}{r^3}-1\Big)\Big].
\end{equation*}
The constants $\eps_4,C_4$ satisfy $\eps_4=\eps_3^2$ and $C_4\geq \eps_4^{-4}\alpha^{-1}(2-\alpha)^{-1}$, while $\delta\in[10^{-4}C_4^{-3},10^4C_4^{-3}]$ is such that $\int_{C_4M}^\infty\rho(s)\,ds=C_4$.
 
The function $h_3$ can be chosen explicitly in terms of $f_3$ and $g_3$, in such a way to complete squares and create positive $0$'s order contributions. The positivity of the  bulk terms in \eqref{OU10} and \eqref{OU20}, together with the choice $\eps_4\ll\eps_3$, is used to show positivity of the total bulk contribution in the transition region. Overall, we derive the desired lower bound on the bulk term,
\begin{equation}\label{OU30}
\begin{split}
\sum_{j=1}^4\big(L^j_{(1)}&+L^j_{(2)}+L^j_{(4)}+L^j_{(3)}\big)\gtrsim_\alpha e^{\beta}\Big\{\frac{(r-r^\ast)^2}{r^2}\frac{(\partial_1\psi)^2+(\psi/\sin\theta)^2}{r^3}\\
&+\frac{M^2}{r^3}(\partial_2\psi)^2+\frac{M^2(r-r^\ast)^2}{r^5}(\partial_3\phi)^2+\frac{\psi^2}{r^3}+\frac{(L\psi)^2}{r}\Big\}.
\end{split}
\end{equation}

At the same time one can estimate precisely the size of the term $2\widetilde{P}_{(3)\mu}n^\mu$ at infinity, and use positivity of the function $g_3$ in the transition region to absorb the contributions of the other terms $2\widetilde{P}_{(j)\mu}n^\mu$, $j\in\{1,2,4\}$. Overall, we find that
\begin{equation}\label{OU31}
\begin{split}
\int_{\Sigma_{t}^c}2\big[\widetilde{P}_{(1)\mu}+\widetilde{P}_{(2)\mu}+\widetilde{P}_{(3)\mu}+\widetilde{P}_{(4)\mu} \big]n_0^\mu\,d\mu_{t}\approx_{\alpha} \int_{\Sigma_{t}^c}&e^\beta\big[e(\phi)^2+e(\psi)^2\big]\,d\mu_t.
\end{split}
\end{equation}
Finally we find that
\begin{equation}\label{OU32}
2\big[\widetilde{P}_{(1)\mu}+\widetilde{P}_{(2)\mu}+\widetilde{P}_{(3)\mu}+\widetilde{P}_{(4)\mu}\big]k^\mu\geq 0\qquad \text{ along }\mathcal{N}^c_{[0,T]}.
\end{equation}
The theorem follows from \eqref{OU30}--\eqref{OU32}, and the divergence identity \eqref{OU7}.

\section{The wave equation in the Schwarzschild spacetime}\label{MoraEst}

We show first how to prove Theorem \ref{MainTheorem} in the simplest case: $a=0$ (the Schwarzschild spacetime) and $\psi=0$. 
In this case $B=0$ and we are simply considering $\Z$-invariant solutions of the wave equation
\begin{equation*}
\square\phi=0.
\end{equation*}

In the rest of this section we use the coordinates $(\theta,r,u_+,\phi_+)$ and the induced vector-fields $\partial_1=\partial_\theta,\,\partial_2=\partial_r,\,\partial_3=\partial_t,\,\partial_4=\partial_{\phi}$, see \eqref{cha3}--\eqref{zq3}. For simplicity of notation, we identify functions that depend on $r$ (or on some of the other variables) with the corresponding functions defined on the spacetime. 

Notice that  
\begin{equation}\label{mora1}
q^2=r^2,\qquad p=\frac{2M}{r},
\end{equation}
with $p$ introduced  in \eqref{eq:useful-id}.
The nontrivial components of the metric are
\begin{equation}\label{mora2}
\begin{split}
&\g^{11}=r^{-2},\\
&\g^{22}=1-p,\\
&\g^{23}=p\chi,\\
&\g^{33}=\frac{-1+p^2\chi^2}{1-p},\\
&\g^{44}=\frac{1}{r^2(\sin\theta)^2}.
\end{split}
\end{equation}
Given a function $H$ that depends only on $r$, the formula \eqref{waveop} shows that
\begin{equation}\label{waveop2}
\square H=\g^{22}\partial_2^2H+D^2\partial_2H=\frac{r-2M}{r}\partial_2^2H+\frac{2r-2M}{r^2}\partial_2H.
\end{equation}
Similarly, if $m$ is a 1-form with $m_4=0,\mathcal{L}_\T m=0,\mathcal{L}_\Z m=0$, then
\begin{equation}\label{waveop3}
\begin{split}
&\D^\mu m_\mu=\g^{\al\be}\partial_\al m_\be+\big[\partial_\mu\g^{\mu\nu}+(1/2)\g^{\mu\nu}\partial_\mu\log|r^4(\sin\theta)^2|\big]m_\nu\\
&=\frac{1}{r^2}\Big[\partial_1 m_1+\frac{\cos\theta}{\sin\theta}m_1\Big]+\Big[(1-p)\partial_2m_2+\frac{2r-2M}{r^2}m_2\Big]+p\Big[\chi\partial_2m_3+(\chi'+\chi/r)m_3\Big].
\end{split}
\end{equation}

Therefore, given a vector-field 
\begin{equation}\label{mora3}
X=f(r)\partial_2+g(r)\partial_3,
\end{equation}
as in \eqref{kra7.1}, and a $1$-form $Y$ with $Y_4=0$, and letting
\begin{equation*}
{}^{(Y)}Q_{\mu\nu}=Y_\mu Y_\nu-(1/2)\g_{\mu\nu}(Y_\rho Y^\rho),
\end{equation*}
we have, see \eqref{kra7.4}--\eqref{Qy2},
\begin{equation}\label{mora4}
\begin{split}
{}^{(Y)}Q_{\mu\nu}{}^{(X)}\pi^{\mu\nu}&=(Y_1)^2\frac{-f'(r)}{r^2}+
(Y_2)^2\frac{-f(r)(2r-2M)+f'(r)(r^2-2Mr)}{r^2}\\
&+(Y_3)^2\Big[-f(r)\partial_2\g^{33}+2g'(r)\g^{23}-f'(r)\g^{33}-\frac{2rf(r)\g^{33}}{r^2}\Big]\\
&+2Y_2Y_3\frac{-2Mrf(r)\chi'(r)-2Mf(r)\chi(r)+g'(r)(r^2-2Mr)}{r^2},
\end{split}
\end{equation}
\begin{equation}\label{mora5}
\begin{split}
2{}^{(Y)}Q(n,X)&=(Y_1)^2\frac{g(r)}{r^2}+(Y_2)^2[g(r)(1-p)-2f(r)\g^{23}]\\
&+(Y_3)^2[-g(r)\g^{33}]+2Y_2Y_3[-f(r)\g^{33}],
\end{split}
\end{equation}
and
\begin{equation}\label{mora6}
\begin{split}
2{}^{(Y)}Q(k,X)&=(Y_1)^2\frac{-f(r)}{r^2}+(Y_2)^2[f(r)(1-p)]\\
&+(Y_3)^2[-f(r)\g^{33}+2g(r)\g^{23}]+2Y_2Y_3[g(r)(1-p)].
\end{split}
\end{equation}
Here $f'$ and $g'$ denote the derivatives with respect to $r$ of the functions $f$ and $g$, and
\begin{equation}\label{mora6.5}
n=-\g^{3\mu}\partial_\mu=-\g^{23}\partial_2-\g^{33}\partial_3,\qquad k=\g^{2\mu}\partial_\mu=(1-p)\partial_2+\g^{23}\partial_3.
\end{equation}

In this section we prove the following:

\begin{theorem}\label{MainTheorem3}
Assume that $M\in (0,\infty)$, $N_0=4$, $a=0$ and $c_0:=2M-\overline{\varep}M$, where $\overline{\varep}\in [0,1)$ is a sufficiently small constant. Assume that $T\geq 0$, and $\phi\in C^k([0,T]:H^{N_0-k}(\Sigma_t^{c_0}))$, $k\in[0,N_0]$, is a $\Z$-invariant real-valued solution of the wave equation
\begin{equation}\label{mora10}
\square\phi=0.
\end{equation}
Then, for any $\alpha\in(0,2)$ and any $t_1\leq t_2\in[0,T]$,
\begin{equation}\label{outs100}
\mathcal{E}^{c_0}_\alpha(t_2)+\mathcal{B}_\alpha^{c_0}(t_1,t_2)\leq \overline{C}_\alpha\mathcal{E}^{c_0}_\alpha(t_1),
\end{equation}
where $\overline{C}_\alpha$ is a large constant that may depend on $\alpha$, 
\begin{equation}\label{outs4}
E_\mu:=\D_\mu\phi,\qquad L\phi:=\chi_{\geq 4M}(r)\Big(\partial_2+\frac{1}{1-p}\partial_3\Big)\phi=\chi_{\geq 4M}(r)\Big(E_2+\frac{1}{1-p}E_3\Big),
\end{equation}
\begin{equation}\label{outs101}
\mathcal{E}^{c_0}_\alpha(t):=\int_{\Sigma_t^{c_0}}\frac{r^\alpha}{M^{\alpha}}\Big[(E_1/r)^2+(L\phi)^2+M^2r^{-2}\big[(E_2)^2+(E_3^2)\big]+r^{-2}\phi^2\Big]\,d\mu_t,
\end{equation}
\begin{equation}\label{outs102}
\begin{split}
\mathcal{B}_\alpha^{c_0}(t_1,t_2):=\int_{\mathcal{D}^{c_0}_{[t_1,t_2]}}\frac{r^\alpha}{M^{\alpha}}\Big\{&\frac{(r-3M)^2}{r^3}\frac{(E_1)^2}{r^2}+\frac{1}{r}(L\phi)^2\\
&+\frac{1}{r^3}\phi^2+\frac{M^{2}}{r^{3}}\Big[(E_2)^2+\frac{(r-3M)^2}{r^2}(E_3)^2\Big]\Big\}\,d\mu.
\end{split}
\end{equation}
\end{theorem}

The rest of the section is concerned with the proof of Theorem \ref{MainTheorem3}. Let
\begin{equation}\label{mora11}
Q_{\mu\nu}:=E_\mu E_\nu-(1/2)\g_{\mu\nu}(E_\rho E^\rho),\qquad J_\nu:=\D^\mu Q_{\mu\nu}=\mathcal{N}E_\nu.
\end{equation}
For any vector-field $X$, real scalar function $w$, and $1$-form $m$ we define
\begin{equation}\label{mora12}
P_\mu=P_\mu[X,w,m]:=Q_{\mu\nu}X^\nu+\frac{1}{2}w\phi E_\mu-\frac{1}{4}\phi^2\D_\mu w+\frac{1}{4}m_\mu\phi^2.
\end{equation} 
The formula \eqref{zq52.5} becomes
\begin{equation}\label{mora13}
2\D^\mu P_\mu=\mathcal{T}[X,w,m]:={}^{(X)}\pi^{\mu\nu}Q_{\mu\nu}
+wE^\mu E_\mu+\phi m^\mu E_\mu+\frac{1}{2}\phi^2\big(\D^\mu m_\mu-\square w\big).
\end{equation}

We use the divergence identity
\begin{equation}\label{mora14}
\int_{\Sigma_{t_1}^c}P_\mu n_0^\mu\,d\mu_{t_1}=\int_{\Sigma_{t_2}^c}P_\mu n_0^\mu\,d\mu_{t_2}+\int_{\mathcal{N}^c_{[t_1,t_2]}}P_\mu k_0^\mu\,d\mu_c+\int_{\mathcal{D}^c_{[t_1,t_2]}}\D^\mu P_\mu\,d\mu,
\end{equation}
where $t_1,t_2\in[0,T]$, $c\in(c_0,2M]$, $n_0:=n/|\g^{33}|^{1/2}$, $k_0:=k/|\g^{22}|^{1/2}$, and the integration is with respect to the natural measures induced by the metric $\g$. We would like to find multipliers $(X,w,m)$ in such a way that the contributions of the integrals in \eqref{mora14} are all nonnegative.

\subsection{The multipliers $(X_{(k)},w_{(k)},m_{(k)})$, $k\in\{1,2\}$}\label{multipliers} In this subsection we define three multipliers $(X_{(k)},w_{(k)},m_{(k)})$, $k\in\{1,2,3\}$, which are used to generate positive terms in the divergence identity \eqref{mora14}. The first multiplier $(X_{(1)},w_{(1)},m_{(1)})$ is relevant in a neighborhood of the trapped set $\{r=3M\}$ and the second multiplier $(X_{(2)},w_{(2)},m_{(2)})$ is relevant in a neighborhood of the horizon $\{r=2M\}$. The third multiplier $(X_{(3)},w_{(3)},m_{(3)})$ generates outgoing energies at infinity; at the same time it contains a large multiple of the vector-field $\partial_3$ which helps with the positivity of the boundary integrals $P_\mu n^\mu$.

\subsubsection{Analysis around the trapped set $r=3M$} This is similar to the construction in \cite{MaMeTaTo}. We define the first multiplier $(X_{(1)},w_{(1)},m_{(1)})$ by the formulas
\begin{equation}\label{moro6}
\begin{split}
&X_{(1)}:=f_1(r)\partial_2+g_1(r)\partial_3,\qquad f_1(r):=\frac{a_1(r)\Delta}{r^2},\qquad g_1(r):=\frac{a_1(r)\chi(r)2M}{r}+1,\\
&w_{(1)}(r):=f'_1(r)+f_1(r)\partial_r\log\big(r^4/\Delta)-\eps_1\widetilde{w}(r),\qquad m_{(1)}\equiv 0,\\
&\widetilde{w}(r):=M^2(r-33M/16)^3(r-3M)^2r^{-8}\mathbf{1}_{[33M/16,\infty)}(r),
\end{split}
\end{equation}
where $a_1:(0,\infty)\to\mathbb{R}$ is a smooth increasing function to be fixed, $\lim_{r\to\infty}a_1(r)=1$, and $\eps_1\in(0,1]$ is a small constant. Using \eqref{mora4},
\begin{equation*}
Q_{\mu\nu}{}^{(X_{(1)})}\pi^{\mu\nu}+w_{(1)}E_\mu E^\mu=\big[K_{(1)}^{11}(E_1)^2+
K_{(1)}^{22}(E_2)^2+K_{(1)}^{33}(E_3)^2+2K_{(1)}^{23}E_2E_3\big],
\end{equation*}
where
\begin{equation*}
\begin{split}
K_{(1)}^{11}&=\frac{-f'_1(r)}{r^2}+w_{(1)}(r)\g^{11}=\frac{2(r-3M)}{r^4}a_1-\eps_1\widetilde{w}\g^{11},\\
K_{(1)}^{22}&=\frac{-f_1(r)(2r-2M)+f'_1(r)\Delta}{r^2}+w_{(1)}(r)\g^{22}=\frac{2\Delta^2}{r^4}a'_1-\eps_1\widetilde{w}\g^{22},\\
K_{(1)}^{33}&=-f_1(r)\partial_2\g^{33}+2g'_1(r)\g^{23}-f'_1(r)\g^{33}-\frac{2rf_1(r)\g^{33}}{r^2}+w_{(1)}(r)\g^{33}=\frac{8M^2\chi^2}{r^2}a'_1-\eps_1\widetilde{w}\g^{33},\\
K_{(1)}^{23}&=\frac{-2Mrf_1(r)\chi'(r)-2Mf_1(r)\chi(r)+g'_1(r)\Delta}{r^2}+w_{(1)}(r)\g^{23}=\frac{4M\Delta\chi}{r^3}a'_1-\eps_1\widetilde{w}\g^{23}.
\end{split}
\end{equation*}
where $a'_1$ denotes the $r$ derivative of the function $a_1$. Therefore
\begin{equation}\label{moro10}
\begin{split}
Q_{\mu\nu}{}^{(X_{(1)})}\pi^{\mu\nu}+w_{(1)}E_\mu E^\mu&=\frac{2(r-3M)a_1-\eps_1\widetilde{w}r^2}{r^4}(E_1)^2\\
&+\Big(2a'_1-\frac{\eps_1\widetilde{w}}{1-p}\Big)\Big(\frac{\Delta}{r^2}E_2+\frac{2M\chi}{r}E_3\Big)^2+\frac{\eps_1\widetilde{w}}{1-p}(E_3)^2.
\end{split}
\end{equation}
Moreover
\begin{equation}\label{moro11}
\phi m^\mu_{(1)} E_\mu+\frac{1}{2}\phi^2\big(\D^\mu m_{(1)\mu}-\square w_{(1)}\big)=-\frac{1}{2}\square w_{(1)}\phi^2.
\end{equation}

We define now the important function $a_1(r)$. Assume $\kappa:\mathbb{R}\to\mathbb{R}$ is an increasing smooth function satisfying $\kappa(y)=y$ on $[-1,\infty)$ and $\kappa(y)=-2$ on $(-\infty,-3]$. We set
\begin{equation}\label{moro35}
\begin{split}
&R(r):=(r-3M)(r+2M)+6M^2\log\Big[\frac{r-2M}{M}\Big],\\
&a_1(r):=r^{-2}\delta^{-1}\kappa(\delta R(r))+\Big[\frac{M}{r}-\frac{6M^2}{r^2}\log\Big(\frac{r-2M}{M}\Big)\Big]\chi_{\geq DM}(r),
\end{split}
\end{equation}
where $\delta:=\eps_2^2M^{-2}$ is a small constant and $D\gg 1$ is a large constant. The function $a_1$ is well defined, using the formula above, for $r>2M$. Clearly $a_1(r)=-2r^{-2}\delta^{-1}$ for $r$ sufficiently close to $2M$. Therefore $a_1$ can be extended smoothly by this formula to the full interval $r\in(c_0,\infty)$.

Clearly
\begin{equation}\label{moro35.15}
R'(r)=2r-M+\frac{6M^2}{r-2M}.
\end{equation}
The function $R$ is increasing on $(2M,\infty)$. Let $r_\delta$ denote the unique number in $(2M,\infty)$ with the property that $R(r_\delta)=-1/\delta$, and notice that
\begin{equation*}
\frac{r_\delta-2M}{M}\approx e^{-(6\delta M^2)^{-1}}.
\end{equation*}

Clearly $a_1(3M)=0$,
\begin{equation}\label{moro35.1}
a'_1(r)=r^{-2}\Big[R'(r)\kappa'(\delta R(r))-\frac{2\kappa(\delta R(r))}{\delta r}\Big]
\end{equation}
if $r\leq DM$, and 
\begin{equation}\label{moro35.10}
\begin{split}
a'_1(r)&=\frac{12M^2}{r^3}+\Big[\frac{M}{r}-\frac{6M^2}{r^2}\log\Big(\frac{r-2M}{M}\Big)\Big]\chi'_{\geq DM}(r)\\
&+\Big[\frac{M}{r^2}-\frac{12M^2}{r^3}\log\Big(\frac{r-2M}{M}\Big)+\frac{6M^2}{r^2(r-2M)}\Big](1-\chi_{\geq DM}(r))
\end{split}
\end{equation}
if $r\geq r_\delta$. 
In view of \eqref{moro35.1}, if $r\in (c_0,r_\delta]$ then $a'_1(r)\geq 2\delta^{-1}r^{-3}$. On the other hand, if $r\in [r_\delta,\infty)$ then $a'_1(r)\geq 12M^2r^{-3}$.
Therefore
\begin{equation}\label{moro35.11}
a_1(3M)=0\qquad\text{ and }\qquad a'_1(r)\geq 12 M^2r^{-3}\qquad\text{ for }r\in (c_0,\infty),
\end{equation}
provided that $\delta\leq (10M)^{-2}$.

Let
\begin{equation}\label{moro35.12}
h_1(r):=f'_1(r)+f_1(r)\partial_r\log\big(r^4/\Delta)=\frac{r-2M}{r^3}\partial_r\big(r^2a_1(r)\big).
\end{equation}
We calculate, as before,
\begin{equation}\label{moro37.1}
h_1(r)=\frac{r-2M}{r^3}R'(r)\kappa'(\delta R(r))
\end{equation}
if $r\leq DM$, and 
\begin{equation}\label{moro37.10}
\begin{split}
h_1(r)&=\frac{r-2M}{r^3}\Big\{2r-\Big[M-\frac{6M^2}{r-2M}\Big](1-\chi_{\geq DM}(r))+\Big[Mr-6M^2\log\Big(\frac{r-2M}{M}\Big)\Big]\chi'_{\geq DM}(r)\Big\}
\end{split}
\end{equation}
if $r\geq r_\delta$. Letting 
\begin{equation*}
\widetilde{R}(r):=\frac{r-2M}{r^3} R'(r)=\frac{2}{r}-\frac{5M}{r^2}+\frac{8M^2}{r^3},
\end{equation*}
we have
\begin{equation*}
\begin{split}
(\square h_1)(r)&=\frac{\partial_2(\Delta\cdot\partial_2h_1)}{r^2}\\
&=r^{-2}\Big\{\kappa'(\delta R(r))\partial_r[\Delta\widetilde{R}'(r)]+\delta^2\kappa'''(\delta R(r))r^7\widetilde{R}(r)^3(r-2M)^{-1}\\
&+\delta\kappa''(\delta R(r))[3r^4\widetilde{R}(r)\widetilde{R}'(r)+4r^3\widetilde{R}(r)^2]\Big\}
\end{split}
\end{equation*}
if $r\leq DM$, and
\begin{equation*}
\begin{split}
(\square h_1)(r)&=\frac{\partial_2(\Delta\cdot\partial_2h_1)}{r^2}\\
&=r^{-2}\partial_r[\Delta\widetilde{R}'(r)]+O(Mr^{-4})\mathbf{1}_{[DM,\infty)}(r)\\
&=-\frac{2M}{r^4}\Big(7-\frac{44M}{r}+\frac{72M^2}{r^2}\Big)+O(Mr^{-4})\mathbf{1}_{[DM,\infty)}(r).
\end{split}
\end{equation*}
if $r\geq r_\delta$. Therefore, the last two identities show that
\begin{equation}\label{moro31.7}
\begin{split}
(\square h_1)(r)&=-\frac{2M}{r^4}\Big(7-\frac{44M}{r}+\frac{72M^2}{r^2}\Big)+O(Mr^{-4})\mathbf{1}_{[DM,\infty)}(r)\\
&+M^{-3}O(1)\mathbf{1}_{(c_0,r_\delta]}(r)+O\Big(\frac{\delta^2M^2}{r-2M}\Big)\mathbf{1}_{[r'_\delta,r_\delta]}(r),
\end{split}
\end{equation}
where $r'_\delta$ denotes the unique number in $(2M,\infty)$ with the property that $R(r'_\delta)=-2/\delta$. Notice that
\begin{equation}\label{moro30.2}
7-\frac{44M}{r}+\frac{72 M^2}{r^2}\geq 1/10\qquad\text{ for any }r\geq M.
\end{equation}
Therefore, since $w_{(1)}=h_1-\eps_1\widetilde{w}$, it follows that
\begin{equation}\label{moro30.3}
-\frac{1}{2}(\square w_{(1)})(r)\geq \frac{M}{10r^4}-\frac{C_1M}{r^4}\mathbf{1}_{[DM,\infty)}(r)-\frac{C_1}{M^3}\mathbf{1}_{(c_0,r_\delta]}(r)-\frac{C_1\delta^2M^2}{r-2M}\mathbf{1}_{[r'_\delta,r_\delta]}(r),
\end{equation}
for a sufficiently large constant $C_1$, provided that the constant $\eps_1$ is sufficiently small. Using also \eqref{moro10}--\eqref{moro11} and \eqref{moro35.11},
\begin{equation}\label{moro31}
\begin{split}
\mathcal{T}[X_{(1)},w_{(1)},m_{(1)}]&\geq \frac{(2-C_1\eps_1)(r-3M)a_1(r)}{r^4}(E_1)^2\\
&+(2-C_1\eps_1)a'_1(r)\big((1-p)E_2+p\chi(r)E_3\big)^2+\eps_1\widetilde{w}(r)(E_3)^2+\frac{M}{10r^4}\phi^2\\
&-\frac{C_1M}{r^4}\mathbf{1}_{[DM,\infty)}(r)\phi^2-\frac{C_1}{M^3}\mathbf{1}_{(c_0,r_\delta]}(r)\phi^2-\frac{C_1\delta^2M^2}{r-2M}\mathbf{1}_{[r'_\delta,r_\delta]}(r)\phi^2,
\end{split}
\end{equation}
for a sufficiently large constant $C_1$, provided that the constant $\eps_1$ is sufficiently small.

The remaining contributions $2P_\mu n_0^\mu$ and $2P_\mu k_0^\mu$ in the divergence identity \eqref{mora14} cannot be 
estimated effectively at this time. We will prove partial estimates for these terms in Lemma \ref{flux} below, after we construct 
the second multiplier $(X_{(2)},w_{(2)},m_{(2)})$ and show how to fix some of the parameters.

\subsubsection{Analysis in a neighborhood of the horizon} In a small neighborhood of the horizon we need to use the redshift effect. For this we define the second multiplier $(X_{(2)},w_{(2)},m_{(2)})$ by the formulas
\begin{equation}\label{moro12}
\begin{split}
&X_{(2)}:=f_2(r)\partial_2+g_2(r)\partial_3,\qquad f_2(r):=-\eps_2a_2(r),\qquad g_2(r):=\eps_2 a_2(r)(1-\eps_2),\\
&w_{(2)}(r):=-2\eps_2 a_2(r)/r,\qquad m_{(2)2}:=\eps_2M^{-2}\gamma(r),\qquad m_{(2)3}:=\eps_2M^{-2}\gamma(r)
\end{split}
\end{equation}
where $\eps_2$ is a small positive constant (recall that $\delta=\eps_2^2M^{-2}$),
\begin{equation}\label{moro12.5}
a_2(r):=
\begin{cases}
M^{-3}(9M/4-r)^3\qquad&\text{ if }r\leq 9M/4,\\
0\qquad&\text{ if }r\geq 9M/4,
\end{cases}
\end{equation}
and $\gamma:[c_0,\infty)\to[0,1]$ is a suitable function (to be fixed later) supported in $[c_0,17M/8]$ and satisfying $\gamma(2M)=1/2$.

Notice that $\chi=1$ in the support of the functions $a_2$ and $\gamma$. As before, we calculate
\begin{equation*}
Q_{\mu\nu}{}^{(X_{(2)})}\pi^{\mu\nu}+w_{(2)}E_\mu E^\mu=\big[K_{(2)}^{11}(E_1)^2+
K_{(2)}^{22}(E_2)^2+K_{(2)}^{33}(E_3)^2+2K_{(2)}^{23}E_2E_3\big],
\end{equation*}
where
\begin{equation*}
\begin{split}
K_{(2)}^{11}&=\frac{-f'_2(r)}{r^2}+w_{(2)}(r)\g^{11}=\eps_2\frac{ra'_2-2a_2}{r^3},\\
K_{(2)}^{22}&=\frac{-f_2(r)(2r-2M)+f'_2(r)\Delta}{r^2}+w_2(r)\g^{22}=\eps_2\Big[-\frac{r-2M}{r}a'_2+\frac{2M}{r^2}a_2\Big],\\
K_{(2)}^{33}&=-f_2(r)\partial_2\g^{33}+2g'_2(r)\g^{23}-f'_2(r)\g^{33}-\frac{2rf_2(r)\g^{33}}{r^2}+w_2(r)\g^{33}\\
&=\eps_2\Big[-\frac{r-2M+4\eps_2M}{r}a'_2+\frac{2M}{r^2}a_2\Big],\\
K_{(2)}^{23}&=\frac{-2Mrf_2(r)\chi'(r)-2Mf_2(r)\chi(r)+g'_2(r)\Delta}{r^2}+w_2(r)\g^{23}\\
&=-\eps_2\Big[-\frac{(1-\eps_2)(r-2M)}{r}a'_2(r)+\frac{2M}{r^2}a_2(r)\Big].
\end{split}
\end{equation*}
Using the explicit formula \eqref{moro12.5}, it is easy to see that 
\begin{equation}\label{moro13}
\begin{split}
&Q_{\mu\nu}{}^{(X_{(2)})}\pi^{\mu\nu}+w_{(2)}E_\mu E^\mu\\
&\geq\mathbf{1}_{(c_0,9M/4)}(r)(9M/4-r)^2M^{-3}\Big[C_2^{-1}\eps_2(E_2-E_3)^2+C_2^{-1}\eps_2^2(E_3)^2-C_2\eps_2(E_1)^2/r^2\Big],
\end{split}
\end{equation}
for a sufficiently large constant $C_2$, provided that $\eps_2$ is sufficiently small and $c_0$ is sufficiently close to $2M$. Moreover, using the definitions and \eqref{waveop2}--\eqref{waveop3},
\begin{equation*}
\phi m_{(2)}^\mu E_\mu+\frac{1}{2}\phi^2\big(\D^\mu m_{(2)\mu}-\square w_{(2)}\big)=\frac{\eps_2\gamma}{M^2}\phi(E_2-E_3)+\frac{\eps_2}{2}\phi^2\Big(\frac{1}{M^2}\gamma'+\frac{2}{rM^2}\gamma+2\square(a_2/r)\Big).
\end{equation*}
Therefore, recalling also that $\gamma\in[0,1]$ and completing the square, 
\begin{equation}\label{moro14}
\begin{split}
\mathcal{T}[X_{(2)},w_{(2)},m_{(2)}]&\geq\frac{\eps_2}{2M^2}\phi^2\gamma'+M^{-1}\eps_2^4\mathbf{1}_{(c_0,17M/8)}(r)\big[(E_2)^2+(E_3)^2\big]\\
&-C_2\eps_2\mathbf{1}_{(c_0,9M/4)}(r)\big[M^{-1}(E_1)^2/r^2+M^{-3}\phi^2\big],
\end{split}
\end{equation}
provided that $\eps_2$ is sufficiently small and $c_0$ is sufficiently close to $2M$.

We examine now \eqref{moro31} and \eqref{moro14} and fix the constant $\eps_1,\eps_2$ and the function $\gamma$ such that the sum $\mathcal{T}[X_{(1)},w_{(1)},m_{(1)}]+\mathcal{T}[X_{(2)},w_{(2)},m_{(2)}]$ is nonnegative when $r\in(c_0,DM]$. For the positivity of the zero order term we need that
\begin{equation}\label{other2}
\frac{M}{20r^4}+\frac{\eps_2\gamma'(r)}{2M^2}\geq \frac{C_1}{M^3}\mathbf{1}_{(c_0,r_\delta]}(r)+\frac{C_1\delta^2M^2}{r-2M}\mathbf{1}_{[r'_\delta,r_\delta]}(r)+\frac{C_2\eps_2}{M^3}\mathbf{1}_{(c_0,9M/4)}(r).
\end{equation}
Recall that $\delta=\eps_2^2M^{-2}$. The point is that
\begin{equation*}
\int_{c_0}^\infty\frac{C_1}{M^3}\mathbf{1}_{(c_0,r_\delta]}(r)+\frac{C_1\delta^2M^2}{r-2M}\mathbf{1}_{[r'_\delta,r_\delta]}(r)\,dr\leq \frac{C_1^2\eps_2^2}{M^2},
\end{equation*}
provided that $2M-c_0\leq\eps_2^2$. This is easy to see if one recalls the definitions of $r_\delta$ and $r'_\delta$. Therefore, assuming that $\eps_2$ is sufficiently small, one can fix the function $\gamma$ to achieve the inequality \eqref{other2}, while still preserving the other properties of $\gamma$, namely 
\begin{equation}\label{other3}
\gamma:[c_0,\infty)\to[0,1]\text{ is supported in }[c_0,17M/8]\text{ and satisfies }\gamma(2M)=1/2.
\end{equation}
Indeed, the function $\gamma$ can be chosen to increase on the interval $(c_0,r_\delta]$ and then decrease for $r\geq 2r_\delta-2M$ in a way to satisfy both \eqref{other2} and \eqref{other3}.

Notice that the sum of the first order terms in $\mathcal{T}[X_{(1)},w_{(1)},m_{(1)}]+\mathcal{T}[X_{(2)},w_{(2)},m_{(2)}]$ is nonnegative and nondegenerate if we simply have $\eps_1,\eps_2>0$ sufficiently small. Therefore, one can fix the parameters $\eps_1,\eps_2$ and the function $\gamma$ in such a way that
\begin{equation}\label{other4}
\begin{split}
\mathcal{T}[X_{(1)},w_{(1)},&m_{(1)}]+\mathcal{T}[X_{(2)},w_{(2)},m_{(2)}]\\
&\geq \eps_3\Big[\frac{(r-3M)^2}{r^3}(E_1/r)^2+\frac{M^2}{r^3}(E_2)^2+\frac{M^2(r-3M)^2}{r^5}(E_3)^2+\frac{M}{r^4}\phi^2\Big]\\
&-\eps_3^{-1}\frac{M}{r^4}\mathbf{1}_{[DM,\infty)}(r)\phi^2,
\end{split}
\end{equation}
for a constant $\eps_3>0$ sufficiently small (relative to $\eps_1$ and $\eps_2$). The parameter $D$ will be fixed later, sufficiently large depending on $\eps_3$.

We can prove now some partial bounds on the remaining terms
\begin{equation*}
2(P_{(1)\mu}+P_{(2)\mu})n_0^\mu,\qquad 2(P_{(1)\mu}+P_{(2)\mu}) k_0^\mu,
\end{equation*}
in the divergence identity \eqref{mora14}, where $P_{(k)}:=P[X_{(k)},w_{(k)},m_{(k)}]$, $k\in\{1,2\}$.

\begin{lemma}\label{flux}
There is a sufficiently small absolute constant $\eps_3$ such that
\begin{equation}\label{other5}
2(P_{(1)\mu}+P_{(2)\mu}) k^\mu\geq\eps_3\big[(E_1/r)^2+(E_2)^2(2-c/M)+M^{-2}\phi^2\big]-\eps_3^{-1}(E_3)^2.
\end{equation}
along $\mathcal{N}^c_{[t_1,t_2]}$. Also
\begin{equation}\label{other15}
\begin{split}
2(P_{(1)\mu}+P_{(2)\mu}) n^\mu&\geq -\eps_3^{-1}\big[\widetilde{F}_0+\mathbf{1}_{[8M,2DM]}(r)(E_3)^2\big]\\
&-\frac{\chi_{\geq 8M}(r)(1-p)}{r^2}\partial_2(r\phi^2)+\eps_3(E_2)^2\mathbf{1}_{(c_0,17M/8]}(r),
\end{split}
\end{equation}
and
\begin{equation}\label{other15.2}
\begin{split}
2(P_{(1)\mu}+P_{(2)\mu}) n^\mu&\leq \eps_3^{-1}\big[\widetilde{F}_0+\mathbf{1}_{[8M,2DM]}(r)(E_3)^2+(E_2)^2\mathbf{1}_{(c_0,17M/8]}(r)\big]\\
&-\frac{\chi_{\geq 8M}(r)(1-p)}{r^2}\partial_2(r\phi^2),
\end{split}
\end{equation}
where
\begin{equation}\label{other9}
\widetilde{F}_0=(E_1/r)^2+(L\phi)^2+M^2r^{-2}\big[(E_2)^2|1-p|+(E_3^2)\big]+r^{-2}\phi^2.
\end{equation}
\end{lemma} 

\begin{proof} We start with the term $2(P_{(1)\mu}+P_{(2)\mu}) k^\mu$,
\begin{equation*}
\begin{split}
2(P_{(1)\mu}+P_{(2)\mu}) k^\mu&=2k^\mu Q_{\mu\nu}\big(X^\nu_{(1)}+X^\nu_{(2)}\big)+(w_{(1)}+w_{(2)})\phi E_\mu k^\mu\\
&-\frac{1}{2}\phi^2k^\mu(\D_\mu w_{(1)}+\D_\mu w_{(2)})+\frac{1}{2}k^\mu m_{(2)\mu}\phi^2.
\end{split}
\end{equation*}
When $r=c\in(c_0,2M]$ and assuming that $2M-c$ is sufficiently small, we use the definitions, the identity $m_{(2)3}(2M)\geq \eps_2/(2M^2)$, and the identities \eqref{mora6}. We have
\begin{equation*}
2k^\mu Q_{\mu\nu}\big(X^\nu_{(1)}+X^\nu_{(2)}\big)\geq \eps_3\big[(E_1/r)^2+(E_2)^2(2-c/M)\big]-\eps_3^{-1}(E_3)^2-\eps_3^{-1}|E_2E_3|(2-c/M),
\end{equation*} 
\begin{equation*}
\big|(w_{(1)}+w_{(2)})\phi E_\mu k^\mu\big|\leq\eps_3^{-1}M^{-1}|\phi|\big[(2-c/M)|E_2|+|E_3|\big],
\end{equation*} 
and
\begin{equation*}
-\frac{1}{2}\phi^2k^\mu(\D_\mu w_{(1)}+\D_\mu w_{(2)})+\frac{1}{2}k^\mu m_{(2)\mu}\phi^2\geq \eps_3M^{-2}\phi^2,
\end{equation*} 
provided that $\eps_3$ is sufficiently small. The bound \eqref{other5} follows by further reducing $\eps_3$ and assuming that $2M-c$ is sufficiently small.

We consider now the term $2(P_{(1)\mu}+P_{(2)\mu}) n^\mu$,
\begin{equation*}
\begin{split}
2(P_{(1)\mu}+P_{(2)\mu}) n^\mu&=2n^\mu Q_{\mu\nu}\big(X^\nu_{(1)}+X^\nu_{(2)}\big)+(w_{(1)}+w_{(2)})\phi E_\mu n^\mu\\
&-\frac{1}{2}\phi^2n^\mu(\D_\mu w_{(1)}+\D_\mu w_{(2)})+\frac{1}{2}n^\mu m_{(2)\mu}\phi^2.
\end{split}
\end{equation*}
Using the definitions and the identities \eqref{mora5} we estimate
\begin{equation}\label{other8}
\big|\phi^2n^\mu(\D_\mu w_{(1)}+\D_\mu w_{(2)})\big|+\big|n^\mu m_{(2)\mu}\phi^2\big|\leq\eps_3^{-1}M^2r^{-4}\phi^2.
\end{equation}
Moreover, with $\widetilde{F}_0$ as in \eqref{other9}, we write
\begin{equation*}
\begin{split}
2n^\mu &Q_{\mu\nu}X^\nu_{(1)}+w_{(1)}\phi E_\mu n^\mu=(E_1)^2\frac{g_1(r)}{r^2}+(E_2)^2[g_1(r)(1-p)-2f_1(r)\g^{23}]\\
&+(E_3)^2[-g_1(r)\g^{33}]+2E_2E_3[-f_1(r)\g^{33}]+w_1\phi(-\g^{33}E_3-\g^{23}E_2)\\
&\geq -(10\eps_3)^{-1}\widetilde{F}_0+\big[(E_2)^2(1-p)+(E_3)^2(1-p)^{-1}+2E_2E_3a_1+\phi E_3 w_1(1-p)^{-1}\big]\chi_{\geq 8M}(r).
\end{split}
\end{equation*}
Using the definitions and the formula \eqref{moro37.10},
\begin{equation*}
\begin{split}
|a_1-1|\chi_{\geq 8M}(r)&\lesssim Mr^{-1}\mathbf{1}_{[8M,2DM]}(r)+M^2r^{-2}\mathbf{1}_{[8M,\infty)}(r),\\
|w_1-2(1-p)/r|\chi_{\geq 8M}(r)&\lesssim Mr^{-2}\mathbf{1}_{[8M,2DM]}(r)+M^2r^{-3}\mathbf{1}_{[8M,\infty)}(r).
\end{split}
\end{equation*}
Therefore
\begin{equation*}
\begin{split}
2n^\mu Q_{\mu\nu}X^\nu_{(1)}+w_{(1)}\phi E_\mu n^\mu&\geq -(9\eps_3)^{-1}\big[\widetilde{F}_0+\mathbf{1}_{[8M,2DM]}(r)(E_3)^2\big]-\Big[\frac{\phi^2}{r^2}+\frac{2\phi}{r}E_2\Big]\chi_{\geq 8M}(r)(1-p)\\
&\geq -(8\eps_3)^{-1}\big[\widetilde{F}_0+\mathbf{1}_{[8M,2DM]}(r)(E_3)^2\big]-\frac{\chi_{\geq 8M}(r)(1-p)}{r^2}\partial_2(r\phi^2).
\end{split}
\end{equation*}

Similarly, using also the observation that $-f_2(2M)\gtrsim 1$,
\begin{equation*}
2n^\mu Q_{\mu\nu}X^\nu_{(2)}+w_{(2)}\phi E_\mu n^\mu\geq -(8\eps_3)^{-1}\widetilde{F}_0+\eps_3(E_2)^2\mathbf{1}_{(c_0,17M/8]}(r).
\end{equation*}
The bound \eqref{other15} follows using the last two inequalities and \eqref{other8}. 

The proof of the upper bound \eqref{other15.2} follows in a similar way.
\end{proof}

\begin{remark}\label{OtherEnergy}
At this point one can recover the energy estimate of Marzuola--Metcalfe--Tataru--Tohaneanu \cite[Theorem 1.2]{MaMeTaTo},
\begin{equation*}
\mathcal{E}^{c_0}(t_2)+\mathcal{B}^{c_0}(t_1,t_2)\leq \overline{C}\mathcal{E}^{c_0}(t_1),
\end{equation*}
where 
\begin{equation*}
\mathcal{E}^{c_0}(t):=\int_{\Sigma_t^{c_0}}\Big[(E_1/r)^2+(E_2)^2+(E_3^2)\Big]\,d\mu_t,
\end{equation*}
\begin{equation*}
\begin{split}
\mathcal{B}^{c_0}(t_1,t_2):=\int_{\mathcal{D}^{c_0}_{[t_1,t_2]}}\Big[\frac{(r-3M)^2}{r^3}(E_1/r)^2+\frac{M^2}{r^3}(E_2)^2+\frac{M^2(r-3M)^2}{r^5}(E_3)^2+\frac{M}{r^4}\phi^2\Big]\,d\mu.
\end{split}
\end{equation*}
To see this, we simply set $D:=\infty$ and add in a very large multiple of the Killing vector-field $\partial_3$. The spacetime integral $\mathcal{B}^{c_0}(t_1,t_2)$ is generated by the right-hand side of \eqref{other4} (some of the powers of $r$ in the spacetime integral could in fact be improved by reexamining the proof).  The formulas for the nondegenerate energies $\mathcal{E}^{c_0}(t)$ follow from the bounds \eqref{other15} and \eqref{other15.2}, the identity \eqref{mora5}, and the Hardy inequality in Lemma \ref{GeneralLemma} (i). The contribution of $P_\mu k^\mu$ along $\mathcal{N}^c_{[t_1,t_2]}$ becomes nonnegative, in view of \eqref{mora6}, and can be neglected.
\end{remark}

\subsection{Outgoing energies}\label{outgoingEn} To prove the stronger estimates in Theorem \ref{MainTheorem3} we consider now a multiplier $(X_{(3)},w_{(3)},m_{(3)})$ of the form 
\begin{equation}\label{outs}
\begin{split}
&X_{(3)}=f_3\partial_2+\Big(\frac{f_3}{1-p}+g_3\Big)\partial_3,\qquad w_{(3)}=\frac{2f_3}{r},\\
&m_{(3)1}=m_{(3)4}=0,\qquad m_{(3)2}=\frac{2h_3}{r(1-p)},\qquad m_{(3)3}=-\frac{2h_3}{r},
\end{split}
\end{equation}
where $f_3,g'_3,h_3$ are smooth functions supported in $\{r\geq 8M\}$, which depend only on $r$. The function $g_3$ is not supported in $\{r\geq 8M\}$, it is in fact a very large constant in the region $r\in[c,10M]$.

As before, using \eqref{mora4}, we calculate
\begin{equation*}
Q_{\mu\nu}{}^{(X_{(3)})}\pi^{\mu\nu}+w_{(3)}E_\mu E^\mu=\big[K^{11}_{(3)}(E_1)^2+
K^{22}_{(3)}(E_2)^2+K^{33}_{(3)}(E_3)^2+2K^{23}_{(3)}E_2E_3\big],
\end{equation*}
where
\begin{equation*}
\begin{split}
&K^{11}_{(3)}=\frac{-f_3'(r)}{r^2}+w_{(3)}(r)\g^{11}=\frac{2f_3-rf'_3}{r^3},\\
&K^{22}_{(3)}=\frac{-f_3(r)(2r-2M)+f'_3(r)\Delta}{r^2}+w_{(3)}(r)\g^{22}=(1-p)f'_3-\frac{2Mf_3}{r^2},\\
&K^{33}_{(3)}=-f_3(r)\partial_2\g^{33}-f'_3(r)\g^{33}-\frac{2f_3(r)\g^{33}}{r}+w_{(3)}(r)\g^{33}=\frac{f'_3}{1-p}-\frac{2Mf_3}{r^2(1-p)^2},\\
&K^{23}_{(3)}=(1-p)\Big(\frac{f_3}{1-p}+g_3\Big)'=f'_3-\frac{2Mf_3}{r^2(1-p)}+(1-p)g'_3.
\end{split}
\end{equation*}
Moreover
\begin{equation*}
\begin{split}
\phi m_{(3)}^\mu E_\mu&+\frac{1}{2}\phi^2\big(\D^\mu m_{(3)\mu}-\square w_{(3)}\big)\\
&=2h_3\frac{\phi}{r}\Big(E_2+\frac{E_3}{1-p}\Big)+\phi^2\Big[\frac{h'_3}{r}+\frac{h_3}{r^2}-\frac{(1-p)f''_3}{r}-\frac{2Mf'_3}{r^3}+\frac{2Mf_3}{r^4}\Big].
\end{split}
\end{equation*}
Set
\begin{equation}\label{outs3}
\begin{split}
&H_3:=(1-p)f'_3-\frac{2Mf_3}{r^2}+(1-p)^2g'_3,\\
&h_3:=H_3\cdot (1-\widetilde{\alpha}),
\end{split}
\end{equation}
where $\widetilde{\alpha}=(2-\alpha)/10>0$. The identities above show that
\begin{equation*}
\begin{split}
\mathcal{T}[X_{(3)},w_{(3)},m_{(3)}]&=\frac{(E_1)^2}{r^2}\frac{2f_3-rf'_3}{r}+H_3\Big(E_2+\frac{E_3}{1-p}\Big)^2-(1-p)^2g'_3\Big[(E_2)^2+\frac{(E_3)^2}{(1-p)^2}\Big]\\
&+2h_3\frac{\phi}{r}\Big(E_2+\frac{E_3}{1-p}\Big)+\phi^2\Big[\frac{h_3}{r^2}+\frac{h'_3}{r}-\frac{(1-p)f''_3}{r}-\frac{2Mf'_3}{r^3}+\frac{2Mf_3}{r^4}\Big].
\end{split}
\end{equation*}
After completing the square this becomes
\begin{equation}\label{outs7}
\begin{split}
\mathcal{T}[X_{(3)},&w_{(3)},m_{(3)}]=\frac{(E_1)^2}{r^2}\frac{2f_3-rf'_3}{r}+H_3\Big(L\phi+\frac{(1-\widetilde{\alpha})\phi}{r}\Big)^2-(1-p)^2g'_3\Big[(E_2)^2+\frac{(E_3)^2}{(1-p)^2}\Big]\\
&+\phi^2\Big[\frac{(\widetilde{\alpha}-\widetilde{\alpha}^2)H_3-\widetilde{\alpha}rH'_3}{r^2}+\frac{6Mf_3}{r^4}-\frac{2Mf'_3}{r^3}+\frac{(1-p)^2g''_3}{r}+\frac{4M(1-p)g'_3}{r^3}\Big].
\end{split}
\end{equation}

Using \eqref{mora5} we calculate
\begin{equation*}
\begin{split}
2P_{(3)\mu} n^\mu&=2Q_{\mu\nu}X^\nu_{(3)} n^\mu+w_{(3)}\phi E_\mu n^\mu-\frac{1}{2}\phi^2n^\mu\D_\mu w_{(3)}+\frac{1}{2}n^\mu m_{(3)\mu}\phi^2\\
&=\frac{(E_1)^2}{r^2}\Big[\frac{f_3}{1-p}+g_3\Big]+(E_2)^2\big[f_3+g_3(1-p)\big]+(E_3)^2\Big[\frac{f_3}{(1-p)^2}+\frac{g_3(1-p^2\chi^2)}{1-p}\Big]\\
&+2E_2E_3\frac{f_3}{1-p}+\frac{2f_3}{r(1-p)}\phi E_3+\frac{m_{(3)3}}{2(1-p)}\phi^2\\
&=\frac{(E_1)^2}{r^2}\Big[\frac{f_3}{1-p}+g_3\Big]+f_3\Big[E_2+\frac{E_3}{1-p}+\frac{\phi}{r}\Big]^2-f_3\frac{\phi^2}{r^2}\\
&-2f_3E_2\frac{\phi}{r}+g_3(1-p)\Big[(E_2)^2+\frac{(E_3)^2(1-p^2\chi^2)}{(1-p)^2}\Big]-\frac{h_3}{r(1-p)}\phi^2.
\end{split}
\end{equation*}
Therefore
\begin{equation}\label{outs6}
\begin{split}
2P_{(3)\mu} n^\mu&=\frac{(E_1)^2}{r^2}\Big[\frac{f_3}{1-p}+g_3\Big]+f_3\Big[L\phi+\frac{\phi}{r}\Big]^2+g_3(1-p)\Big[(E_2)^2+\frac{(E_3)^2(1-p^2\chi^2)}{(1-p)^2}\Big]\\
&-\frac{1}{r^2}\partial_2\big[f_3r\phi^2\big]+\phi^2\Big[\frac{\widetilde{\alpha}H_3}{r(1-p)}+\frac{2Mf_3}{r^3(1-p)}-\frac{(1-p)g'_3}{r}\Big].
\end{split}
\end{equation}

\subsection{Proof of the Theorem \ref{MainTheorem3}}\label{endproof}
We compare now the expressions \eqref{outs7} and \eqref{outs6} with the lower bounds in \eqref{other4} and \eqref{other15}. We would like to fix the functions $f_3$ and $g_3$ and the constant $D$ in such a way that the sum of the corresponding expressions is bounded from below. More precisely, the sum of the spacetime integrals is pointwise bounded from below, while the sum of the integrals on the surfaces $\Sigma_t^c$ is bounded from below after integration by parts and the use of a simple Hardy-type inequality. 

One should think of the functions $f_3$ and $g_3$ in the following way: the function $f_3$ vanishes when $r\leq 8M$ and behaves like $r^\alpha$ as $r\to\infty$. On the other hand the function $g_3$ is an extremely large constant when $r\leq C_4M$, for some large constant $C_4$ but vanishes as $r\to\infty$ at a rate of $r^{\alpha-2}$. More precisely, we are looking for functions $f_3,g_3$ of the form
\begin{equation}\label{outs40}
f_3(r)=\eps_4\chi_{\geq 8M}(r)e^{\beta(r)},\qquad g_3(r)=\int_{r}^\infty \rho(s)\,ds,
\end{equation}
where $\eps_4=\eps_3^2$ is a small constant, $C_4=C_4(\alpha)\geq \eps_4^{-4}\alpha^{-1}(2-\alpha)^{-1}$ is a large constant (to be fixed), and $\beta,\rho:(c,\infty)\to[0,\infty)$ are smooth functions satisfying
\begin{equation}\label{outs41}
\begin{split}
\beta(r)\in[-10,0]\text{ and }M\beta'(r)\in[1/10,10]&\qquad\text{ if }r\in(c,8M],\\
\max\Big(\frac{\alpha}{100r},\frac{4M}{r^2}+\frac{1}{r}\mathbf{1}_{[8M,C_4M]}(r)\Big)\leq \beta'(r)\leq \frac{2}{r}&\qquad\text { if }r\in[8M,\infty),\\
\rho(r)=0\text{ and }g_3(r)\in[C_4/2,2C_4]&\qquad\text{ if }r\leq C_4M,\\
\rho(r)\leq \frac{\eps_4}{100}\beta'(r)e^{\beta(r)}\text{ and }\rho'(r)\leq\frac{\eps_4M}{100r^3}e^{\beta(r)}&\qquad\text{ if }r\geq C_4M,\\
\frac{e^\beta M^2}{r^2}\leq g_3(r)\leq\frac{C_4^{10}e^\beta M^2}{r^2}&\qquad\text{ if }r\geq C_4M,\\
(1-2\widetilde{\alpha})H_3(r)-rH'_3(r)\geq 0&\qquad\text{ if }r\in [16M,\infty).
\end{split}
\end{equation}
A specific choice satisfying these conditions is given in \eqref{outs50}--\eqref{outs51}. As a result of these conditions, we clearly have
\begin{equation}\label{outs41.9}
\begin{split}
&g'_3=-\rho,\\
&H_3\geq\frac{\eps_4}{100}e^{\beta}\beta'\chi_{\geq 8M},\\
&e^{\beta(r)}\in[r/(100M),r^2/M^2]\qquad\text{ for }r\in (c,C_4M].
\end{split}
\end{equation}

Let
\begin{equation*}
\begin{split}
(X,w,m)&:=(X_{(1)},w_{(1)},m_{(1)})+(X_{(2)},w_{(2)},m_{(2)})+(X_{(3)},w_{(3)},m_{(3)}),\\
\mathcal{T}[X,w,m]&:=\mathcal{T}[X_{(1)},w_{(1)},m_{(1)}]+\mathcal{T}[X_{(2)},w_{(2)},m_{(2)}]+\mathcal{T}[X_{(3)},w_{(3)},m_{(3)}],\\
P_{\mu}&:=P_{(1)\mu}+P_{(2)\mu}+P_{(3)\mu}.
\end{split}
\end{equation*}

Our next lemma contains the main bounds on the terms in the divergence identity \eqref{mora14}.

\begin{lemma}\label{BigBounds} 
Assume that the conditions \eqref{outs41} hold and that $C_4$ sufficiently large (depending on $\eps_4$). Then there is an absolute constant $\eps_5=\eps_5(\alpha)>0$ sufficiently small such that
\begin{equation}\label{outs43}
\begin{split}
\mathcal{T}[X,w,m]&\geq\eps_5\Big[\Big(\frac{2}{r}e^\beta-\beta'e^\beta+\frac{100}{r}\Big)\frac{(r-3M)^2}{r^2}\frac{(E_1)^2}{r^2}+e^{\beta}\beta'(L\phi)^2\\
&+\Big(\rho+\frac{M^2}{r^3}\Big)(E_2)^2+\Big(\rho+\frac{M^2(r-3M)^2}{r^5}\Big)(E_3)^2+\frac{e^\beta\beta'}{r^2}\phi^2\Big].
\end{split}
\end{equation}
Moreover, for any $t\in[0,T]$,
\begin{equation}\label{outs43.1}
\int_{\Sigma_{t}^c}2P_\mu n_0^\mu\,d\mu_{t}\geq \eps_5\int_{\Sigma_{t}^c}e^\beta\frac{(E_1)^2}{r^2}+e^\beta(L\phi)^2+g_3\big[(E_2)^2+(E_3)^2\big]+\frac{e^\beta\beta'}{r}\phi^2\,d\mu_t
\end{equation}
and
\begin{equation}\label{outs43.2}
\int_{\Sigma_{t}^c}2P_\mu n_0^\mu\,d\mu_{t}\leq \eps_5^{-1}\int_{\Sigma_{t}^c}e^\beta\frac{(E_1)^2}{r^2}+e^\beta(L\phi)^2+g_3\big[(E_2)^2+(E_3)^2\big]+\frac{e^\beta\beta'}{r}\phi^2\,d\mu_t.
\end{equation}
Finally,
\begin{equation}\label{outs43.3}
2P_\mu k^\mu\geq \eps_5\Big[\frac{(E_1)^2}{M^2}+(E_2)^2\frac{2M-c}{M}+(E_3)^2+\frac{\phi^2}{M^2}\Big]\qquad\text{ along }\mathcal{N}^c_{[t_1,t_2]},
\end{equation}
\end{lemma}

\begin{proof}
We start with the proof of \eqref{outs43}. Using the definitions we have
\begin{equation}\label{outs42}
\begin{split}
\frac{2f_3-rf'_3}{r}&=\eps_4 e^{\beta}\big[(2/r-\beta')\chi_{\geq 8M}-\chi'_{\geq 8M}\big],\\
\frac{6Mf_3}{r^4}-\frac{2Mf'_3}{r^3}+\frac{(1-p)^2g''_3}{r}+\frac{4M(1-p)g'_3}{r^3}&\geq\frac{\eps_4 M}{100r^4}e^\beta\chi_{\geq 8M}-\frac{2\eps_4 M}{r^3}e^\beta\chi'_{\geq 8M}.
\end{split}
\end{equation}
We combine the formulas \eqref{other4} and \eqref{outs7} to estimate
\begin{equation*}
\mathcal{T}[X,w,m]\geq I_1+I_2+I'_2+I_3,
\end{equation*}
where
\begin{equation*}
\begin{split}
&I_1:=\frac{(E_1)^2}{r^2}\frac{2f_3-rf'_3}{r}+\eps_3\frac{(E_1)^2}{r^2}\frac{(r-3M)^2}{r^3},\\
&I_2:=H_3\Big(L\phi+\frac{(1-\widetilde{\alpha})\phi}{r}\Big)^2,\\
&I'_2:=-(1-p)^2g'_3\Big[(E_2)^2+\frac{(E_3)^2}{(1-p)^2}\Big]+\eps_3\Big[\frac{M^2}{r^3}(E_2)^2+\frac{M^2(r-3M)^2}{r^5}(E_3)^2\Big],\\
&I_3:=\phi^2\Big[\frac{(\widetilde{\alpha}-\widetilde{\alpha}^2)H_3-\widetilde{\alpha}rH'_3}{r^2}+\frac{\eps_4 M}{100r^4}e^\beta\chi_{\geq 8M}-\frac{2\eps_4 M}{r^3}e^\beta\chi'_{\geq 8M}+\frac{\eps_3M}{r^4}-\eps_3^{-1}\frac{M}{r^4}\mathbf{1}_{[DM,\infty)}(r)\Big].
\end{split}
\end{equation*}
Using \eqref{outs41}, \eqref{outs41.9}, and \eqref{outs42} it is easy to see that, for some sufficiently small constant $\eps_5$ (which may depend on $\alpha$),
\begin{equation*}
\begin{split}
I_1&\geq\eps_5\Big[e^\beta\Big(\frac{2}{r}-\beta'\Big)\chi_{\geq 8M}+\frac{(r-3M)^2}{r^3}\Big]\frac{(E_1)^2}{r^2},\\
I_2+I'_2&\geq \eps_5e^{\beta}\beta'\chi_{\geq 8M}\Big(L\phi+\frac{(1-\widetilde{\alpha})\phi}{r}\Big)^2+\eps_5\Big(\rho+\frac{M^2}{r^3}\Big)(E_2)^2+\eps_5\Big(\rho+\frac{M^2(r-3M)^2}{r^5}\Big)(E_3)^2,\\
I_3&\geq \eps_5\Big(\frac{M}{r^4}e^{\beta}\chi_{\geq 8M}+\frac{e^{\beta}\beta'}{r^2}\Big)\phi^2,
\end{split}
\end{equation*}
provided that $\eps_4$ is fixed (sufficiently small relative to $\eps_3$), and $D$ is sufficiently large depending on $\eps_4$ such that $e^{\beta(DM)}\geq \eps_4^{-4}$). The bound \eqref{outs43} follows. 

To prove \eqref{outs43.1} we combine now the formulas \eqref{outs6}, \eqref{other15}, and \eqref{other9} to estimate
\begin{equation*}
2P_{\mu} n^\mu\geq I_4+I_5+I_6-\frac{1}{r^2}\partial_2\big[f_3r\phi^2\big]-\frac{\chi_{\geq 8M}(r)(1-p)}{r^2}\partial_2(r\phi^2),
\end{equation*}
where
\begin{equation*}
\begin{split}
I_4:&=\frac{(E_1)^2}{r^2}\Big[\frac{f_3}{1-p}+g_3\Big]-\eps_3^{-1}\frac{(E_1)^2}{r^2},\\
I_5:&=f_3\Big(L\phi+\frac{\phi}{r}\Big)^2+g_3(1-p)\Big[(E_2)^2+\frac{(E_3)^2(1-p^2\chi^2)}{(1-p)^2}\Big]+\eps_3(E_2)^2\mathbf{1}_{(c_0,17M/8]}(r)\\
&-\eps_3^{-1}\Big[(L\phi)^2+\frac{M^2}{r^{2}}\big[(E_2)^2|1-p|+(E_3^2)\big]\Big],\\
I_6:&=\phi^2\Big[\frac{\widetilde{\alpha}F_3}{r(1-p)}+\frac{2Mf_3}{r^3(1-p)}-\frac{(1-p)g'_3}{r}\Big]-\eps_3^{-1}\frac{1}{r^2}\phi^2.
\end{split}
\end{equation*}
Using \eqref{outs41}, \eqref{outs41.9}, and \eqref{outs42} it follows that
\begin{equation*}
\begin{split}
I_4&\geq \eps_5e^\beta\frac{(E_1)^2}{r^2},\\
I_5+I_6&\geq \eps_5e^\beta\Big(L\phi+\frac{\phi}{r}\Big)^2+[g_3(1-p)+\eps_3\mathbf{1}_{(c_0,17M/8]}(r)]\frac{(E_2)^2}{2}+\eps_5 g_3(E_3)^2\\
&+\eps_4\frac{\widetilde{\alpha}e^\beta\beta'}{1000r}\phi^2-10\eps_3^{-1}\frac{1}{r^2}\phi^2,
\end{split}
\end{equation*}
provided that $C_4$ is sufficiently large (relative to $\eps_4$) and $|c_0-2M|\leq C_4^{-10}$ is sufficiently small. Using the Hardy inequalities in Lemma \ref{GeneralLemma} (i) and (ii) it is easy to see that the integral on the negative term $-10\eps_3^{-1}r^{-2}\phi^2$ in $I_6$ along $\Sigma_t^c$ can be absorbed by the integrals of the positive terms $\eps_4\frac{\widetilde{\alpha}e^\beta\beta'}{1000r}\phi^2$ and $g_3(1-p)\frac{(E_2)^2}{2}$, provided that the constant $C_4$ is sufficiently large.

Moreover, notice that for any $t\in[0,T]$
\begin{equation*}
\int_{\Sigma_t^c}2P_\mu n_0^\mu\,d\mu_{t}=C\int_{\mathbb{S}^2}\int_{c}^\infty2P_\mu n^\mu r^2(\sin\theta)\,drd\theta.
\end{equation*}
After integration by parts in $r$ it follows that
\begin{equation*}
\begin{split}
\Big|\int_{\mathbb{S}^2} &\int_{c}^\infty\frac{1}{r^2}\partial_2\big[f_3r\phi^2\big]r^2(\sin\theta)\,drd\theta\Big|\\
&+\Big|\int_{\mathbb{S}^2} \int_{c}^\infty\frac{\chi_{\geq 8M}(r)(1-p)}{r^2}\partial_2(r\phi^2)r^2(\sin\theta)\,drd\theta\Big|\leq \eps_4^{-1}\int_{\Sigma_t^c}\frac{1}{r^2}\phi^2\,d\mu_t,
\end{split}
\end{equation*}
so these terms can also be absorbed. The desired bound \eqref{outs43.1} follows.

The proof of \eqref{outs43.2} is similar, starting from the inequality \eqref{other15.2} and the identity \eqref{outs6}. To prove \eqref{outs43.3} we start from the bound \eqref{other5},
\begin{equation*}
2(P_{(1)\mu}+P_{(2)\mu}) k^\mu\geq\eps_3\big[(E_1/r)^2+(E_2)^2(2-c/M)+M^{-2}\phi^2\big]-\eps_3^{-1}(E_3)^2.
\end{equation*}
The identity \eqref{mora6} shows that
\begin{equation*}
2P_{(3)\mu} k^\mu=2k^\mu Q_{\mu\nu}X^\nu_{(3)}=2g_3(c)p(E_3)^2+2g_3(c)(1-p)E_2E_3.
\end{equation*}
The lower bound \eqref{outs43.3} follows since $g_3(c)\in[C_4/2,2C_4]$, provided that $C_4$ is sufficiently large and $|c-2M|/M$ is sufficiently small.
\end{proof}

\begin{proof}[Proof of Theorem \ref{MainTheorem3}] We can now complete the proof of Theorem \ref{MainTheorem3}, using Lemma \ref{BigBounds} and the divergence identity. We have to fix functions $\beta$ and $\rho$ satisfying \eqref{outs41}. With $\alpha$ as in the statement of the theorem, we define first the smooth function $\beta$ by setting $\beta(8M)=0$ and
\begin{equation}\label{outs50}
\begin{split}
&\beta'(r)=\Big(\frac{4M}{r^2}+\frac{1}{r}\Big)\big(1-\chi_{\geq C_4^4M}(r)\big)+\frac{\alpha}{r}\chi_{\geq C_4^4M}(r).
\end{split}
\end{equation}
This choice clearly satisfies the first two conditions in \eqref{outs41}. Then we define  
\begin{equation}\label{outs51}
\rho(r)=\delta M^{-1}\Big[\chi_{\geq C_4M}(r)+\chi_{\geq 4C_4^4M}(r)\Big(C_4^7e^{\beta(r)}\frac{M^3}{r^3}-1\Big)\Big],
\end{equation}
where $\delta\in[10^{-4}C_4^{-3},10^4C_4^{-3}]$ is such that $\int_{C_4M}^\infty\rho(s)\,ds=C_4$.

Notice that
\begin{equation}\label{outs52}
e^{\beta(r)}\approx \frac{r}{M}\text{ if }r\leq 10 C_4^4M\quad \text{ and }\quad e^{\beta (r)}\approx C_4^4\Big(\frac{r}{C_4^4M}\Big)^\alpha\text{ if }r\geq (1/10) C_4^4M.
\end{equation} 
The other bounds in \eqref{outs41} follow easily. Moreover, the definitions show that
\begin{equation*}
\begin{split}
&e^{\beta(r)}\approx_{C_5}\frac{r^\alpha}{M^\alpha},\qquad \beta'(r)\approx_{C_5}\frac{1}{r},\qquad \Big(\frac{2}{r}-\beta'(r)\Big)\approx_{C_5}\frac{1}{r},\\
&\rho(r)\approx_{C_5}\chi_{\geq C_4M}(r)\frac{M^{2-\alpha}}{r^{3-\alpha}}, \qquad g_3(r)\approx_{C_5}\frac{r^{\alpha-2}}{M^{\alpha-2}}.
\end{split}
\end{equation*}
for some large constant $C_5$, where $A\approx_{C_5}B$ means $A\in[C_5^{-1}B,C_5B]$. The desired conclusion of the theorem follows from Lemma \ref{BigBounds} and the divergence identity.
\end{proof}

\section{Proof of Theorem \ref{MainTheorem}}\label{lok}

In this section we prove Theorem \ref{MainTheorem}. We still use some of the ideas from the previous section. We use the more complicated divergence identities \eqref{zq51.5} and \eqref{zq52.5},
\begin{equation}\label{moro3}
\begin{split}
2\D^\mu P_\mu&=2X^\nu J_\nu+Q_{\mu\nu}{}^{(X)}\pi^{\mu\nu}+w(E_\alpha E^\alpha+F_\alpha F^\alpha+M_\alpha M^\alpha)+(\phi m^\mu\D_\mu\phi+\psi{m'}^\mu\D_\mu\psi)\\
&+\frac{1}{2}\phi^2(\D^\mu m_\mu-\square w)+\frac{1}{2}\psi^2(\D^\mu m'_\mu-\square w)+w(\phi\mathcal{N}_\phi+\psi\mathcal{N}_\psi),
\end{split}
\end{equation} 
where
\begin{equation}\label{moro2.01}
E_\mu=\D_\mu\phi+\psi A^{-1}\D_\mu B,\quad F_\mu=\D_\mu\psi-\phi A^{-1}\D_\mu B,\quad M_\mu=\frac{\phi\D_\mu B-\psi\D_\mu A}{A},
\end{equation}
\begin{equation}\label{moro1.99}
Q_{\mu\nu}:=E_\mu E_\nu+F_\mu F_\nu+M_\mu M_\nu-(1/2)\g_{\mu\nu}(E_\alpha E^\alpha+F_\alpha F^\alpha+M_\alpha M^\alpha),
\end{equation}
\begin{equation}\label{moro2}
P_\mu=P_\mu[X,w,m,m']=Q_{\mu\nu}X^\nu+\frac{1}{2}w(\phi E_\mu+\psi F_\mu)-\frac{1}{4}\D_\mu w(\phi^2+\psi^2)+\frac{1}{4}(m_\mu\phi^2+m'_\mu\psi^2),
\end{equation}
and
\begin{equation}\label{moro4}
J_\nu=\frac{2\D_\nu BM^\mu E_\mu-2\D_\nu AM^\mu F_\mu}{A}+\mathcal{N}_\phi E_\nu+\mathcal{N}_\psi F_\nu.
\end{equation}

Recall (see \eqref{zq4}) that
\begin{equation}\label{ell11}
A=\frac{\Sigma^2(\sin\theta)^2}{q^2},\qquad B=-\Big[2aM(3\cos\theta-(\cos\theta)^3)+\frac{2a^3M(\sin\theta)^4\cos\theta}{q^2}\Big].
\end{equation}
These formulas show that
\begin{equation}\label{ell12}
\begin{split}
&A^{-1}\D_1B=\frac{6aM q^2\sin\theta}{\Sigma^2}-\frac{2a^3M[4\sin\theta q^2-5(\sin\theta)^3q^2+2a^2(\sin\theta)^3(\cos\theta)^2]}{\Sigma^2q^2},\\
&A^{-1}\D_2B=\frac{4ra^3M(\sin\theta)^2\cos\theta}{q^2\Sigma^2},\\
&A^{-1}\D_1A=\frac{2\cos\theta}{\sin\theta}-\frac{2a^2\Delta\sin\theta\cos\theta}{\Sigma^2}-\frac{2a^2\sin\theta\cos\theta}{q^2},\\
&A^{-1}\D_2A=\frac{4r(r^2+a^2)-a^2(\sin\theta)^2(2r-2M)}{\Sigma^2}-\frac{2r}{q^2}.
\end{split}
\end{equation}
Notice that
\begin{equation}\label{ell13}
r^{-1}\Big|\frac{\D_1 B}{A}\Big|+\Big|\frac{\D_2 B}{A}\Big|+r^{-1}\Big|\frac{\D_1 A}{A}-\frac{2\cos\theta}{\sin\theta}\Big|+\Big|\frac{\D_2 A}{A}-\frac{2}{r}\Big|\lesssim aMr^{-3}.
\end{equation}
and
\begin{equation}\label{ell13.9}
\begin{split}
\frac{|E_1-\D_1\phi|}{r}+|E_2-\D_2\phi|+|E_3-\D_3\phi|&\lesssim aMr^{-3}\big(|\phi|+|\psi|\big),\\
\frac{|F_1-\D_1\psi|}{r}+|F_2-\D_2\psi|+|F_3-\D_3\psi|&\lesssim aMr^{-3}\big(|\phi|+|\psi|\big),\\
\Big|\frac{M_1}{r}+\frac{2\cos\theta}{r\sin\theta}\psi\Big|+\Big|M_2+\frac{2\psi}{r}\Big|+|M_3|&\lesssim aMr^{-3}\big(|\phi|+|\psi|\big).
\end{split}
\end{equation}

Letting $\og^{\al\be}$ denote the Schwarzschild components of the metric, see \eqref{mora2}, and $\g^{\al\be}$ the Kerr components, we notice that
\begin{equation}\label{DiffComp}
\begin{split}
&\g^{11}=\og^{11}+O(a^2r^{-4}),\qquad\quad \g^{22}=\og^{22}+O(a^2r^{-2}),\\
&\g^{23}=\og^{23}+O(a^2M^2r^{-4}),\qquad \g^{33}=\og^{33}+O(a^2r^{-2}).
\end{split}
\end{equation}

We notice that the term $J_1$ in \eqref{moro4} is singular when $\theta=0$, due to the fraction $\D_1A/A$. To eliminate this 
singularity we work with a modification of the $1$-form $P$, namely
\begin{equation}\label{moro2.1}
\widetilde{P}_\mu=\widetilde{P}_\mu[X,w,m,m']:=P_\mu-\frac{X^\nu\D_\nu A}{A}\frac{\D_\mu A}{A}\psi^2.
\end{equation}
Then
\begin{equation}\label{moro3.1}
2\D^\mu \widetilde{P}_\mu=2\D^\mu P_\mu-4\frac{X^\nu\D_\nu A}{A}\frac{\D_\mu A}{A}\psi\D^\mu\psi-
2\D^\mu\Big[\frac{X^\nu\D_\nu A}{A}\frac{\D_\mu A}{A}\Big]\psi^2=\sum_{j=1}^5 L^j,
\end{equation}
where
\begin{equation}\label{moro3.2}
\begin{split}
&L^1=L^1[X,w,m,m']:=Q_{\mu\nu}{}^{(X)}\pi^{\mu\nu}+w(E_\alpha E^\alpha+F_\alpha F^\alpha+M_\alpha M^\alpha),\\
&L^2=L^2[X,w,m,m']:=\phi m^\mu\D_\mu\phi+\psi{m'}^\mu\D_\mu\psi,\\
&L^3=L^3[X,w,m,m']:=\frac{1}{2}\phi^2(\D^\mu m_\mu-\square w)+\frac{1}{2}\psi^2(\D^\mu m'_\mu-\square w),\\
&L^4=L^4[X,w,m,m']:=-2\D^\mu\Big[\frac{X^\nu\D_\nu A}{A}\frac{\D_\mu A}{A}\Big]\psi^2,\\
&L^5=L^5[X,w,m,m']:=2X^\nu J_\nu-4\frac{X^\nu\D_\nu A}{A}\frac{\D_\mu A}{A}\psi\D^\mu\psi+w(\phi\mathcal{N}_\phi+\psi\mathcal{N}_\psi).
\end{split}
\end{equation}

The terms $L^1,L^2,L^3$ are similar to the corresponding terms we estimated in the proof of Theorem \ref{MainTheorem3}. The main new terms are $L^4$ and the quadratic part of $L^5$. We describe these terms below.

\begin{lemma}\label{clan0}
Assuming that $X=f\partial_2+g\partial_3$, where $f$ may depend only on $r$, we have
\begin{equation}\label{moro43}
L^4=-8\frac{\g^{22}}{r}\partial_2\big[r^{-1}f\big]\psi^2+O(a^2r^{-5})\big[|f|+r|f'|\big]\psi^2
\end{equation}
and
\begin{equation}\label{moro43.1}
\begin{split}
|L^5|&\lesssim \frac{aM}{r^4}|f|\big(|\phi|+|\psi|\big)\Big\{\sum_{Y\in\{E,F\}}\Big(\frac{|Y_1|}{r}+\frac{M}{r}|Y_2|+\frac{M}{r}|Y_3|\Big)+\frac{1}{r}\big(|\phi|+|\psi|\big)\Big\}\\
&+|\mathcal{N}_\phi|\big|2fE_2+2gE_3+w\phi\big|+|\mathcal{N}_\psi|\big|2fF_2+2gF_3+w\psi\big|.
\end{split}
\end{equation}
\end{lemma}

\begin{proof} We rewrite
\begin{equation*} 
L^4=-2\D^\mu\Big[\frac{X^\nu\D_\nu A}{A}\frac{\D_\mu A}{A}\Big]\psi^2=-2\D^\mu\Big[\frac{f\D_2 A}{A}\frac{\D_\mu A}{A}\Big]\psi^2.
\end{equation*}
In view of \eqref{zq5} and \eqref{ell12},
\begin{equation*}
\Big|\frac{f\D_2 A}{A}\D^\mu\Big[\frac{\D_\mu A}{A}\Big]\Big|=\Big|\frac{f\D_2 A}{A}\frac{\D_\mu B\D^\mu B}{A^2}\Big|\lesssim \frac{a^2M^2}{r^7}|f|.
\end{equation*}
Also
\begin{equation*}
\Big|\g^{11}\partial_1\Big[\frac{f\partial_2 A}{A}\Big]\frac{\partial_1 A}{A}\Big|\lesssim \frac{a^2}{r^5}|f|.
\end{equation*}
and
\begin{equation*}
\Big|\g^{22}\partial_2\Big[\frac{f\partial_2 A}{A}\Big]\frac{\partial_2 A}{A}-\g^{22}\partial_2\Big[\frac{2f}{r}\Big]\frac{2}{r}\Big|\lesssim \frac{a^2}{r^5}|f|+\frac{a^2}{r^4}|f'|.
\end{equation*}
The desired formula \eqref{moro43} follows.

We estimate now the term $L^5$. We start by rewriting
\begin{equation*}
\begin{split}
L^5&=2X^\nu \Big[\frac{2\D_\nu BM^\mu E_\mu-2\D_\nu AM^\mu F_\mu}{A}\Big]-4\frac{X^\nu\D_\nu A}{A}\frac{\D_\mu A}{A}\psi\D^\mu\psi\\
&+\mathcal{N}_\phi\big(2X^\nu E_\nu+w\phi\big)+\mathcal{N}_\psi\big(2X^\nu F_\nu+w\psi\big).
\end{split}
\end{equation*}
Using \eqref{ell12} and \eqref{moro2.01}, we estimate
\begin{equation*}
\Big|2X^\nu \frac{2\D_\nu BM^\mu E_\mu}{A}\Big |\lesssim \frac{a^2M}{r^5}|f|\big(|\phi|+|\psi|\big)\big[|E_1/r|+Mr^{-1}|E_2|+Mr^{-1}|E_3|\big],
\end{equation*}
and
\begin{equation*}
\Big|2X^\nu \Big[\frac{-2\D_\nu AM^\mu F_\mu}{A}\Big]-4\frac{X^\nu\D_\nu A}{A}\frac{\D_\mu A}{A}\psi\D^\mu\psi\Big|\lesssim \frac{aM}{r^4}|f||\phi|\Big[|F_1/r|+\frac{M}{r}|F_2|+\frac{M}{r}|F_3|+\frac{1}{r}|\psi|\Big].
\end{equation*}
The desired formula \eqref{moro43.1} follows.
\end{proof}

As in the proof of Theorem \ref{MainTheorem3}, our goal is to choose suitable multipliers $(X,w,m,m')$ in a such a way that the quadratic terms in the divergence formula
\begin{equation}\label{red2}
\int_{\Sigma_{t_1}^c}\widetilde{P}_\mu n_0^\mu\,d\mu_{t_1}=\int_{\Sigma_{t_2}^c}\widetilde{P}_\mu n_0^\mu\,d\mu_{t_2}+\int_{\mathcal{N}^c_{[t_1,t_2]}}\widetilde{P}_\mu k_0^\mu\,d\mu_c+\int_{\mathcal{D}^c_{[t_1,t_2]}}\D^\mu \widetilde{P}_\mu\,d\mu
\end{equation}
are nonegative, where $t_1,t_2\in[0,T]$, $c\in(c_0,\r_h]$, $n_0:=n/|\g^{33}|^{1/2}$, $k_0:=k/|\g^{22}|^{1/2}$, and the integration is with respect to the natural measures induced by the metric $\g$. 

\subsection{The multipliers $(X_{(k)},w_{(k)},m_{(k)},m'_{(k)})$, $k\in\{1,2,3,4\}$} In this subsection we introduce the main multipliers. The multipliers $(X_{(k)},w_{(k)},m_{(k)},m'_{(k)})$, $k\in\{1,2,3\}$ are analogous to the multipliers $(X_{(k)},w_{(k)},m_{(k)})$, $k\in\{1,2,3\}$, used in the analysis of the wave equation in Schwarzschild spacetime in the previous section. On the other hand, the multiplier $(X_{(4)},w_{(4)},m_{(4)},m'_{(4)})$, which is supported in a small region close to the trapped set, is new and is used mostly to control the contribution of the new term $L^4$ in \eqref{moro3.2}.  

\subsubsection{Analysis around the trapped set} As in the previous section, we start by constructing the multiplier $(X_{(1)},w_{(1)},m_{(1)},m'_{(1)})$, which is relevant in a neighborhood of the trapped set. For now our main concern is the positivity of the spacetime integral $\D^\mu \widetilde{P}_\mu$; as in the proof of Theorem \ref{MainTheorem3}, the positivity of the surfaces integrals along $\Sigma_{t}^c$ and $\mathcal{N}^c_{[t_1,t_2]}$ can only be addressed after the other multipliers are introduced. 

It is important to recall that we are in the axially symmetric case. Therefore the relevant trapped null geodesics are still confined to a codimension $1$ set. More precisely, recalling that $a\ll M$, it is easy to see that the equation $r^3-3Mr^2+a^2r+Ma^2=0$ has a unique solution $r^\ast\in(c_0,\infty)$. 
Moreover, $r^\ast\in[3M-a^2/M,3M]$ and 
\begin{equation}\label{moro9}
\big|r^3-3Mr^2+a^2r+Ma^2-(r-r^\ast)r^2\big|\lesssim (a^2/M)r|r-r^\ast|\qquad\,\text{ if }r\in (c_0,\infty).
\end{equation} 

We start by setting, as before,
\begin{equation}\label{moro6.7}
\begin{split}
&X_{(1)}:=f_1(r)\partial_2+g_1(r)\partial_3,\qquad f_1(r):=\frac{a_1(r)\Delta}{r^2},\qquad g_1(r):=\frac{a_1(r)\chi(r)2M}{r}+1,\\
&w_{(1)}(r,\theta):=f'_1(r)+f_1(r)\partial_r\log\big(\Sigma^2/\Delta)-\eps_1\widetilde{w}(r),\\
&\widetilde{w}(r):=M^2(r-33M/16)^3(r-r^\ast)^2r^{-8}\mathbf{1}_{[33M/16,\infty)}(r),\\
&m_{(1)}=m'_{(1)}:=0,\\
\end{split}
\end{equation}
where $a_1:(0,\infty)\to\mathbb{R}$ is a smooth function to be fixed, $\lim_{r\to\infty}a_1(r)=1$, $\eps_1\in(0,1]$ is a small constant and $\Sigma^2=(r^2+a^2)^2-a^2(\sin\theta)^2\Delta$ is as 
in \eqref{zq2}. 

Let 
\begin{equation*}
L^j_{(1)}:=L^j[X_{(1)},w_{(1)},m_{(1)},m'_{(1)}],
\end{equation*}
for $j\in\{1,2,3,4,5\}$, see \eqref{moro3.2}. Notice that
\begin{equation}\label{moro11.7}
L^2_{(1)}=0,\qquad L^3_{(1)}=-\frac{1}{2}\square w_{(1)}(\phi^2+\psi^2).
\end{equation}
Using \eqref{kra7.4},
\begin{equation*}
L^1_{(1)}=\sum_{Y\in\{E,F,M\}}\big[K_{(1)}^{11}(Y_1)^2+
K_{(1)}^{22}(Y_2)^2+K_{(1)}^{33}(Y_3)^2+2K_{(1)}^{23}Y_2Y_3\big],
\end{equation*}
where
\begin{equation*}
\begin{split}
&K_{(1)}^{11}=\frac{-f'_1(r)}{q^2}+w_{(1)}(r,\theta)\g^{11},\\
&K_{(1)}^{22}=\frac{-f_1(r)(2r-2M)+f'_1(r)\Delta}{q^2}+w_{(1)}(r,\theta)\g^{22},\\
&K_{(1)}^{33}=-f_1(r)\partial_2\g^{33}+2g'_1(r)\g^{23}-f'_1(r)\g^{33}-\frac{2rf_1(r)\g^{33}}{q^2}+w_{(1)}(r,\theta)\g^{33},\\
&K_{(1)}^{23}=\frac{-2Mrf_1(r)\chi'(r)-2Mf_1(r)\chi(r)+g'_1(r)\Delta}{q^2}+w_{(1)}(r,\theta)\g^{23}.
\end{split}
\end{equation*}

Simple calculations, using also \eqref{exp2.1}, show that
\begin{equation}\label{moro8.7}
\begin{split}
&\partial_r\log\big(\Sigma^2/\Delta)=\frac{\Delta\partial_r\Sigma^2-\Sigma^2\partial_r\Delta}{\Delta\Sigma^2}=\frac{2(r^2+a^2)(r^3-3Mr^2+a^2r+Ma^2)}{\Delta\Sigma^2},\\
&\g^{33}=-\frac{\Sigma^2}{q^2\Delta}+\frac{4M^2r^2}{q^2\Delta}\chi(r)^2.
\end{split}
\end{equation}
Using also the formulas \eqref{moro6} and \eqref{exp2.1} we calculate
\begin{equation*}
\begin{split}
&K_{(1)}^{11}=a_1(r)\frac{2(r^2+a^2)(r^3-3Mr^2+a^2r+Ma^2)}{r^2q^2\Sigma^2}-\eps_1\widetilde{w}(r)\g^{11},\\
&K_{(1)}^{22}=\frac{2\Delta^2}{q^2r^2}\Big[a'_1(r)+a_1(r)\frac{-2a^2(r^2+a^2)+a^2(\sin\theta)^2(r^2-3Mr+2a^2)}{\Sigma^2r}\Big]-\eps_1\widetilde{w}(r)\g^{22},\\
&K_{(1)}^{33}=\frac{8M^2\chi(r)^2}{q^2}\Big[a'_1(r)+a_1(r)\frac{-2a^2(r^2+a^2)+a^2(\sin\theta)^2(r^2-3Mr+2a^2)}{\Sigma^2r}\Big]-\eps_1\widetilde{w}(r)\g^{33},\\
&K_{(1)}^{23}=\frac{4M\Delta\chi(r)}{q^2r}\Big[a'_1(r)+a_1(r)\frac{-2a^2(r^2+a^2)+a^2(\sin\theta)^2(r^2-3Mr+2a^2)}{\Sigma^2r}\Big]-\eps_1\widetilde{w}(r)\g^{23}.
\end{split}
\end{equation*}

Therefore
\begin{equation}\label{moro10.7}
\begin{split}
L^1_{(1)}&\geq\sum_{Y\in\{E,F,M\}}\Big\{\frac{(2-a/M)a_1(r)(r-r^\ast)-\eps_1r^4\widetilde{w}(r)q^{-2}}{r^4}(Y_1)^2\\
&+\Big[(2-a/M)a'_1(r)-\eps_1\widetilde{w}(r)\frac{r^4}{q^2\Delta}\Big]\Big(\frac{\Delta}{r^2}Y_2+\frac{2M\chi(r)}{r}Y_3\Big)^2+\eps_1\widetilde{w}(r)\frac{\Sigma^2}{q^2\Delta}(Y_3)^2\Big\},
\end{split}
\end{equation}
provided that $a$ is sufficiently small and
\begin{equation}\label{littletest}
a_1(r^\ast)=0\qquad\text{ and }\qquad a'_1(r)\geq a^{1/2}M^{3/2}r^{-3}|a_1(r)|\text{ for }r\in(c_0,\infty).
\end{equation}
This condition is clearly satisfied by the function $a_1$ defined below.

The important function $a_1$ is defined as in the proof of Theorem \ref{MainTheorem3}, see \eqref{moro35},
\begin{equation}\label{moro35.7}
\begin{split}
&R(r):=(r-r^\ast)(r+2M)+6M^2\log\Big(\frac{r-\r_h}{r^\ast-\r_h}\Big),\\
&a_1(r):=r^{-2}\delta^{-1}\kappa(\delta R(r))+\Big[\frac{r^\ast-2M}{r}-\frac{6M^2}{r^2}\log\Big(\frac{r-\r_h}{r^\ast-\r_h}\Big)\Big]\chi_{\geq DM}(r),
\end{split}
\end{equation}
where $\delta:=\eps_2^2M^{-2}$ is a small constant and $D\gg 1$ is a large constant. This function can be analyzed as in section \ref{MoraEst}, see \eqref{moro35.15}--\eqref{moro30.3}, once we observe that
\begin{equation*}
\r_h=2M+O(a^2/M),\qquad r^\ast=3M+O(a^2/M),\qquad \Sigma^2=r^4+O(a^2r^2).
\end{equation*}
Recalling also the identities \eqref{DiffComp} and defining
\begin{equation}\label{moro35.127}
h_1(r,\theta):=f'_1(r)+f_1(r)\partial_r\log\big(\Sigma^2/\Delta)=\frac{\Delta}{\Sigma^2}\partial_r\big[a_1(r)\Sigma^2r^{-2}\big],
\end{equation}
we estimate, as in \eqref{moro31.7},
\begin{equation}\label{moro31.77}
\begin{split}
(\square h_1)(r,\theta)&=-\frac{2M}{r^4}\Big(7-\frac{44M}{r}+\frac{72M^2}{r^2}\Big)+O(ar^{-4})+O(Mr^{-4})\mathbf{1}_{[DM,\infty)}(r)\\
&+M^{-3}O(1)\mathbf{1}_{(c_0,r_\delta]}(r)+O\Big(\frac{\delta^2M^2}{r-\r_h}\Big)\mathbf{1}_{[r'_\delta,r_\delta]}(r),
\end{split}
\end{equation}
where $r_\delta$ and $r'_\delta$ denote the unique numbers in $(\r_h,\infty)$ with the property that $R(r_\delta)=-1/\delta$ and $R(r'_\delta)=-2/\delta$. We also have, compare with \eqref{moro35.11},
\begin{equation}\label{moro35.117}
a_1(r^\ast)=0\qquad\text{ and }\qquad a'_1(r)\geq 10 M^2r^{-3}\qquad\text{ for }r\in (c_0,\infty),
\end{equation}
if $\delta$ is sufficiently small. In particular, this implies \eqref{littletest} if $a$ is sufficiently small relative to $\eps_2$. 

The bound \eqref{moro10.7} shows that
\begin{equation}\label{moro31.777}
\begin{split}
L_{(1)}^1\geq\sum_{Y\in\{E,F,M\}}\Big\{&\frac{(2-C_1\eps_1)a_1(r)(r-r^\ast)}{r^4}(Y_1)^2+\eps_1\widetilde{w}(r)(Y_3)^2\\
&+(2-C_1\eps_1)a'_1(r)\Big(\frac{\Delta}{r^2}Y_2+\frac{2M\chi(r)}{r}Y_3\Big)^2\Big\},
\end{split}
\end{equation}
for a sufficiently large constant $C_1$, provided that the constant $\eps_1$ is sufficiently small and $a/M\leq\eps_1$. Moreover, the identities \eqref{moro11.7} and \eqref{moro31.77} show that
\begin{equation}\label{moro31.778}
\begin{split}
L^3_{(1)}&\geq \frac{M(1-C_1\eps_1)}{r^4}\Big(7-\frac{44M}{r}+\frac{72M^2}{r^2}\Big)(\phi^2+\psi^2)-\frac{C_1M}{r^4}\mathbf{1}_{[DM,\infty)}(r)(\phi^2+\psi^2)\\
&-\frac{C_1}{M^3}\mathbf{1}_{(c_0,r_\delta]}(r)(\phi^2+\psi^2)-\frac{C_1\delta^2M^2}{r-\r_h}\mathbf{1}_{[r'_\delta,r_\delta]}(r)(\phi^2+\psi)^2.
\end{split}
\end{equation}

The bounds \eqref{moro43} and \eqref{moro43.1} and the definitions show that 
\begin{equation}\label{moro43.5}
L^2_{(1)}=0,
\end{equation}
\begin{equation}\label{moro43.3}
\begin{split}
L^4_{(1)}&=-8\frac{\g^{22}}{r}\partial_2\big[r^{-1}f_1\big]\psi^2+O(a^2r^{-5})\big[|f_1|+r|f'_1|\big]\psi^2\\
&=\Big[-\frac{8\Delta^2}{q^2r^4}a'_1(r)+\frac{8\Delta(r^2-4Mr)}{q^2r^5}a_1(r)\Big]\psi^2+O(a^2r^{-5})\big[|a_1|+|r-\r_h||a'_1|\big]\psi^2.
\end{split}
\end{equation}
and
\begin{equation}\label{moro43.4}
\begin{split}
|L^5_{(1)}|&\lesssim \frac{aM}{r^4}|f|\big(|\phi|+|\psi|\big)\Big\{\sum_{Y\in\{E,F\}}\Big(\frac{|Y_1|}{r}+\frac{M}{r}|Y_2|+\frac{M}{r}|Y_3|\Big)+\frac{1}{r}\big(|\phi|+|\psi|\big)\Big\}\\
&+|\mathcal{N}_\phi|\big|2f_1E_2+2g_1E_3+w_1\phi\big|+|\mathcal{N}_\psi|\big|2f_1F_2+2g_1F_3+w_1\psi\big|.
\end{split}
\end{equation}

Using \eqref{moro31.77} and \eqref{moro43.3}, together with the inequalities in the last line of \eqref{ell13.9}, after possibly increasing the constant $C_1$ we have
\begin{equation}\label{moro31.79}
\begin{split}
L_{(1)}^1&+L_{(1)}^4\geq\sum_{Y\in\{E,F\}}\Big\{\frac{(2-C_1\eps_1)a_1(r)(r-r^\ast)}{r^4}(Y_1)^2+\eps_1\widetilde{w}(r)(Y_3)^2\\
&+(2-C_1\eps_1)a'_1(r)\Big(\frac{\Delta}{r^2}Y_2+\frac{2M\chi(r)}{r}Y_3\Big)^2\Big\}+\frac{8\Delta(r^2-4Mr)}{r^7}a_1(r)\psi^2\\
&+\frac{(2-C_1\eps_1)a_1(r)(r-r^\ast)}{r^4}\frac{4(\cos\theta)^2\psi^2}{(\sin\theta)^2}-C_1\frac{a^2|a_1(r)|+\eps_1r^2|r-\r_h|a'_1(r)}{r^5}(\phi^2+\psi^2).\\
\end{split}
\end{equation}

\subsubsection{Analysis in a neighborhood of the horizon} In a small neighborhood of the horizon we need to use the redshift effect. As in subsection \ref{multipliers}, we define
\begin{equation}\label{moro12.9}
\begin{split}
&X_{(2)}:=f_2(r)\partial_2+g_2(r)\partial_3,\qquad f_2(r):=-\eps_2a_2(r),\qquad g_2(r):=\eps_2 a_2(r)(1-\eps_2),\\
&w_{(2)}(r):=-2\eps_2 a_2(r)/r,\\
&m_{(2)2}=m_{(2)3}=m'_{(2)2}=m'_{(2)3}:=\eps_2M^{-2}\gamma(r),\qquad m_{(2)1}=m_{(2)4}=m'_{(2)1}=m'_{(2)4}:=0,
\end{split}
\end{equation}
where $\eps_2$ is a small positive constant (recall that $\delta=\eps_2^2M^{-2}$),
\begin{equation}\label{moro12.59}
a_2(r):=
\begin{cases}
M^{-3}(9M/4-r)^3\qquad&\text{ if }r\leq 9M/4,\\
0\qquad&\text{ if }r\geq 9M/4,
\end{cases}
\end{equation}
and $\gamma:[c_0,\infty)\to[0,1]$ is a function supported in $[c_0,17M/8]$, and satisfying $\gamma(\r_h)=1/2$ and a property similar to \eqref{other2}. 

Let $L^j_{(2)}:=L^j[X_{(2)},w_{(2)},m_{(2)},m'_{(2)}]$, $j\in\{1,2,3,4,5\}$. As in the proof of Theorem \ref{MainTheorem3}, see Lemma \ref{flux} and \eqref{other4}, the multipliers $(X_{(1)},w_{(1)},m_{(1)},m'_{(1)})$ and $(X_{(2)},w_{(2)},m_{(2)},m'_{(2)})$ can be combined to prove the following:

\begin{lemma}\label{flux2}
The constants $\eps_1,\eps_2$ can be fixed small enough such that there is a sufficiently small absolute constant $\eps_3>0$ with the property that
\begin{equation}\label{other4.9}
\begin{split}
\sum_{j=1}^4\big(L^j_{(1)}+L^j_{(2)}\big)&\geq \eps_3\sum_{Y\in\{E,F,M\}}\Big[\frac{(r-r^\ast)^2}{r^3}(Y_1/r)^2+\frac{M^2}{r^3}(Y_2)^2+\frac{M^2(r-r^\ast)^2}{r^5}(Y_3)^2\Big]\\
&+\eps_3\frac{M}{r^4}\big(\phi^2+\psi^2\big)-\eps_3^{-1}\frac{M}{r^4}\mathbf{1}_{[DM,\infty)}(r)\big(\phi^2+\psi^2\big)+\widetilde{L},\\
\end{split}
\end{equation}
where 
\begin{equation}\label{other4.99}
\begin{split}
\widetilde{L}:=&\frac{8\Delta(r^2-4Mr)}{r^7}a_1(r)\psi^2+(1-2C_1\eps_1)\mathbf{1}_{[r^\ast,\infty)}(r)\Big\{\frac{M}{r^4}\Big(7-\frac{44M}{r}+\frac{72M^2}{r^2}\Big)\psi^2\\
&+\frac{8a_1(r)(r-r^\ast)}{r^4}\frac{(\cos\theta)^2}{(\sin\theta)^2}\psi^2+\frac{2a_1(r)(r-r^\ast)}{r^4}(F_1)^2+2a'_1(r)\frac{\Delta^2}{r^4}(F_2)^2\Big\},
\end{split}
\end{equation}
provided that $a/M$ and $(\r_h-c_0)/M$ are very small relative to $\eps_3$. Moreover
\begin{equation}\label{other5.9}
\begin{split}
2(\widetilde{P}_{(1)\mu}+\widetilde{P}_{(2)\mu}) k^\mu&\geq\eps_3\sum_{Y\in\{E,F,M\}}\big[(Y_1/r)^2+(Y_2)^2(\r_h-c)/M\big]+\eps_3M^{-2}\big(\phi^2+\psi^2\big)\\
&-\eps_3^{-1}\big[(E_3)^2+(F_3)^2\big],
\end{split}
\end{equation}
along $\mathcal{N}^c_{[t_1,t_2]}$. Also
\begin{equation}\label{other15.9}
\begin{split}
2(\widetilde{P}_{(1)\mu}+\widetilde{P}_{(2)\mu}) n^\mu&\geq -\eps_3^{-1}\big\{\widetilde{e}_0+\mathbf{1}_{[8M,2DM]}(r)\big[(E_3)^2+(F_3)^2\big]\big\}\\
&-\frac{\chi_{\geq 8M}(r)(1-p)}{r^2}\partial_2(r\phi^2+r\psi^2)+\eps_3\big[(E_2)^2+(F_2)^2\big]\mathbf{1}_{(c_0,17M/8]}(r),
\end{split}
\end{equation}
and
\begin{equation}\label{other15.29}
\begin{split}
2(\widetilde{P}_{(1)\mu}&+\widetilde{P}_{(2)\mu}) n^\mu\leq \eps_3^{-1}\big\{\widetilde{e}_0+\mathbf{1}_{[8M,2DM]}(r)\big[(E_3)^2+(F_3)^2\big]\big\}\\
&-\frac{\chi_{\geq 8M}(r)(1-p)}{r^2}\partial_2(r\phi^2+r\psi^2)+\eps_3^{-1}\big[(E_2)^2+(F_2)^2\big]\mathbf{1}_{(c_0,17M/8]}(r),\\
\end{split}
\end{equation}
where
\begin{equation}\label{other9.9}
\begin{split}
\widetilde{e}_0&=\frac{(E_1)^2+(F_1)^2+(M_1)^2}{r^2}+(L\phi)^2+(L\psi)^2\\
&+\frac{M^2|r-\r_h|}{r^3}\big[(E_2)^2+(F_2)^2\big]+\frac{M^2}{r^2}\big[(E_3^2)+(F_3)^2\big]+\frac{1}{r^2}(\phi^2+\psi^2).
\end{split}
\end{equation}
Finally,
\begin{equation}\label{other15.49}
\begin{split}
\big|L^5_{(1)}&\big|+\big|L^5_{(2)}\big|\leq\frac{\eps_3^{-1}aM|r-r^\ast|}{r^5}\big(|\phi|+|\psi|\big)\\
&\times\Big\{\sum_{Y\in\{E,F\}}\Big(\frac{|Y_1|}{r}+\frac{M(|Y_2|+|Y_3|)}{r}\Big)+\frac{1}{r}\big(|\phi|+|\psi|\big)\Big\}
+\eps_3^{-1}\big[e(\phi,\mathcal{N}_\phi)+e(\psi,\mathcal{N}_\psi)\big].
\end{split}
\end{equation}
\end{lemma} 

\begin{proof}
The order of the constants to keep in mind is
\begin{equation}\label{ConstOrder}
\max\big(a/M,(\r_h-c_0)/M\big)\ll\eps_3\ll\min(\eps_1,\eps_2)\leq\max(\eps_1,\eps_2)\ll C_1^{-1}\ll 1.
\end{equation}
Most of the proof follows in the same way as in Lemma \ref{flux}, using the identities/inequalities \eqref{Qy1}--\eqref{Qy2}, \eqref{moro31.77}, \eqref{moro43.4}, and \eqref{moro31.79}

The term $\widetilde{L}$ is new, when compared to the corresponding inequality \eqref{other4} in the case of the pure wave equation. 
It is necessary to have this term because of the term $L_{(1)}^4$ in \eqref{moro43.3}, which leads to the term 
\begin{equation*}
\frac{8\Delta(r^2-4Mr)}{r^7}a_1(r)\psi^2
\end{equation*} 
in \eqref{moro31.79}. This term is clearly nonnegative if $r\leq r^\ast$ or $r\geq 4M$; however, for $r\in [r^\ast,4M]$ we need an additional multiplier to control this term. The other terms in \eqref{other4.99} are coming from corresponding terms in \eqref{moro31.79} and \eqref{moro31.77}, and their role is to help $\widetilde{L}$ become positive. We show how to control this term below.
\end{proof}

\subsubsection{The new multiplier $(X_{(4)},w_{(4)},m_{(4)},m'_{(4)})$} We define, with $a_1$ as in \eqref{moro35.7},
\begin{equation}\label{clan1}
\begin{split}
&X_{(4)}:=0,\qquad w_{(4)}=0,\qquad m_{(4)}=0,\\
&\widetilde{m}'_{(4)1}(r,\theta):=-(1-2C_1\eps_1)
\frac{8(r-r^\ast)a_1(r)\chi_{\leq 6R}(r)}{r^2}\frac{\cos\theta}{\sin\theta}\mathbf{1}_{[r^\ast,\infty)}(r),\\
&\widetilde{m}'_{(4)2}(r):=(1-2C_1\eps_1)\frac{2b(r)}{\Delta},\qquad \widetilde{m}'_{(4)3}:=0,\qquad\widetilde{m}'_{(4)4}:=0,
\end{split}
\end{equation}
for some function $b$ supported in $[r^\ast,4M]$ to be fixed. We prove the following:

\begin{lemma}\label{clan2}
Letting $L^j_{(4)}:=L^j[X_{(4)},w_{(4)},m_{(4)},m'_{(4)}]$, $j\in\{1,2,3,4,5\}$, we have
\begin{equation}\label{clan2.1}
L^1_{(4)}=L^4_{(4)}=L^5_{(4)}=0
\end{equation}
and, for some constant $C_2$ sufficiently large,
\begin{equation}\label{clan2.2}
\widetilde{L}+L^2_{(4)}+L^3_{(4)}\geq -C_2(a+|\r_h-c_0|)r^{-4}(\phi^2+\psi^2).
\end{equation}
Moreover,
\begin{equation}\label{clan2.3}
\big|2\widetilde{P}_{(4)\mu}n^\mu\big|\lesssim \eps_3^{-1}\psi^2/r^2\qquad \text{ and }\qquad 2\widetilde{P}_{(4)\mu}k^\mu=0\,\,\text{ along }\,\,\mathcal{N}^c_{[t_1,t_2]}.
\end{equation}
\end{lemma}

\begin{proof} The identities in \eqref{clan2.1} are clear. The inequality in \eqref{clan2.2} is also clear in the regions $\{r\leq r^\ast\}$ and $\{r\geq 12M\}$. 

Using the formula \eqref{divop} we calculate, in the region $\{r\in[r^\ast,12M]\}$,
\begin{equation*}
\begin{split}
&\frac{1}{2}\D^\mu \widetilde{m}'_{(4)\mu}=(1-2C_1\eps_1)\Big[\frac{4(r-r^\ast)a_1(r)\chi_{\leq 6R}(r)}{q^2r^2}+\frac{b'(r)}{q^2}\Big],\\
&\psi\widetilde{m}'^\mu_{(4)}\D_\mu\psi=(1-2C_1\eps_1)\Big[-\frac{8(r-r^\ast)a_1(r)\chi_{\leq 6R}(r)}{q^2r^2}\frac{\cos\theta}{\sin\theta}\psi\D_1\psi+\frac{2b(r)}{q^2}\psi\D_2\psi\Big].
\end{split}
\end{equation*}
Therefore, in the region $\{r\in[r^\ast,12M]\}$,
\begin{equation*}
\begin{split}
L^2_{(4)}&+L^3_{(4)}+\widetilde{L}=\frac{8\Delta(r^2-4Mr)}{r^7}a_1(r)\psi^2+(1-2C_1\eps_1)\Big\{\frac{M}{r^4}\Big(7-\frac{44M}{r}+\frac{72M^2}{r^2}\Big)\psi^2\\
&+\frac{8a_1(r)(r-r^\ast)}{r^4}\frac{(\cos\theta)^2}{(\sin\theta)^2}\psi^2+\frac{2a_1(r)(r-r^\ast)}{r^4}(F_1)^2+2a'_1(r)\frac{\Delta^2}{r^4}F^2_2\Big\}\\
&+(1-2C_1\eps_1)\Big[\frac{4(r-r^\ast)a_1(r)\chi_{\leq 6R}(r))}{q^2r^2}+\frac{b'(r)}{q^2}\Big]\psi^2\\
&+(1-2C_1\eps_1)\Big[-\frac{8(r-r^\ast)a_1(r)\chi_{\leq 6R}(r)}{q^2r^2}\frac{\cos\theta}{\sin\theta}\psi\D_1\psi+\frac{2b(r)}{q^2}\psi\D_2\psi\Big].
\end{split}
\end{equation*}
Recalling \eqref{ell13.9}, we may replace $\D_1\psi$ and $\D_2\psi$ with $F_1$ and $F_2$, up to acceptable errors. Then we divide by $(1-2C_1\eps_1)$ and complete squares. For \eqref{clan2.2} it suffices to prove that
\begin{equation*}
\begin{split}
-C_2a&\leq \frac{8\Delta(r^2-4Mr)a_1(r)}{r^7(1-2C_1\eps_1)}+\frac{M}{r^4}\Big(7-\frac{44M}{r}+\frac{72M^2}{r^2}\Big)\\
&+\Big[\frac{4(r-r^\ast)a_1(r)\chi_{\leq 6R}(r))}{r^4}+\frac{b'(r)}{r^2}\Big]-\frac{b(r)^2}{2\Delta^2a'_1(r)},
\end{split}
\end{equation*}
 for any $r\in[r^\ast,12M]$, for some function $b$ supported in $[r^\ast,4M]$ to be fixed. After algebraic simplifications, it suffices to prove that, for any $r\in[r^\ast,4M]$,
\begin{equation}\label{moro56}
0\leq \frac{M}{r^4}\Big(7-\frac{44M}{r}+\frac{72M^2}{r^2}\Big)+\frac{8\Delta(r-4M)a_1(r)}{r^6(1-2C_1\eps_1)}+\Big[\frac{4(r-r^\ast)a_1(r)}{r^4}+\frac{b'(r)}{r^2}\Big]-\frac{b(r)^2}{2a'_1(r)\Delta^2}.
\end{equation}
We multiply both sides of \eqref{moro56} by $r^6/M^3$. It suffices to find a function $b$ supported in $[r^\ast,4M]$ such that, for $r\in[r^\ast,4M]$,
\begin{equation}\label{moro57}
1\lesssim\frac{r^4b'(r)}{M^3}+\Big(\frac{7r^2}{M^2}-\frac{44r}{M}+72\Big)-\frac{r^4b(r)^2}{2M^3a'_1(r)(r-2M)^2}+4a_1(r)\Big(\frac{3r^3}{M^3}-\frac{15r^2}{M^2}+\frac{16r}{M}\Big).
\end{equation}

Let
\begin{equation*}
r=(3+s)M,\qquad \widetilde{b}(s):=b((3+s)M).
\end{equation*}
Notice also that, for $s\in[0,1]$,
\begin{equation*}
\big|a_1((3+s)M)-\widetilde{a}_1(s)\big|+\big|Ma'_1((3+s)M)-\widetilde{a}'_1(s)\big|\lesssim a,
\end{equation*}
where
\begin{equation}\label{moro58}
\widetilde{a}_1(s):=\frac{5s+s^2+6\log(1+s)}{(3+s)^2},\qquad\widetilde{a}'_1(s):=\frac{33+s-12\log(1+s)-12\frac{s}{s+1}}{(3+s)^3}.
\end{equation}
For \eqref{moro57} it suffices to prove that, for $s\in[0,1]$,
\begin{equation}\label{moro59}
1\lesssim \widetilde{b}'(s)-\frac{\widetilde{b}(s)^2}{2\widetilde{a}'_1(s)(1+s)^2}+\frac{7s^2-2s+3}{(3+s)^4}+4\widetilde{a}_1(s)\frac{3s^2+3s-2}{(3+s)^3}.
\end{equation}

Notice that $\widetilde{a}'_1(s)(1+s)^2\geq 1$ for any $s\in[0,1]$. Indeed, using \eqref{moro58},
\begin{equation*}
\begin{split}
(3+s)^3\big[\widetilde{a}'_1(s)(1+s)^2-1]&=(1+s)[33+10s-11s^2+12(1+s)(s-\log(1+s))]-(3+s)^3\\
&=12(1+s)^2(s-\log(1+s))+6+16s-10s^2-12s^3\\
&\geq 0.
\end{split}
\end{equation*}
Therefore, for \eqref{moro59} it suffices to prove that, for $s\in[0,1]$,
\begin{equation}\label{moro60}
1\lesssim \widetilde{b}'(s)-\frac{\widetilde{b}(s)^2}{2}+\frac{7s^2-2s+3}{(3+s)^4}+4\widetilde{a}_1(s)\frac{3s^2+3s-2}{(3+s)^3}.
\end{equation}

Moreover, for $s\in[0,1]$,
\begin{equation*}
\begin{split}
\frac{7s^2-2s+3}{(3+s)^4}+4\widetilde{a}_1(s)\frac{3s^2+3s-2}{(3+s)^3}&=\frac{9-91s+167s^2+115s^3-24s^4}{(3+s)^5}\\
&+\frac{24(-2+3s+3s^2)[\log(1+s)-s+s^2/2]}{(3+s)^5}\\
&\geq \frac{9-91s+167s^2+91s^3}{(3+s)^5}\\
&\geq \frac{9(1-10s+18s^2)}{(3+s)^5}\mathbf{1}_{[1/10,1]}(s)+\frac{4s^2}{(3+s)^5}+10^{-10}.
\end{split}
\end{equation*}
Therefore, to prove \eqref{moro60} it suffices to find a function $\widetilde{b}$ supported in $[1/10,1]$ such that
\begin{equation}\label{moro61}
\widetilde{b}'(s)+\frac{9(1-10s+18s^2)}{(3+s)^5}\geq 0\quad\text{ and }\quad|\widetilde{b}(s)|\leq \frac{\sqrt{2}s}{16}
\end{equation}
for any $s\in[1/10,1]$. 

Notice that $1-10s+18s^2=18(s-s_1)(s-s_2)$ where $s_1=(5-\sqrt{7})/18$, $s_2=(5+\sqrt{7})/18$. We define $\widetilde{b}(s)=0$ for $s\leq s_1$ and
\begin{equation*}
\widetilde{b}(s):=\int_{s_1}^s\frac{9(10\rho-1-18\rho^2)}{3^5}\,d\rho
\end{equation*}
for $s\in[s_1,s_2]$. The desired inequalities \eqref{moro61} are easy to verify for $s\in[1/10,s_2]$, and, 
moreover, $\widetilde{b}(s_2)=7^{3/2}9^{-4}\leq 3\cdot 10^{-3}$. 

On the other hand, for $s\geq s_2$, we would like to define the function $\widetilde{b}$ decreasing, still satisfying \eqref{moro61}, and vanishing for $s\geq 1$. The only condition for this to be possible is the inequality
\begin{equation*}
\int_{s_2}^1\frac{9(1-10\rho+18\rho^2)}{4^5}\,d\rho\geq \widetilde{b}(s_2),
\end{equation*}
which is easy to verify. This completes the proof of the main inequality \eqref{clan2.2}.

The identity and the inequality in \eqref{clan2.3} follow from definitions.
\end{proof}

As a consequence of Lemma \ref{flux2} and Lemma \ref{clan2} we have:

\begin{corollary}\label{flux4}
There is a sufficiently small absolute constant $\eps_3>0$ with the property that
\begin{equation}\label{ext1}
\begin{split}
\sum_{j=1}^4\big(L^j_{(1)}+L^j_{(2)}+L^j_{(4)}\big)&\geq \eps_3\sum_{Y\in\{E,F,M\}}\Big[\frac{(r-r^\ast)^2}{r^3}(Y_1/r)^2+\frac{M^2}{r^3}(Y_2)^2+\frac{M^2(r-r^\ast)^2}{r^5}(Y_3)^2\Big]\\
&+\eps_3\frac{M}{r^4}\big(\phi^2+\psi^2\big)-\eps_3^{-1}\frac{M}{r^4}\mathbf{1}_{[DM,\infty)}(r)\big(\phi^2+\psi^2\big),
\end{split}
\end{equation}
and
\begin{equation}\label{ext2}
\begin{split}
2(\widetilde{P}_{(1)\mu}+\widetilde{P}_{(2)\mu}+\widetilde{P}_{(4)\mu}) k^\mu&\geq\eps_3\sum_{Y\in\{E,F,M\}}\big[(Y_1/r)^2+(Y_2)^2(\r_h-c)/M\big]\\
&+\eps_3M^{-2}\big(\phi^2+\psi^2\big)-\eps_3^{-1}\big[(E_3)^2+(F_3)^2\big],
\end{split}
\end{equation}
along $\mathcal{N}^c_{[t_1,t_2]}$. Moreover, with $\widetilde{e}_0$ as in \eqref{other9.9},
\begin{equation}\label{ext3}
\begin{split}
2(\widetilde{P}_{(1)\mu}&+\widetilde{P}_{(2)\mu}+\widetilde{P}_{(4)\mu}) n^\mu\geq -\eps_3^{-1}\big\{\widetilde{e}_0+\mathbf{1}_{[8M,2DM]}(r)\big[(E_3)^2+(F_3)^2\big]\big\}\\
&-\frac{\chi_{\geq 8M}(r)(1-p)}{r^2}\partial_2(r\phi^2+r\psi^2)+\eps_3\big[(E_2)^2+(F_2)^2\big]\mathbf{1}_{(c_0,17M/8]}(r),
\end{split}
\end{equation}
and
\begin{equation}\label{ext4}
\begin{split}
2(\widetilde{P}_{(1)\mu}&+\widetilde{P}_{(2)\mu}+\widetilde{P}_{(4)\mu}) n^\mu\leq \eps_3^{-1}\big\{\widetilde{e}_0+\mathbf{1}_{[8M,2DM]}(r)\big[(E_3)^2+(F_3)^2\big]\big\}\\
&-\frac{\chi_{\geq 8M}(r)(1-p)}{r^2}\partial_2(r\phi^2+r\psi^2)+\eps_3^{-1}\big[(E_2)^2+(F_2)^2\big]\mathbf{1}_{(c_0,17M/8]}(r).
\end{split}
\end{equation}
Finally,
\begin{equation}\label{ext5}
\begin{split}
\big|L^5_{(1)}&\big|+\big|L^5_{(2)}\big|+\big|L^5_{(4)}\big|\leq\frac{\eps_3^{-1}aM|r-r^\ast|}{r^5}\big(|\phi|+|\psi|\big)\\
&\times\Big\{\sum_{Y\in\{E,F\}}\Big(\frac{|Y_1|}{r}+\frac{M(|Y_2|+|Y_3|)}{r}\Big)+\frac{1}{r}\big(|\phi|+|\psi|\big)\Big\}
+\eps_3^{-1}\big[e(\phi,\mathcal{N}_\phi)+e(\psi,\mathcal{N}_\psi)\big].
\end{split}
\end{equation}
\end{corollary}

These inequalities should be compared with the inequalities \eqref{other4} and the corresponding inequalities in Lemma \ref{flux}.

\subsubsection{Outgoing energies} Finally, as in subsection \ref{outgoingEn}, we define $(X_{(3)},w_{(3)},m_{(3)},m'_{(3)})$ by
\begin{equation}\label{outs.9}
\begin{split}
&X_{(3)}:=f_3\partial_2+\Big(\frac{f_3}{1-\widetilde{p}}+g_3\Big)\partial_3,\qquad w_{(3)}:=\frac{2f_3}{r},\qquad m'_{(3)}:=m_{(3)},\\
&m_{(3)1}:=m_{(3)4}:=0,\quad m_{(3)2}:=\frac{2h_3}{r(1-\widetilde{p})},\quad m_{(3)3}:=-\frac{2h_3}{r},
\end{split}
\end{equation}
where $\widetilde{p}:=2M/r$, and $f_3,g_3$ are defined by
\begin{equation}\label{outs40.7}
f_3(r):=\eps_4\chi_{\geq 8M}(r)e^{\beta(r)},\qquad g_3(r):=\int_{r}^\infty \Big[\rho(s)+\frac{\eps_4M^2}{s^3}f_3(s)\Big]\,ds,
\end{equation}
where 
\begin{equation}\label{outs50.7}
\beta(8M):=0,\qquad \beta'(r):=\Big(\frac{4M}{r^2}+\frac{1}{r}\Big)\big(1-\chi_{\geq C_4^4M}(r)\big)+\frac{\alpha}{r}\chi_{\geq C_4^4M}(r),
\end{equation}
and
\begin{equation}\label{outs51.7}
\rho(r):=\delta M^{-1}\Big[\chi_{\geq C_4M}(r)+\chi_{\geq 4C_4^4M}(r)\Big(C_4^7e^{\beta(r)}\frac{M^3}{r^3}-1\Big)\Big].
\end{equation}
The constants $\eps_4,C_4$ satisfy $\eps_4=\eps_3^2$ and $C_4\geq \eps_4^{-4}\alpha^{-1}(2-\alpha)^{-1}$, while $\delta\in[10^{-4}C_4^{-3},10^4C_4^{-3}]$ is such that $\int_{C_4M}^\infty\rho(s)\,ds=C_4$. Recall \eqref{outs52},
\begin{equation}\label{Alx3}
e^{\beta(r)}\approx \frac{r}{M}\text{ if }r\leq 10 C_4^4M\quad \text{ and }\quad e^{\beta (r)}\approx C_4^4\Big(\frac{r}{C_4^4M}\Big)^\alpha\text{ if }r\geq (1/10) C_4^4M.
\end{equation}
Notice the additional term $M^2s^{-3}f_3(s)$ in the definition of the function $g_3$; this term is needed in order to be able to estimate the contributions of the new terms containing the small coefficient $a$, in a way that is uniform as $\alpha\to 0$ or $\alpha\to 2$.

Also let
\begin{equation}\label{outs3.7}
H_3:=(1-\widetilde{p})f'_3-\frac{2Mf_3}{r^2}-(1-\widetilde{p})^2\rho-\frac{\eps_4M^2f_3}{r^3}(1-\widetilde{p})^2,\qquad h_3:=H_3\cdot (1-\widetilde{\alpha}),
\end{equation}
where $\widetilde{\alpha}:=(2-\alpha)/10$. Recall the bounds \eqref{outs41} and \eqref{outs41.9},
\begin{equation}\label{outs41.7}
\begin{split}
\beta(r)\in[-10,0]\text{ and }M\beta'(r)\in[1/10,10]&\qquad\text{ if }r\in(c,8M],\\
\max\Big(\frac{\alpha}{100r},\frac{4M}{r^2}+\frac{1}{r}\mathbf{1}_{[8M,C_4M]}(r)\Big)\leq \beta'(r)\leq \frac{2}{r}&\qquad\text { if }r\in[8M,\infty),\\
\rho(r)=0\text{ and }g_3(r)\in[C_4/2,2C_4]&\qquad\text{ if }r\leq C_4M,\\
\rho(r)\leq \frac{\eps_4}{100}\beta'(r)e^{\beta(r)}\text{ and }\rho'(r)\leq\frac{\eps_4M}{100r^3}e^{\beta(r)}&\qquad\text{ if }r\geq C_4M,\\
\frac{e^\beta M^2}{r^2}\leq g_3(r)\leq\frac{C_4^{10}e^\beta M^2}{r^2}&\qquad\text{ if }r\geq C_4M,\\
(1-2\widetilde{\alpha})H_3(r)-rH'_3(r)\geq 0&\qquad\text{ if }r\in [16M,\infty),
\end{split}
\end{equation}
\begin{equation}\label{outs41.97}
\begin{split}
&g'_3=-\rho-\eps_4M^2r^{-3}f_3,\\
&\big|H_3-(1-\widetilde{p})f'_3\big|\leq\frac{(2+\eps_4)Mf_3}{r^2}+\rho,\\
&e^{\beta(r)}\in[r/(100M),r^2/M^2]\qquad\text{ for }r\in (c,C_4M],
\end{split}
\end{equation}
and
\begin{equation}\label{outs42.7}
\begin{split}
\frac{2f_3-rf'_3}{r}&=\eps_4 e^{\beta}\big[(2/r-\beta')\chi_{\geq 8M}-\chi'_{\geq 8M}\big],\\
\frac{6Mf_3}{r^4}-\frac{2Mf'_3}{r^3}+\frac{(1-\widetilde{p})^2g''_3}{r}+\frac{4M(1-\widetilde{p})g'_3}{r^3}&\geq\frac{\eps_4 M}{100r^4}e^\beta\chi_{\geq 8M}-\frac{2\eps_4 M}{r^3}e^\beta\chi'_{\geq 8M}.
\end{split}
\end{equation}

Notice that
\begin{equation}\label{Alx1}
\g^{33}=-\frac{r^2+a^2}{\Delta}+O(a^2Mr^{-3})\qquad\text{ if }r\geq 5M/2.
\end{equation}

\begin{proof}[Proof of Theorem \ref{MainTheorem}] Let $L^j_{(3)}:=L^j[X_{(3)},w_{(3)},m_{(3)},m'_{(3)}]$, $j\in\{1,2,3,4,5\}$. As in the proof of \eqref{outs7}, we have 
\begin{equation*}
L^1_{(3)}=\sum_{Y\in \{E,F,M\}}\big[K^{11}_{(3)}(Y_1)^2+K^{22}_{(3)}(Y_2)^2+K^{33}_{(3)}(Y_3)^2+2K^{23}_{(3)}Y_2Y_3\big],
\end{equation*}
where, with $O':=O[a^2r^{-2}(f_3/r+f'_3)]$,
\begin{equation*}
\begin{split}
&K_{(3)}^{11}=\frac{-f'_3(r)}{q^2}+w_{(3)}(r)\g^{11}=\frac{2f_3-rf'_3}{rq^2},\\
&K_{(3)}^{22}=\frac{-f_3(r)(2r-2M)+f'_3(r)\Delta}{q^2}+w_{(3)}(r)\g^{22}=(1-\widetilde{p})f'_3-\frac{2Mf_3}{r^2}+O',\\
&K_{(3)}^{33}=-f_3(r)\partial_2\g^{33}-f'_3(r)\g^{33}-\frac{2rf_3(r)\g^{33}}{q^2}+w_{(3)}(r)\g^{33}= \frac{f'_3}{1-\widetilde{p}}-\frac{2Mf_3}{r^2(1-\widetilde{p})^2}+O',\\
&K_{(3)}^{23}=\Big(\frac{f_3}{1-\widetilde{p}}+g_3\Big)'\frac{\Delta}{q^2}=f'_3-\frac{2Mf_3}{r^2(1-\widetilde{p})}-\rho(1-\widetilde{p})-\frac{\eps_4M^2 f_3}{r^3}(1-\widetilde{p})+O'.
\end{split}
\end{equation*}
Therefore
\begin{equation*}
\begin{split}
L^1_{(3)}&\geq \sum_{Y\in \{E,F,M\}}\Big\{\frac{2f_3-rf'_3}{rq^2}(Y_1)^2+H_3\Big[Y_2+\frac{Y_3}{1-\widetilde{p}}\Big]^2+\Big[\rho+\frac{\eps_4M^2f_3}{2r^3}\Big][(1-\widetilde{p})^2(Y_2)^2+(Y_3)^2]\Big\}\\
&-aMr^{-3}e^{\beta(r)}\chi_{\geq 5M}(r)[(Y_2)^2+(Y_3)^2].
\end{split}
\end{equation*}
Also, using also \eqref{divop}, \eqref{waveop}, the definitions \eqref{moro3.2}, and Lemma \ref{clan0},
\begin{equation*}
\begin{split}
L^2_{(3)}&\geq \frac{2h_3}{r}\phi\Big[\D_2\phi+\frac{\D_3\phi}{1-\widetilde{p}}\Big]+\frac{2h_3}{r}\psi\Big[\D_2\psi+\frac{\D_3\psi}{1-\widetilde{p}}\Big]\\
&-aMr^{-4}e^{\beta(r)}\chi_{\geq 5M}(r)\big[|\phi||\D_2\phi|+|\phi||\D_3\phi|+|\psi||\D_2\psi|+|\psi||\D_3\psi|\big],
\end{split}
\end{equation*}
\begin{equation*}
L^3_{(3)}\geq (\phi^2+\psi^2)\Big[\frac{h_3}{r^2}+\frac{h'_3}{r}+\frac{2Mf_3}{r^4}-\frac{2Mf'_3}{r^3}-\frac{(1-\widetilde{p})f''_3}{r}\Big]-(\phi^2+\psi^2)ar^{-4}e^{\beta(r)}\chi_{\geq 5M}(r),
\end{equation*}
and
\begin{equation*}
L^4_{(3)}\geq \frac{8(1-\widetilde{p})}{r^3}(f_3-rf'_3)\psi^2-\psi^2ar^{-4}e^{\beta(r)}\chi_{\geq 5M}(r).
\end{equation*}
We combine now the $M_2^2$ term in the right-hand side of $L_{(3)}^1$ and $L^4_{(3)}$. Recalling also the definition and \eqref{ell13.9} we have $(M_2)^2\geq 4r^{-2}\psi^2-(\phi^2+\psi^2)aMr^{-4}$. Therefore, 
\begin{equation*}
H_3(M_2)^2+L^4_{(3)}\geq -(\phi^2+\psi^2)ar^{-4}e^{\beta(r)}\chi_{\geq 5M}(r),
\end{equation*}
using the second inequality in \eqref{outs41.97} and the definitions.

We add up the estimates above and complete the square to conclude that
\begin{equation*}
\begin{split}
L^1_{(3)}&+L^2_{(3)}+L^3_{(3)}+L^4_{(3)}\geq \sum_{Y\in\{E,F,M\}}\frac{2f_3-rf'_3}{2r^3}(Y_1)^2\\
&+H_3\Big(E_2+\frac{E_3}{1-\widetilde{p}}+\frac{(1-\widetilde{\alpha})\phi}{r}\Big)^2+H_3\Big(F_2+\frac{F_3}{1-\widetilde{p}}+\frac{(1-\widetilde{\alpha})\psi}{r}\Big)^2\\
&+\Big[\rho+\frac{\eps_4M^2f_3}{2r^3}\Big][(1-\widetilde{p})^2(E_2)^2+(E_3)^2+(1-\widetilde{p})^2(F_2)^2+(F_3)^2]\\
&+(\phi^2+\psi^2)\Big[\frac{(\widetilde{\alpha}-\widetilde{\alpha}^2)H_3-\widetilde{\alpha}rH'_3}{r^2}+\frac{H'_3}{r}+\frac{2Mf_3}{r^4}-\frac{2Mf'_3}{r^3}-\frac{(1-\widetilde{p})f''_3}{r}\Big]\\
&-\eps_3^{-1}ar^{-4}e^{\beta(r)}\chi_{\geq 5M}(r)(\phi^2+\psi^2).
\end{split}
\end{equation*}
Combining this with \eqref{ext1} and estimating as in the proof of Lemma \ref{BigBounds} we conclude that
\begin{equation}\label{Alx10}
\begin{split}
\sum_{j=1}^4\big(L^j_{(1)}&+L^j_{(2)}+L^j_{(4)}+L^j_{(3)}\big)\geq \sum_{Y\in\{E,F,M\}}\eps_4^2\Big(\frac{e^{\beta}(2-r\beta')}{r}+\frac{100}{r}\Big)\frac{(r-r^\ast)^2}{r^2}\frac{(Y_1)^2}{r^2}\\
&+\eps_4^2\Big(\frac{\widetilde{\alpha}^2e^{\beta}\beta'}{r^2}+\frac{M e^{\beta}}{r^4}\Big)\big(\phi^2+\psi^2\big)+\sum_{Y\in\{E,F\}}\eps_4^2\frac{M^2e^{\beta}}{100r^3}\Big[(Y_2)^2+\frac{(r-r^\ast)^2}{r^2}(Y_3)^2\Big]\\
&+\eps_4^2e^{\beta}\beta'\Big[\Big(E_2+\frac{E_3}{1-\widetilde{p}}+\frac{(1-\widetilde{\alpha})\phi}{r}\Big)^2+\Big(F_2+\frac{F_3}{1-\widetilde{p}}+\frac{(1-\widetilde{\alpha})\psi}{r}\Big)^2\Big],
\end{split}
\end{equation}
provided that $D$ is taken large enough and $\eps_4$ is sufficiently small. 

Moreover, using Lemma \ref{clan0},
\begin{equation*}
\begin{split}
|L^5_{(3)}| &\leq \frac{aM}{r^4}\eps_4e^{\beta}\chi_{\geq 8M}\big(|\phi|+|\psi|\big)\Big\{\sum_{Y\in\{E,F\}}\frac{|Y_1|+M|Y_2|+M|Y_3|}{r}+\frac{1}{r}\big(|\phi|+|\psi|\big)\Big\}\\
&+e^{\beta}e(\phi,\mathcal{N}_\phi)+e^{\beta}e(\psi,\mathcal{N}_{\psi}).
\end{split}
\end{equation*}
Combining this with \eqref{ext5}, \eqref{Alx10}, and \eqref{ell13.9}, we obtain the final lower bound on the space-time term, for some small constant $\eps_5=\eps_5(\alpha)$,
\begin{equation}\label{Alx12}
\begin{split}
\sum_{j=1}^5\big(L^j_{(1)}&+L^j_{(2)}+L^j_{(4)}+L^j_{(3)}\big)\geq \eps_5e^{\beta}\Big\{\frac{(r-r^\ast)^2}{r^2}\frac{(\partial_1\phi)^2+(\partial_1\psi)^2+(\psi/\sin\theta)^2}{r^3}\\
&+\frac{M^2}{r^3}\big[(\partial_2\phi)^2+(\partial_2\psi)^2\big]+\frac{M^2(r-r^\ast)^2}{r^5}\big[(\partial_3\phi)^2+(\partial_3\psi)^2\big]\\
&+\frac{\phi^2+\psi^2}{r^3}+\frac{(L\phi)^2+(L\psi)^2}{r}\Big\}-e^{\beta}\big[e(\phi,\mathcal{N}_\phi)+e(\psi,\mathcal{N}_{\psi})\big].
\end{split}
\end{equation}

We consider now the contribution of $\widetilde{P}_{(3)\mu}n^\mu$. Using \eqref{Qy1} and the definitions we write
\begin{equation*}
\begin{split}
2\widetilde{P}_{(3)\mu}n^\mu&=2Q_{\mu\nu}X_{(3)}^\nu n^\mu+w_{(3)}(\phi E_\mu+\psi F_\mu)n^\mu-\frac{n^\mu\D_\mu w_{(3)}}{2}(\phi^2+\psi^2)\\
&+\frac{n^\mu}{2}(m_{(3)\mu}\phi^2+m'_{(3)\mu}\psi^2)-2\frac{X_{(3)}^\nu\D_\nu A}{A}\frac{n^\mu\D_\mu A}{A}\psi^2\\
&=\frac{m_{(3)3}(-\g^{33})}{2}(\phi^2+\psi^2)+\sum_{Y\in\{E,F,M\}}\Big[\frac{(Y_1)^2}{q^2}\Big(\frac{f_3}{1-\widetilde{p}}+g_3\Big)+\frac{\Delta(Y_2)^2}{q^2}\Big(\frac{f_3}{1-\widetilde{p}}+g_3\Big)\\
&+(Y_3)^2(-\g^{33})\Big(\frac{f_3}{1-\widetilde{p}}+g_3\Big)+2Y_2Y_3(-\g^{33})f_3\Big]+\frac{2f_3}{r}(-\g^{33})(\phi E_3+\psi F_3).
\end{split}
\end{equation*}
As before, the main point is that the function $g_3$ is extremely large when $r$ is small. We can combine this last identity with the bounds \eqref{ext3} and \eqref{ext4}, as in the proof of Lemma \ref{BigBounds} to conclude that, for any $t\in[0,T]$,
\begin{equation}\label{Alx11}
\begin{split}
\int_{\Sigma_{t}^c}2\big[\widetilde{P}_{(1)\mu}+\widetilde{P}_{(2)\mu}+\widetilde{P}_{(3)\mu}+\widetilde{P}_{(4)\mu} \big]n_0^\mu\,d\mu_{t}\approx_{\alpha} \int_{\Sigma_{t}^c}&e^\beta\big[e(\phi)^2+e(\psi)^2\big]\,d\mu_t.
\end{split}
\end{equation}

Finally, using \eqref{Qy2}, the contribution of $\widetilde{P}_{(3)\mu}k^\mu$ along $\mathcal{N}^c_{[0,T]}$ is 
\begin{equation*}
\begin{split}
2\widetilde{P}_{(3)\mu}k^\mu&=2Q_{\mu\nu}X_{(3)}^\nu k^\mu+w_{(3)}(\phi E_\mu+\psi F_\mu)k^\mu-\frac{k^\mu\D_\mu w_{(3)}}{2}(\phi^2+\psi^2)\\
&+\frac{k^\mu}{2}(m_{(3)\mu}\phi^2+m'_{(3)\mu}\psi^2)-2\frac{X_{(3)}^\nu\D_\nu A}{A}\frac{k^\mu\D_\mu A}{A}\psi^2\\
&=\sum_{Y\in\{E,F\}}\big[2g_3(c)\g^{23}(Y_3)^2+2Y_2Y_3g_3(c)\g^{22}\big].
\end{split}
\end{equation*}
Combining with \eqref{ext2} we obtain
\begin{equation}\label{Alx14}
2\big[\widetilde{P}_{(1)\mu}+\widetilde{P}_{(2)\mu}+\widetilde{P}_{(3)\mu}+\widetilde{P}_{(4)\mu}\big]k^\mu\geq 0\qquad \text{ along }\mathcal{N}^c_{[0,T]}.
\end{equation}
The theorem follows from \eqref{Alx12}, \eqref{Alx11}, \eqref{Alx14}, and the divergence identity \eqref{red2}.
\end{proof}

\section{Proof of Corollary \ref{MainCoro}}\label{CoRoProof}

In this section we provide a proof of Corollary \ref{MainCoro}. The main issue is the degeneracy of the weights in the bulk term at $r=r^\ast$. We compensate for this by losing derivatives. More precisely:

\begin{lemma}\label{NW0}
Assume that $(\phi,\psi)\in C^k([0,T]:\mathbf{H}^{6-k}(\Sigma_t^{c_0}))$, $k\in[0,6]$, is a solution of the system \eqref{asump0} with $\mathcal{N}_\phi=\mathcal{N}_\psi=0$. Then
\begin{equation}\label{NW1}
\begin{split}
\mathcal{BB}_\alpha^{c_0}(t_1,t_2)+\sum_{k=0}^2\int_{\Sigma_{t_2}^{c_0}}\frac{r^\alpha}{M^{\alpha}}\big[e(\phi_k)^2+e(\psi_k)^2\big]\,d\mu_t
&\lesssim_\alpha \sum_{k=0}^2\int_{\Sigma_{t_1}^{c_0}}\frac{r^\alpha}{M^{\alpha}}\big[e(\phi_k)^2+e(\psi_k)^2\big]\,d\mu_t,
\end{split}
\end{equation}
for any $\alpha\in(0,2)$ and any $t_1\leq t_2\in[0,T]$, where $\phi_k:=M^k\T^k\phi$, $\psi_k:=M^k\T^k\psi$, and
\begin{equation}\label{NW2}
\begin{split}
\mathcal{BB}_\alpha^{c_0}(t_1,t_2):=\int_{\mathcal{D}^{c_0}_{[t_1,t_2]}}&\frac{r^\alpha}{M^{\alpha}}\Big\{\frac{|\partial_1\phi|^2+|\partial_1\psi|^2+\psi^2(\sin\theta)^{-2}}{r^3}+\frac{1}{r}\big[(L\phi)^2+(L\psi)^2\big]\\
&+\frac{1}{r^3}\big(\phi^2+\psi^2\big)+\frac{M^{2}}{r^{3}}\big[(\partial_2\phi)^2+(\partial_2\psi)^2+(\partial_3\phi)^2+(\partial_3\psi)^2\big]\Big\}\,d\mu.
\end{split}
\end{equation}
\end{lemma}

Assuming Lemma \ref{NW0}, it is not hard to complete the proof of Corollary \ref{MainCoro}

\begin{proof}[Proof of Corollary \ref{MainCoro}] We prove the estimate in two steps. Notice first that the inequality \eqref{NW2} is equivalent to
\begin{equation*}
\begin{split}
\int_{t_1}^{t_2}\Big(\int_{\Sigma_{s}^{c_0}}\frac{r^{\alpha-1}}{M^{\alpha}}\big[e(\phi)^2+e(\psi)^2\big]\,d\mu_s\Big)\,ds&+\sum_{k=0}^2\int_{\Sigma_{t_2}^{c_0}}\frac{r^\alpha}{M^{\alpha}}\big[e(\phi_k)^2+e(\psi_k)^2\big]\,d\mu_t\\
&\lesssim_\alpha \sum_{k=0}^2\int_{\Sigma_{t_1}^{c_0}}\frac{r^\alpha}{M^{\alpha}}\big[e(\phi_k)^2+e(\psi_k)^2\big]\,d\mu_t,
\end{split}
\end{equation*}
for any $t_1\leq t_2\in[0,T]$ and $\alpha\in(0,2)$. Let
\begin{equation}\label{NW1.5}
\begin{split}
I_{\beta,l}(s):=\sum_{k=0}^l\int_{\Sigma_{s}^{c_0}}\frac{r^\beta}{M^{\beta}}\big[e(\phi_k)^2+e(\psi_k)^2\big]\,d\mu_s.
\end{split}
\end{equation}
Therefore, for any $\alpha\in (0,2)$, $l\in\{0,1,2\}$, and $t_1\leq t_2\in[0,T]$, we have
\begin{equation}\label{NW4}
\begin{split}
I_{\alpha,l+2}(t_2)+\int_{t_1}^{t_2}\frac{1}{M}I_{\alpha-1,l}(s)\,ds\lesssim_\alpha I_{\alpha,l+2}(t_1).
\end{split}
\end{equation}

We apply \eqref{NW4} first with $\alpha$ close to $2$ and $l=2,4$; the result is
\begin{equation*}
\begin{split}
&\int_{0}^{T}\frac{1}{M}I_{\alpha-1,2}(s)\,ds\lesssim_\alpha I_{\alpha,4}(0)\quad\text{ and }\quad I_{\alpha-1,2}(s')\lesssim_\alpha I_{\alpha-1,2}(s)\quad\text{ if }s\leq s'.
\end{split}
\end{equation*}
These inequalities show easily that
\begin{equation}\label{NW5}
I_{\alpha-1,2}(s)\lesssim_\alpha I_{\alpha,4}(0)\frac{M}{M+s}\quad\text{ for any }\quad s\in[0,T]\quad\text{ and }\quad \alpha\in(0,2).
\end{equation}

To apply this argument again we need to improve slightly on \eqref{NW5}. More precisely, we'd like to show that
\begin{equation}\label{NW6}
I_{1+\eps,2}(s)\lesssim_\eps I_{2,4}(0)\frac{M^{1-2\eps}}{(M+s)^{1-2\eps}}\quad\text{ for any }\quad s\in[0,T]\quad\text{ and }\quad \eps\in(0,1/10].
\end{equation}
Indeed, we estimate
\begin{equation*}
I_{1+\eps,2}(s)\lesssim II(s)+III(s),
\end{equation*}
where, using \eqref{NW5} and \eqref{NW4},
\begin{equation*}
\begin{split}
II(s):=\sum_{k=0}^l\int_{\Sigma_{s}^{c_0},\,r\leq M+s}&\frac{r^{1+\eps}}{M^{1+\eps}}\big[e(\phi_k)^2+e(\psi_k)^2\big]\,d\mu_s\\
&\lesssim I_{1-\eps/2,2}\frac{(M+s)^{7\eps/4}}{M^{7\eps/4}}\lesssim_\eps I_{2,4}(0)\frac{M^{1-2\eps}}{(M+s)^{1-2\eps}}
\end{split}
\end{equation*}
and
\begin{equation*}
\begin{split}
III(s):=\sum_{k=0}^l\int_{\Sigma_{s}^{c_0},\,r\geq M+s}&\frac{r^{1+\eps}}{M^{1+\eps}}\big[e(\phi_k)^2+e(\psi_k)^2\big]\,d\mu_s\\
&\lesssim I_{2-\eps/2,2}\frac{M^{1-3\eps/2}}{(M+s)^{1-3\eps/2}}\lesssim_\eps I_{2,2}(0)\frac{M^{1-3\eps/2}}{(M+s)^{1-3\eps/2}}.
\end{split}
\end{equation*}
The bound \eqref{NW6} follows.

We can now repeat the argument at the beginning of the proof, starting from the bounds,
\begin{equation*}
\begin{split}
&\int_{t_1}^{T}\frac{1}{M}I_{\eps,0}(s)\,ds\lesssim_\alpha I_{1+\eps,2}(t_1)\quad\text{ and }\quad I_{\eps,0}(s')\lesssim_\alpha I_{\eps,0}(s)\quad\text{ if }s\leq s',
\end{split}
\end{equation*}
which follow from \eqref{NW4} and Theorem \ref{MainTheorem}. Using now \eqref{NW6} it follows easily that
\begin{equation*}
I_{\eps,0}(s)\lesssim_\eps I_{2,4}(0)\frac{M^{2-2\eps}}{(M+s)^{2-2\eps}}\quad\text{ for any }\quad s\in[0,T]\quad\text{ and }\quad \eps\in(0,1/10],
\end{equation*}
which gives the conclusion of Corollary \ref{MainCoro}.
\end{proof}

We turn now to the proof of Lemma \ref{NW0}.

\begin{proof}[Proof of Lemma \ref{NW0}] In view of Theorem \ref{MainTheorem}, with the notation \eqref{NW1.5}, we know that
\begin{equation}\label{NW10}
\begin{split}
I_{\alpha,2}(t_2)+\sum_{k=0}^2&\int_{\mathcal{D}^{c_0}_{[t_1,t_2]}}\frac{r^\alpha}{M^{\alpha}}\Big\{\frac{(r-r^\ast)^2}{r^3}\frac{(\partial_1\phi_k)^2+(\partial_1\psi_k)^2+\psi_k^2(\sin\theta)^{-2}}{r^2}\\
&+\frac{1}{r^3}\big(\phi_k^2+\psi_k^2\big)+\frac{M^{2}}{r^{3}}\big[(\partial_2\phi_k)^2+(\partial_2\psi_k)^2\big]\Big\}\,d\mu\lesssim_\alpha I_{\alpha,2}(t_1),
\end{split}
\end{equation}
for any $t_1\leq t_2\in[0,T]$ and $\alpha\in(0,2)$. It suffices to prove that
\begin{equation}\label{NW11}
\begin{split}
\int_{\mathcal{D}^{c_0}_{[t_1,T]}}\frac{r^\alpha}{M^{\alpha}}\widetilde{\chi}(r)\frac{(\partial_1\phi)^2+(\partial_1\psi)^2+\psi^2(\sin\theta)^{-2}}{r^3}\,d\mu\lesssim_\alpha I_{\alpha,2}(t_1),
\end{split}
\end{equation}
where $\widetilde{\chi}:=\chi_{\geq 9M/4}-\chi_{\geq 4M}$. For this we use elliptic estimate and \eqref{NW10}. 

The equation for $\phi$ and the formula \eqref{waveop} show that
\begin{equation}\label{NW15}
\begin{split}
\g^{11}\Big[\partial_1^2\phi+\frac{\cos\theta}{\sin\theta}\partial_1\phi\Big]+\g^{22}\partial_2^2\phi+\frac{2\g^{11}\D_1B}{A}\partial_1\psi=-F_\phi,
\end{split}
\end{equation}
where
\begin{equation*}
\begin{split}
F_{\phi}&:=\g^{33}\partial_3^2\phi+2\g^{23}\partial_2\partial_3\phi+D^2\partial_2\phi+D^3\partial_3\phi\\
&+2\frac{\D^2B\D_2\psi+\D^3B\D_3\psi}{A}-2\frac{\D^\mu B\D_\mu B}{A^2}\phi+2\frac{\D^\mu B\D_\mu A}{A^2}\psi.
\end{split}
\end{equation*}
If follows from \eqref{NW10} that
\begin{equation}\label{NW16}
\int_{\mathcal{D}^{c_0}_{[t_1,T]}}\frac{M^4}{r^3}|F_{\phi}|^2\,d\mu\lesssim_\alpha I_{\alpha,2}(t_1).
\end{equation}
Using then integration by parts and \eqref{NW15}, we have
\begin{equation*}
\begin{split}
\int_{\mathcal{D}^{c_0}_{[t_1,T]}}&\frac{r^\alpha}{M^{\alpha}}\widetilde{\chi}(r)\frac{(\partial_1\phi)^2}{r^3}\,d\mu\lesssim \int_{[t_1,T]\times(0,\pi)\times(c_0,\infty)}\widetilde{\chi}(r)\frac{(\partial_1\phi)^2}{M^3}r^2(\sin\theta)\,drd\theta dt\\
&\lesssim \Big|\int_{[t_1,T]\times(0,\pi)\times(c_0,\infty)}\widetilde{\chi}(r)\phi\cdot \Big[\partial_1^2\phi+\frac{\cos\theta}{\sin\theta}\partial_1\phi\Big]\frac{r^2}{M^3}(\sin\theta)\,drd\theta dt\Big|\\
&\lesssim \Big|\int_{[t_1,T]\times(0,\pi)\times(c_0,\infty)}\widetilde{\chi}(r)\phi\cdot \Big[\Delta\partial_2^2\phi+\frac{2\D_1B}{A}\partial_1\psi+\frac{F_\phi}{\g^{11}}\Big]\frac{r^2}{M^3}(\sin\theta)\,drd\theta dt\Big|.
\end{split}
\end{equation*}
Using \eqref{NW10}, \eqref{NW16}, and integration by parts it follows that
\begin{equation}\label{NW17}
\begin{split}
\int_{\mathcal{D}^{c_0}_{[t_1,T]}}&\frac{r^\alpha}{M^{\alpha}}\widetilde{\chi}(r)\frac{(\partial_1\phi)^2}{r^3}\,d\mu\lesssim_\alpha I_{\alpha,2}(t_1)+[I_{\alpha,2}(t_1)]^{1/2}\Big(\int_{\mathcal{D}^{c_0}_{[t_1,T]}}\widetilde{\chi}(r)\frac{(\partial_1\psi)^2}{r^3}\,d\mu\Big)^{1/2}.
\end{split}
\end{equation}

Similarly, the equation for $\psi$ and the formula \eqref{waveop} show that
\begin{equation*}
\begin{split}
\g^{11}\Big[\partial_1^2\psi+\frac{\cos\theta}{\sin\theta}\partial_1\psi-\frac{4(\cos\theta)^2}{(\sin\theta)^2}\psi\Big]+\g^{22}\partial_2^2\psi-\frac{2\g^{11}\D_1B}{A}\partial_1\phi=-F_\psi,
\end{split}
\end{equation*}
where $F_{\psi}$ satisfies the same bound \eqref{NW16} as $F_{\phi}$, and the additional term in the left-hand side comes from the fraction $\frac{2\cos\theta}{\sin\theta}$ in $A^{-1}\D_1A$ (see \eqref{ell12}). Integrating by parts as before we have
\begin{equation*}
\begin{split}
\int_{\mathcal{D}^{c_0}_{[t_1,T]}}&\frac{r^\alpha}{M^{\alpha}}\widetilde{\chi}(r)\frac{(\partial_1\psi)^2+\psi^2(\sin\theta)^{-2}}{r^3}\,d\mu\lesssim_\alpha I_{\alpha,2}(t_1)+[I_{\alpha,2}(t_1)]^{1/2}\Big(\int_{\mathcal{D}^{c_0}_{[t_1,T]}}\widetilde{\chi}(r)\frac{(\partial_1\phi)^2}{r^3}\,d\mu\Big)^{1/2}.
\end{split}
\end{equation*}
The desired bound \eqref{NW11} follows using also \eqref{NW17}.
\end{proof}

\appendix
\section{Explicit formulas in Kerr spaces}\label{KerrExplicit}

Recall the Kerr spacetimes $\mathcal{K}(m,a)$, in standard Boyer--Lindquist coordinates,
\begin{equation}\label{rak1}
\g=-\frac{q^2\Delta}{\Sigma^2}(dt)^2+\frac{\Sigma^2(\sin\theta)^2}{q^2}\Big(d\phi-\frac{2aMr}{\Sigma^2}dt\Big)^2 +\frac{q^2}{\Delta}(dr)^2+q^2(d\theta)^2,
\end{equation}
where
\begin{equation}\label{rak2}
\begin{cases}
&\Delta=r^2+a^2-2Mr;\\
&q^2=r^2+a^2(\cos\theta)^2;\\
&\Sigma^2=(r^2+a^2)q^2+2Mra^2(\sin\theta)^2=(r^2+a^2)^2-a^2(\sin\theta)^2\Delta.
\end{cases}
\end{equation}
Observe that
\begin{equation}\label{rak3}
(2Mr-q^2)\Sigma^2=-q^4\Delta+4a^2M^2r^2(\sin\theta)^2.
\end{equation}

Recall the change of variables \eqref{cha1}--\eqref{cha2} and let
\begin{equation*}
p:=\frac{2Mr}{q^2}.
\end{equation*}
Therefore
\begin{equation*}
\frac{\Sigma^2}{q^2}=q^2+(p+1)a^2(\sin\theta)^2,\qquad \Delta=q^2(1-p)+a^2(\sin\theta)^2.
\end{equation*}
Recall that
\begin{equation}\label{rak4}
\partial_1=\partial_\theta=\frac{d}{d\theta},\qquad\partial_2=\partial_r=\frac{d}{dr},\qquad\partial_3=\partial_t=\frac{d}{dt_+}=\T,\qquad\partial_4=\partial_\phi=\frac{d}{d\phi_+}=\Z.
\end{equation}
The nontrivial components of the metric $\g$ become
\begin{equation}\label{exp1.1}
\begin{split}
&\g_{11}=q^2,\quad \g_{33}=p-1,\quad \g_{34}=-a(\sin\theta)^2p,\quad \g_{44}=q^2(\sin\theta)^2+(p+1)a^2(\sin\theta)^4,\\
&\g_{22}=\frac{q^2}{\Delta}(1-\chi^2)+(p+1)\chi^2,\quad\g_{23}=p\chi,\quad \g_{24}=-a(\sin\theta)^2(p+1)\chi,
\end{split}
\end{equation}
and, letting $\mathrm{Det}:=-q^2(\sin\theta)^2$,
\begin{equation}\label{exp2.1}
\begin{split}
&\g^{11}=\frac{1}{\g_{11}}=\frac{1}{q^2},\\
&\g^{22}=\frac{\g_{33}\g_{44}-\g_{34}^2}{\mathrm{Det}}=\frac{\Delta}{q^2},\\
&\g^{23}=\frac{\g_{24}\g_{34}-\g_{23}\g_{44}}{\mathrm{Det}}=p\chi,\\
&\g^{24}=\frac{\g_{23}\g_{34}-\g_{33}\g_{24}}{\mathrm{Det}}=\frac{a\chi}{q^2},\\
&\g^{33}=\frac{\g_{22}\g_{44}-\g_{24}^2}{\mathrm{Det}}=-(p+1)\chi^2-\frac{q^2+(p+1)a^2(\sin\theta)^2}{\Delta}(1-\chi^2),\\
&\g^{34}=\frac{\g_{24}\g_{23}-\g_{22}\g_{34}}{\mathrm{Det}}=\frac{-ap}{\Delta}(1-\chi^2),\\
&\g^{44}=\frac{\g_{22}\g_{33}-\g_{23}^2}{\mathrm{Det}}=\frac{\Delta-a^2(\sin\theta)^2(1-\chi^2)}{q^2\Delta(\sin\theta)^2}.
\end{split}
\end{equation}

The metric $\g$ extends to the larger open set
\begin{equation*}
\widetilde{R}=\{(\theta,r,t_+,\phi_+)\in(-\pi,\pi)\times(0,\infty)\times\mathbb{R}\times\mathbb{S}^1\}.
\end{equation*}
Recall also the sets, see \eqref{surf1}--\eqref{surf3},
\begin{equation}\label{surfs}
\begin{split}
&\mathcal{D}^{c}_I=\{(\theta,r,t_+,\phi_+)\in\widetilde{R}:t_+\in I\text{ and }r>c\},\\
&\Sigma^c_t:=\{(\theta,r,t_+,\phi_+)\in\widetilde{R}:t_+=t\text{ and }r>c\},\\
&\mathcal{N}^{c}_I:=\{(\theta,r,t_+,\phi_+)\in\widetilde{R}:t_+\in I\text{ and }r=c\},
\end{split}
\end{equation}
defined for $c\in(0,\infty)$, $t\in\mathbb{R}$, and intervals $I\subseteq\mathbb{R}$.

Notice that
\begin{equation}\label{deriv}
\begin{split}
&\partial_1(q^2)=-2a^2\sin\theta\cos\theta,\qquad \partial_2(q^2)=2r,\\
&\partial_1p=\frac{4Mra^2\sin\theta\cos\theta}{q^4},\qquad \partial_2p=-\frac{2M(r^2-a^2(\cos\theta)^2)}{q^4}.
\end{split}
\end{equation}

Recall the general formula
\begin{equation*}
\Gamma_{\mu\al\be}=\g(\D_{\partial_\be}\partial_\al,\partial_\mu)=\frac{1}{2}(\partial_\al\g_{\be\mu}+\partial_\be\g_{\al\mu}-\partial_\mu\g_{\al\be}).
\end{equation*}
In the case of $\Z$-invariant functions $f$, i.e. if $\Z(f)=0$, we have the general formula
\begin{equation}\label{waveop}
\begin{split}
\square f&=\g^{\al\be}\partial_\al\partial_\be f-\g^{\al\be}\g^{\mu\nu}\Gamma_{\mu\al\be}\partial_\nu f\\
&=\g^{\al\be}\partial_\al\partial_\be f+\big[\partial_\mu\g^{\mu\nu}+(1/2)\g^{\mu\nu}\partial_\mu\log\big|q^4(\sin\theta)^2\big|\big]\partial_\nu f\\
&=\g^{11}\partial_1^2f+\g^{22}\partial_2^2f+\g^{33}\partial_3^2f+2\g^{23}\partial_2\partial_3f+\big[\partial_1\g^{11}+(1/2)\g^{11}\partial_1\log\big|q^4(\sin\theta)^2\big|\big]\partial_1 f\\
&+\big[\partial_2\g^{22}+(1/2)\g^{22}\partial_2\log\big|q^4(\sin\theta)^2\big|\big]\partial_2 f+\big[\partial_2\g^{23}+(1/2)\g^{23}\partial_2\log\big|q^4(\sin\theta)^2\big|\big]\partial_3 f\\
&=\g^{11}\Big[\partial_1^2f+\frac{\cos\theta}{\sin\theta}\partial_1f\Big]+\g^{22}\partial_2^2f+\g^{33}\partial_3^2f+2\g^{23}\partial_2\partial_3f+D^2\partial_2 f+D^3\partial_3 f,
\end{split}
\end{equation}
where
\begin{equation*}
D^2:=\partial_2\g^{22}+\g^{22}\frac{2r}{q^2}=\frac{2r-2M}{q^2},
\qquad D^3:=\partial_2\g^{23}+\g^{23}\frac{2r}{q^2}=\frac{2M\chi(r)+2Mr\chi'(r)}{q^2}.
\end{equation*}
Also, if $m$ is a $1$-form satisfying $m_4=0$ and $\partial_4 m_\alpha=0$, $\alpha\in\{1,2,3,4\}$, then
\begin{equation}\label{divop}
\begin{split}
&\D^\alpha m_\alpha=\g^{\al\be}\partial_\al m_\be-\g^{\al\be}\g^{\mu\nu}\Gamma_{\mu\al\be}m_\nu\\
&=\g^{11}\Big[\partial_1m_1+\frac{\cos\theta}{\sin\theta}m_1\Big]+\g^{22}\partial_2m_2+\g^{33}\partial_3m_3+\g^{23}(\partial_2m_3+\partial_3m_2)+D^2m_2+D^3m_3.
\end{split}
\end{equation}

\subsubsection{Vector-fields}\label{vectors}
Letting
\begin{equation*}
\pi_{\al\be}=(\mathcal{L}_{\partial_2}\g)_{\al\be}=\Gamma_{\al 2 \be}+\Gamma_{\be 2 \al},
\end{equation*}
we calculate
\begin{equation}\label{kra7}
\begin{split}
&\pi_{\al\be}=\partial_2\g_{\al\be},\\
&\pi^{\al\be}=\g^{\al\mu}\g^{\be\nu}\pi_{\mu\nu}=\g^{\al\mu}\g^{\be\nu}\partial_2\g_{\mu\nu}=-\g^{\be\nu}\g_{\mu\nu}\partial_2\g^{\al\mu}=-\partial_2\g^{\al\be},\\
&\pi_{\al\be}\g^{\al\be}=\partial_2\log|q^4(\sin\theta)^2|=4r/q^2.
\end{split}
\end{equation}

Therefore, for any vector field
\begin{equation}\label{kra7.1}
X=f(r)\partial_2+g(r)\partial_3,
\end{equation}
we calculate
\begin{equation}\label{kra7.2}
\begin{split}
{}^{(X)}\pi^{\mu\nu}&:=\D^\mu X^\nu+\D^\nu X^\mu\\
&=f\pi^{\mu\nu}+(\D^\mu f\delta_2^\nu+\D^\nu f\delta_2^\mu)+(\D^\mu g\delta_3^\nu+\D^\nu g\delta_3^\mu)\\
&=f\pi^{\mu\nu}+f'(r)(\g^{\mu 2} \delta_2^\nu+\g^{\nu 2} \delta_2^\mu)+g'(r)(\g^{\mu 2} \delta_3^\nu+\g^{\nu 2} \delta_3^\mu).
\end{split}
\end{equation}

For any $1$-form $Y$ with $Y_4=0$ let
\begin{equation}\label{kra7.3}
{}^{(Y)}Q_{\mu\nu}=Y_\mu Y_\nu-(1/2)\g_{\mu\nu}(Y_\rho Y^\rho).
\end{equation}
We calculate the contraction 
\begin{equation*}
\begin{split}
&{}^{(Y)}Q_{\mu\nu}{}^{(X)}\pi^{\mu\nu}={}^{(X)}\pi^{\mu\nu}Y_\mu Y_\nu-(1/2)\g_{\mu\nu}(Y_\rho Y^\rho){}^{(X)}\pi^{\mu\nu}\\
&=f(r)\pi^{\mu\nu}Y_\mu Y_\nu+2f'(r)Y^2Y_2+2g'(r)Y^2Y_3-(Y_\rho Y^\rho)[2rf(r)/q^2+f'(r)]\\
&=(Y_1)^2\Big[f(r)\pi^{11}-\frac{2rf(r)\g^{11}}{q^2}-f'(r)\g^{11}\Big]+
(Y_2)^2\Big[f(r)\pi^{22}+f'(r)\g^{22}-\frac{2rf(r)\g^{22}}{q^2}\Big]\\
&+(Y_3)^2\Big[f(r)\pi^{33}+2g'(r)\g^{23}-f'(r)\g^{33}-\frac{2rf(r)\g^{33}}{q^2}\Big]\\
&+2Y_2Y_3\Big[f(r)\pi^{23}+g'(r)\g^{22}-\frac{2rf(r)\g^{23}}{q^2}\Big].
\end{split}
\end{equation*}
Using also the formulas \eqref{kra7} and \eqref{exp2.1} this simplifies to
\begin{equation}\label{kra7.4}
\begin{split}
{}^{(Y)}Q_{\mu\nu}{}^{(X)}\pi^{\mu\nu}&=(Y_1)^2\frac{-f'(r)}{q^2}+
(Y_2)^2\frac{-f(r)(2r-2M)+f'(r)\Delta}{q^2}\\
&+(Y_3)^2\Big[-f(r)\partial_2\g^{33}+2g'(r)\g^{23}-f'(r)\g^{33}-\frac{2rf(r)\g^{33}}{q^2}\Big]\\
&+2Y_2Y_3\frac{-2Mrf(r)\chi'(r)-2Mf(r)\chi(r)+g'(r)\Delta}{q^2}.
\end{split}
\end{equation}

Recall the vector-fields $n=-\g^{\mu\nu}\partial_\nu u_+\partial_\mu=-\g^{3\mu}\partial_\mu$ and $k=\g^{\mu\nu}\partial_\nu r\partial_\mu=\g^{2\mu}\partial_\mu$, defined in $\widetilde{R}$, which are normal to the hypersurfaces $\Sigma^c_t$ and $\mathcal{N}^{c}_I$ respectively. We calculate
\begin{equation*}
{}^{(Y)}Q(n,\partial_2)=-\g^{3\mu}Y_\mu Y_2=-\g^{32}(Y_2)^2-\g^{33}Y_2Y_3,
\end{equation*}
\begin{equation*}
{}^{(Y)}Q(n,\partial_3)=-\g^{3\mu}Y_\mu Y_3+(1/2)(Y_\rho Y^\rho)=(1/2)[\g^{11}(Y_1)^2+\g^{22}(Y_2)^2-\g^{33}(Y_3)^2],
\end{equation*}
\begin{equation*}
{}^{(Y)}Q(k,\partial_2)=\g^{2\mu}Y_\mu Y_2-(1/2)(Y_\rho Y^\rho)=(1/2)[-\g^{11}(Y_1)^2+\g^{22}(Y_2)^2-\g^{33}(Y_3)^2],
\end{equation*}
and
\begin{equation*}
{}^{(Y)}Q(k,\partial_3)=\g^{2\mu}Y_\mu Y_3=\g^{23}(Y_3)^2+\g^{22}Y_2Y_3.
\end{equation*}
Therefore, if $X=f(r)\partial_2+g(r)\partial_3$ as in \eqref{kra7.1} then
\begin{equation}\label{Qy1}
\begin{split}
 2{}^{(Y)}Q(n,X)&=(Y_1)^2[g(r)\g^{11}]+(Y_2)^2[g(r)\g^{22}-2f(r)\g^{23}]\\
&+(Y_3)^2[-g(r)\g^{33}]+2Y_2Y_3[-f(r)\g^{33}]
\end{split}
\end{equation}
and 
\begin{equation}\label{Qy2}
\begin{split}
 2{}^{(Y)}Q(k,X)&=(Y_1)^2[-f(r)\g^{11}]+(Y_2)^2[f(r)\g^{22}]\\
&+(Y_3)^2[-f(r)\g^{33}+2g(r)\g^{23}]+2Y_2Y_3[g(r)\g^{22}].
\end{split}
\end{equation}

\subsubsection{Hardy inequalities}\label{Hardy} In this subsection we prove the following lemma:

\begin{lemma}\label{GeneralLemma}

(i) If $c\geq c_0$ and $f\in H^1_{\mathrm{loc}}((c,\infty))$ satisfies $\lim\limits_{D\to\infty}\int_D^{2D}|f(r)|^2\,dr=0$ then
\begin{equation}\label{eli1}
\int_{c}^\infty|f/r|^2\cdot r^2\,dr\lesssim \int_{c}^\infty|f'|^2\cdot r^2\,dr.
\end{equation}

(ii) If $g\in H^1_{\mathrm{loc}}((0,\pi))$ and $p\in[0,10]$ then
\begin{equation}\label{eli2}
\int_0^\pi|g|^2(\sin\theta)^p\,d\theta\lesssim \int_0^\pi|g'|^2(\sin\theta)^{p+2}\,d\theta+\int_0^\pi|g|^2(\sin\theta)^{p+2}\,d\theta.
\end{equation}

(iii) If $f\in H^1_{\mathrm{loc}}((0,\pi))$ then
\begin{equation}\label{eli3}
\int_0^\pi|f'|^2\sin\theta\,d\theta+\int_0^\pi|f|^2(\sin\theta)^{-1}\,d\theta\approx \int_0^\pi\Big|f'-\frac{2\cos\theta}{\sin\theta}f\Big|^2\sin\theta\,d\theta+\int_0^\pi|f|^2\sin\theta\,d\theta.
\end{equation}

(iv) If $g\in L^2_{\mathrm{loc}}((0,\pi))$ then
\begin{equation}\label{eli2.5}
\int_0^\pi|g|^2(\sin\theta)^{-1}\,d\theta\lesssim \int_0^\pi\Big|g'+\frac{\cos\theta}{\sin\theta}g\Big|^2\sin\theta\,d\theta+\int_0^\pi|g|^2\sin\theta\,d\theta.
\end{equation}

(v) If $f\in H^1_{\mathrm{loc}}((0,\pi))$ then
\begin{equation}\label{eli2.8}
\begin{split}
\int_0^{\pi} &|f''|^2\sin\theta+|f'|^2(\sin\theta)^{-1}+|f|^2(\sin\theta)^{-3}\,d\theta\\
&\lesssim \int_0^\pi\Big|f''+\frac{\cos\theta}{\sin\theta}f'-\frac{4(\cos\theta)^2}{(\sin\theta)^2}f\Big|^2\sin\theta\,d\theta+\int_0^\pi |f'|^2\sin\theta+|f|^2(\sin\theta)^{-1}\,d\theta.
\end{split}
\end{equation}
\end{lemma}

\begin{proof} The inequalities in this lemma are standard Hardy-type inequalities, and we provide the proofs mostly for sake of 
completeness. 

For (i) we may assume that $f$ is real-valued and
\begin{equation*}
\int_{c}^\infty|f'(r)|^2r^2\,dr=1.
\end{equation*}
Given $\delta>0$ small and $D\gg 1$ we fix a smooth function $K=K_{\delta,D}:\mathbb{R}\to\mathbb{R}$ supported in the interval $[c+\delta/2,2D]$ with the properties
\begin{equation}\label{eli4}
\begin{split}
&K'(r)=1\qquad\text{ if }r\in[c+\delta,D],\\
&|K'(r)|\lesssim 1\qquad\text{ if }r\in[D,2D],\\
&K'\text{ is increasing on the interval }[c+\delta/2,c+\delta].
\end{split}
\end{equation}
By taking $D$ sufficiently large, we may assume that
\begin{equation*}
\int_D^{2D}|f(r)|^2\,dr\leq 1.
\end{equation*}
Notice that $K(r)\lesssim rK'(r)^{1/2}$ for any $r\in[c,D]$, which follows easily from \eqref{eli4}. Then we estimate, using integration by parts
\begin{equation*}
\begin{split}
\Big|\int_{c}^Df(r)^2K'(r)\,dr\Big|&\lesssim \Big|\int_{\mathbb{R}}f(r)^2K'(r)\,dr\Big|+\int_D^{2D}|f(r)|^2\,dr\\
&\lesssim\int_{\mathbb{R}}|f(r)||f'(r)||K(r)|\,dr+1\\
&\lesssim\int_{c}^D|f(r)||f'(r)||K(r)|\,dr+1\\
&\lesssim \Big|\int_{c}^D|f(r)|^2K'(r)\,dr\Big|^{1/2}\Big|\int_{c}^D|f'(r)|r^2\,dr\Big|^{1/2}+1.
\end{split}
\end{equation*}
Therefore
\begin{equation*}
\Big|\int_{c}^Df(r)^2K'(r)\,dr\Big|\lesssim 1,
\end{equation*}
and the desired inequality follows by letting $\delta\to 0$ and $D\to\infty$.

To prove (ii) we may assume that $g$ is real-valued and 
\begin{equation*}
\int_0^\pi|g'(\theta)|^2(\sin\theta)^{p+2}\,d\theta+\int_0^\pi|g(\theta)|^2(\sin\theta)^{p+2}\,d\theta=1.
\end{equation*}
As before, given $\delta>0$ small we fix a smooth function $K=K_{\delta}:\mathbb{R}\to\mathbb{R}$ supported in the interval $[\delta/2,1/2]$ with the properties
\begin{equation}\label{eli5}
\begin{split}
&K'(\theta)=(\sin\theta)^p\qquad\text{ if }\theta\in[\delta,1/4],\\
&|K'(\theta)|\lesssim 1\qquad\text{ if }\theta\in[1/4,1/2],\\
&K'\text{ is increasing on the interval }[\delta/2,\delta].
\end{split}
\end{equation}
As before, we notice that these assumptions imply that $K(\theta)\lesssim (\sin\theta)^{(p+2)/2}K'(\theta)^{1/2}$ for any $\theta\in[0,1/4]$. Then we estimate, using integration by parts,
\begin{equation*}
\begin{split}
\Big|\int_{0}^{1/4}g(\theta)^2K'(\theta)\,d\theta\Big|&\lesssim \Big|\int_{0}^{1/2}g(\theta)^2K'(\theta)\,d\theta\Big|+1\\
&\lesssim\int_0^{1/2}|g(\theta)||g'(\theta)||K(\theta)|\,d\theta+1\\
&\lesssim\int_0^{1/4}|g(\theta)||g'(\theta)||K(\theta)|\,d\theta+1\\
&\lesssim \Big|\int_{0}^{1/4}|g(\theta)|^2K'(\theta)\,d\theta\Big|^{1/2}\Big|\int_{0}^{1/4}|g'(\theta)|^2(\sin\theta)^{p+2}\,d\theta\Big|^{1/2}+1.
\end{split}
\end{equation*}
Therefore
\begin{equation*}
\Big|\int_{0}^{1/4}g(\theta)^2K'(\theta)\,d\theta\Big|\lesssim 1.
\end{equation*}
Letting $\delta\to 0$ it follows that
\begin{equation*}
\int_0^{1/4}|g(\theta)|^2(\sin\theta)^p\,d\theta\lesssim 1.
\end{equation*}
The change of variables $\theta\to\pi-\theta$ now shows that
\begin{equation*}
\int_{\pi-1/4}^{\pi}|g(\theta)|^2(\sin\theta)^p\,d\theta\lesssim 1,
\end{equation*}
and the desired estimate follows.

To prove (iii), we notice first that the right-hand side of \eqref{eli3} is clearly dominated by the left-hand side. To prove the reverse inequality, let $f(\theta)=(\sin\theta)^2 g(\theta)$ and notice that
\begin{equation*}
f'(\theta)-\frac{2\cos\theta}{\sin\theta}f(\theta)=(\sin\theta)^2g'(\theta).
\end{equation*}
The desired bound follows from the inequality 
\begin{equation*}
\int_0^\pi|g(\theta)|^2(\sin\theta)^3\,d\theta\lesssim \int_0^\pi|g'(\theta)|^2(\sin\theta)^5\,d\theta+\int_0^\pi|g(\theta)|^2(\sin\theta)^5\,d\theta,
\end{equation*}
which is a consequence of \eqref{eli2}.

To prove (iv), we may assume that $g\in H^1_{\mathrm{loc}}((0,\pi))$ is real-valued and let $g(\theta)=h(\theta)/\sin\theta$. Then
\begin{equation*}
g'(\theta)+\frac{\cos\theta}{\sin\theta}g(\theta)=\frac{h'(\theta)}{\sin\theta}.
\end{equation*}
The inequality to prove becomes
\begin{equation}\label{eli2.6}
\int_0^\pi\frac{h(\theta)^2}{(\sin\theta)^3}\,d\theta\lesssim \int_0^\pi\frac{h'(\theta)^2}{\sin\theta}\,d\theta+\int_0^\pi\frac{h(\theta)^2}{\sin\theta}\,d\theta.
\end{equation}
This is nontrivial only if $h'\in L^2((0,\pi))$, which shows that $h'\in L^1((0,\pi))$. Therefore, in proving \eqref{eli2.6} we may assume that $h$ extends to a continuous function on the interval $[0,\pi]$ and $h(0)=h(\pi)=0$ (otherwise the right-hand side of \eqref{eli2.6} is equal to $\infty$). In particular, for any $\theta\in[0,\pi/2]$,
\begin{equation}\label{eli2.7}
h(\theta)=\int_{0}^\theta h'(\mu)\,d\mu.
\end{equation}
For $k\leq 0$ let $c_k:=2^{-k/2}\Big[\int_{[2^{k-1},2^k]}|h'(\mu)|^2\,d\mu\Big]^{1/2}$. The formula \eqref{eli2.7} above shows that
\begin{equation*}
|h(\theta)|\lesssim \sum_{k'\leq k}2^{k'}c_{k'}\qquad \text{ if }k\leq 0\text{ and }\theta\in[2^{k-1},2^{k}].
\end{equation*}
Therefore 
\begin{equation*}
\int_0^1\frac{h(\theta)^2}{(\sin\theta)^3}\,d\theta\lesssim \sum_{k\leq 0}\Big(\sum_{k'\leq k}2^{k'-k}c_{k'}\Big)^2\lesssim \sum_{k\leq 0}c_k^2\lesssim \int_0^1\frac{h'(\theta)^2}{\sin\theta}\,d\theta.
\end{equation*}

The change of variables $\theta\to\pi-\theta$ shows that 
\begin{equation*}
\int_{\pi-1}^\pi\frac{h(\theta)^2}{(\sin\theta)^3}\,d\theta\lesssim\int_{\pi-1}^\pi\frac{h'(\theta)^2}{\sin\theta}\,d\theta+\int_0^\pi\frac{h(\theta)^2}{\sin\theta}\,d\theta,
\end{equation*}
and the desired bound \eqref{eli2.6} follows.

To prove (v), we may assume that $f\in H^2_{\mathrm{loc}}((0,\pi))$ is real-valued and let $f(\theta)=g(\theta)(\sin\theta)^2$. Then
\begin{equation*}
f''(\theta)+\frac{\cos\theta}{\sin\theta}f'(\theta)-\frac{4(\cos\theta)^2}{(\sin\theta)^2}f(\theta)=(\sin\theta)^2g''(\theta)+3\sin\theta\cos\theta g'(\theta)-2(\sin\theta)^2g(\theta).
\end{equation*}
The inequality \eqref{eli2.8} becomes
\begin{equation*}
\begin{split}
\int_0^{\pi} &|g''|^2(\sin\theta)^5+|g'|^2(\sin\theta)^3+|g|^2\sin\theta\,d\theta\\
&\lesssim \int_0^\pi|(\sin\theta)^2g''+3\sin\theta\cos\theta g'|^2\sin\theta\,d\theta+\int_0^\pi |g'|^2(\sin\theta)^5+|g|^2(\sin\theta)^3\,d\theta.
\end{split}
\end{equation*}
In view of the inequality \eqref{eli2} with $p=1$ it suffices to prove that
\begin{equation*}
\int_0^{\pi} |g'|^2(\sin\theta)^3\,d\theta\lesssim \int_0^\pi|(\sin\theta)^2 g''+3\sin\theta\cos\theta g'|^2\sin\theta\,d\theta+\int_0^\pi |g'|^2(\sin\theta)^5\,d\theta.
\end{equation*}
Letting $h(\theta)=(\sin\theta)^3g'(\theta)$ this is equivalent to
\begin{equation*}
\int_0^\pi\frac{h(\theta)^2}{(\sin\theta)^3}\,d\theta\lesssim \int_0^\pi\frac{h'(\theta)^2}{\sin\theta}\,d\theta+\int_0^\pi\frac{h(\theta)^2}{\sin\theta}\,d\theta,
\end{equation*}
which was proved earlier, see \eqref{eli2.6}.
\end{proof}

\subsubsection{The main function spaces}\label{SpacesProp} We summarize now some of the main properties of the spaces $H^m(\Sigma^c_t)$ and $\widetilde{H}^m(\Sigma^c_t)$:

\begin{lemma}\label{App1.5}

Assume $t\in\mathbb{R}$ and $c\geq c_0$.

(i) If $f\in H^1(\Sigma_t^c)$ satisfies $\Z(f)=0$ then
\begin{equation}\label{ell63}
\|f\|_{\widetilde{H}^1(\Sigma_t^c)}\approx \|f\|_{H^1(\Sigma_t^c)}+\|(r\sin\theta)^{-1}f\|_{L^2(\Sigma_t^c)}.
\end{equation}

(ii) If $f\in H^2(\Sigma_t^c)$, satisfies $\Z(f)=0$ then
\begin{equation}\label{ell62}
\|f\|_{\widetilde{H}^2(\Sigma_t^c)}\approx \|f\|_{H^2(\Sigma_t^c)}+\|(r\sin\theta)^{-1}f\|_{H^{1}(\Sigma_t^c)}+\|(r\sin\theta)^{-2}f\|_{L^2(\Sigma_t^c)}.
\end{equation}
\end{lemma}

\begin{proof}[Proof of Lemma \ref{App1.5}] The bound \eqref{ell63} follows easily from the definitions and \eqref{eli3}. We prove now part (ii) and the bounds \eqref{ell62} for $m=2$. In view of the definition, and using also \eqref{ell63},
\begin{equation}\label{ell64}
\begin{split}
\|f\|_{\widetilde{H}^2(\Sigma_t^c)}&\approx \|f\|_{H^2(\Sigma_t^c)}+\|\widetilde{\partial}_2f\|_{\widetilde{H}^{1}(\Sigma_t^c)}+\|(\widetilde{\partial}_1/r)^2f\|_{L^2(\Sigma_t^c)}+\|(\widetilde{\partial}_1/r)f\|_{L^2(\Sigma_t^c)}\\
&\approx \|f\|_{H^2(\Sigma_t^c)}+\|(r\sin\theta)^{-1}\partial_2f\|_{L^2(\Sigma_t^c)}+\|(\widetilde{\partial}_1/r)^2f\|_{L^2(\Sigma_t^c)}+\|(\widetilde{\partial}_1/r)f\|_{L^2(\Sigma_t^c)}.
\end{split}
\end{equation}
Using \eqref{eli1}, we have
\begin{equation*}
\|(r\sin\theta)^{-1}\partial_2f\|_{L^2(\Sigma_t^c)}\lesssim \|\partial_2[(r\sin\theta)^{-1}f]\|_{L^2(\Sigma_t^c)}\lesssim \|(r\sin\theta)^{-1}f\|_{H^1(\Sigma_t^c)}.
\end{equation*}
Moreover, using the definition,
\begin{equation*}
\begin{split}
&\|(\widetilde{\partial}_1/r)^2f\|_{L^2(\Sigma_t^c)}+\|(\widetilde{\partial}_1/r)f\|_{L^2(\Sigma_t^c)}\\
&\lesssim \|(\partial_1/r)^2f\|_{L^2(\Sigma_t^c)}+\|[1+(r\sin\theta)^{-1}](\partial_1/r)f\|_{L^2(\Sigma_t^c)}+\|[1+(r\sin\theta)^{-2}]f\|_{L^2(\Sigma_t^c)}\\
&\lesssim \|f\|_{H^2(\Sigma_t^c)}+\|(r\sin\theta)^{-1}f\|_{H^1(\Sigma_t^c)}+\|(r\sin\theta)^{-2}f\|_{L^2(\Sigma_t^c)}.
\end{split}
\end{equation*}
Using also \eqref{ell64}, it follows that
\begin{equation*}
\|f\|_{\widetilde{H}^2(\Sigma_t^c)}\lesssim \|f\|_{H^2(\Sigma_t^c)}+\|(r\sin\theta)^{-1}f\|_{H^1(\Sigma_t^c)}+\|(r\sin\theta)^{-2}f\|_{L^2(\Sigma_t^c)},
\end{equation*}
as desired.

For the reverse inequality, using \eqref{ell64}, it remains to prove that
\begin{equation}\label{ell65}
\begin{split}
&\|(r\sin\theta)^{-2}f\|_{L^2(\Sigma_t^c)}+\|(r\sin\theta)^{-1}(\partial_1/r)f\|_{L^2(\Sigma_t^c)}\\
&\lesssim \|f\|_{H^2(\Sigma_t^c)}+\|(r\sin\theta)^{-1}\partial_2f\|_{L^2(\Sigma_t^c)}+\|(\widetilde{\partial}_1/r)^2f\|_{L^2(\Sigma_t^c)}+\|(\widetilde{\partial}_1/r)f\|_{L^2(\Sigma_t^c)}.
\end{split}
\end{equation}
Using \eqref{eli3},
\begin{equation*}
\|(r\sin\theta)^{-1}(\widetilde{\partial}_1/r)f\|_{L^2(\Sigma_t^c)}\lesssim\|(\widetilde{\partial}_1/r)^2f\|_{L^2(\Sigma_t^c)}+\|(\widetilde{\partial}_1/r)f\|_{L^2(\Sigma_t^c)}.
\end{equation*}
Also, using \eqref{eli2.8} and then \eqref{ell63},
\begin{equation*}
\begin{split}
&\|(r\sin\theta)^{-2}f\|_{L^2(\Sigma_t^c)}\\
&\lesssim\Big\|r^{-2}\Big[\partial_1^2+\frac{\cos\theta}{\sin\theta}\partial_1-\frac{4(\cos\theta)^2}{(\sin\theta)^2}\Big]f\Big\|_{L^2(\Sigma_t^c)}+\|(\partial_1/r)f\|_{L^2(\Sigma_t^c)}+\|(r\sin\theta)^{-1}f\|_{L^2(\Sigma_t^c)}\\
&\lesssim\|(\widetilde{\partial}_1/r)^2f\|_{L^2(\Sigma_t^c)}+\|(r\sin\theta)^{-1}(\widetilde{\partial}_1/r)f\|_{L^2(\Sigma_t^c)}+\|(\widetilde{\partial}_1/r)f\|_{L^2(\Sigma_t^c)}+\|f\|_{H^1(\Sigma_t^c)}.
\end{split}
\end{equation*}
The desired bound \eqref{ell65} follows from these two estimates.
\end{proof}

\end{document}